\documentclass{amsart}
\usepackage{a4wide}
\usepackage{color,eucal,enumerate,mathrsfs}
\usepackage[normalem]{ulem}
\usepackage{mdwlist}
\usepackage{amsmath}
\usepackage{amssymb,epsfig,bbm}
\usepackage{amsthm}
\numberwithin{equation}{section}
\usepackage{hyperref}
\usepackage{hyphenat}
\usepackage{mathtools}
\usepackage{bigints}
\usepackage{esint}
\usepackage[latin1]{inputenc}
\usepackage[T1]{fontenc}
\usepackage{todonotes}

\newcommand{\N}{\mathbb{N}}
\newcommand{\Q}{\mathbb{Q}}
\newcommand{\R}{\mathbb{R}}

\newcommand{\mm}{{\mbox{\boldmath$m$}}}

\newcommand{\ppi}{{\mbox{\boldmath$\pi$}}}

\newcommand{\sppi}{{\mbox{\scriptsize\boldmath$\pi$}}}

\newcommand{\sfd}{{\sf d}}

\newcommand{\Id}{{\rm Id}}                          
\newcommand{\Kliminf}{K\kern-3pt-\kern-2pt\mathop{\rm lim\,inf}\limits}  
\newcommand{\supp}{\mathop{\rm supp}\nolimits}   
\newcommand{\Lip}{\mathop{\rm Lip}\nolimits}          
\renewcommand{\d}{{\mathrm d}}

\newcommand{\restr}[1]{\lower3pt\hbox{$|_{#1}$}}
\newcommand{\la}{{\langle}}                  
\newcommand{\ra}{{\rangle}}
\newcommand{\eps}{\varepsilon}  
\newcommand{\nchi}{{\raise.3ex\hbox{$\chi$}}}
\newcommand{\weakto}{\rightharpoonup}

\newcommand{\limi}{\varliminf}
\newcommand{\lims}{\varlimsup}
\DeclareMathOperator{\aplip}{{\rm ap\,-lip}}
\DeclareMathOperator*{\aplimi}{{\rm ap\,-}\limi}
\DeclareMathOperator*{\aplims}{{\rm ap\,-}\lims}

\newcommand{\fr}{\penalty-20\null\hfill$\blacksquare$}                      

\newcommand{\e}{{\rm{e}}}                          
 \newcommand{\X}{{\rm X}}
 \newcommand{\Y}{{\rm Y}}

\newcommand{\LIP}{{\rm LIP}}
\newcommand{\lip}{{\rm lip}}

\renewcommand{\div}{{\rm div}}
\newcommand{\testv}{{\rm TestV}}

\renewcommand{\mm}{\mathfrak m}                                

\newcommand{\test}[1]{{\rm Test}(#1)}

\newtheorem{theorem}{Theorem}[section]

\newtheorem{corollary}[theorem]{Corollary}
\newtheorem{lemma}[theorem]{Lemma}
\newtheorem{proposition}[theorem]{Proposition}

\newtheorem{definition}[theorem]{Definition}

\newcounter{Counter}

\newtheorem{remark}[theorem]{Remark}

\newcommand{\beq}{\begin{equation}}
\newcommand{\eeq}{\end{equation}}

\renewcommand{\H}{\mathscr H}

\title{Parallel transport on non-collapsed ${\sf RCD}(K,N)$ spaces}

\author[Emanuele Caputo]{Emanuele Caputo}\address[Emanuele Caputo]{SISSA, Via Bonomea 265, 34136 Trieste, Italy.}\email{ecaputo@sissa.it}
\author[Nicola Gigli]{Nicola Gigli}\address[Nicola Gigli]{SISSA, Via Bonomea 265, 34136 Trieste, Italy.}\email{ngigli@sissa.it}
\author[Enrico Pasqualetto]{Enrico Pasqualetto}\address[Enrico Pasqualetto]{Scuola Normale Superiore, Piazza dei Cavalieri, 7, 56126 Pisa, Italy.}\email{enrico.pasqualetto@sns.it}

\begin{document}


\begin{abstract} We provide a general theory for parallel transport on non-collapsed ${\sf RCD}$ spaces obtaining both existence and uniqueness results. Our theory covers the case of geodesics and, more generally, of curves obtained via the flow of sufficiently regular time dependent vector fields: the price that we pay for this generality is that we cannot study parallel transport along a single such curve, but only along almost all of these (in a sense related to the notions of Sobolev vector calculus and Regular Lagrangian Flow in the nonsmooth setting).

The class of ${\sf ncRCD}$ spaces contains finite dimensional Alexandrov spaces with curvature bounded from below, thus our construction provides a way of speaking about parallel transport in this latter setting alternative to the one proposed by Petrunin in \cite{Petrunin98}. The precise relation between the two approaches is yet to be understood.

\end{abstract}
\maketitle
\tableofcontents
\section{Introduction}
A basic and fundamental notion in Riemannian geometry is that of parallel transport. In the classical smooth framework existence and uniqueness of such object are typically obtained working in coordinates and using standard results in ODE's with smooth coefficients. When working in non-smooth environment, one of the challenges one faces is that of understanding up to what extent standard objects existing in the smooth category have a natural counterpart: in this paper we are concerned with parallel transport on non-collapsed ${\sf RCD}$ spaces. 

To put things into context,  let us mention that in the setting of finite dimensional spaces with sectional curvature bounded from below, an existence theory for parallel transport has been provided in \cite{Petrunin98}  - we shall comment on the relation between such notion and ours toward the end of the introduction. Another non-smooth setting where  parallel transport has been investigated is  Otto's `manifold' made of probability measures equipped with the Wasserstein $W_2$ distance, see \cite{Giglimemo} and \cite{Lott17}.

On ${\sf RCD}$ spaces (\cite{AmbrosioGigliSavare11-2}, \cite{Gigli12}) it has been shown in \cite{Gigli14} that there is a well defined notion of covariant derivative: this should be intended in an appropriate Sobolev sense, i.e.\ it makes sense to define the concept of vector field with covariant derivative in $L^2$. `Covariant derivative' and `parallel transport' are two extremely close concepts, and thus it is natural to expect that the latter also exists in the same generality,  provided one pays due attention to the way to formulate the concept in relation with distributional notions of covariant differentiation. An attempt in this direction has been made in \cite{GP20}: there the problem of parallel transport is formulated not along a single given curve - as it is customary in the smooth setting - but rather along a `test plan'. These are probability measures on the space of curves (that should be thought of as measures on `well distributed smooth curves') that are tightly linked to Sobolev calculus on metric measure spaces (\cite{AmbrosioGigliSavare11}, \cite{Gigli14}). Perhaps surprisingly, the well-posedness of covariant differentiation on ${\sf RCD}$ spaces seems not sufficient to derive well-posedness of parallel transport and the main result in \cite{GP20} can be roughly summarized as follows: 
\begin{itemize}
\item[i)] it makes sense to define what is `covariant differentiation along a test plan $\ppi$',
\item[ii)] to such differentiation operator, one can associate the `W' and `H' Sobolev spaces $\mathscr W^{1,2}(\ppi)$ and $\mathscr H^{1,2}(\ppi)$, defined respectively in duality or as closure of  `smooth' vector fields,
\item[iii)] in the Sobolev space $\mathscr H^{1,2}(\ppi)$ uniqueness of parallel transport can be proved (but not existence).
\end{itemize}
On top of these results, we shall see in the appendix of this manuscript that  in $\mathscr W^{1,2}(\ppi)$, by a general argument based on approximation by viscosity, existence of parallel transport can be established (but not uniqueness).

\bigskip

It is therefore natural to look for some sort of intermediate space between $\mathscr H^{1,2}(\ppi)$ and $\mathscr W^{1,2}(\ppi)$ where both existence and uniqueness of parallel transport can be obtained. Alternatively, one might try to impose additional regularity on either the space or the test plan involved to get a better theory. In this manuscript we pursue both directions: we shall in fact work with only certain types of test plans on a relevant subclass of ${\sf RCD}$ spaces, obtain 
\[
\text{existence and uniqueness of parallel transport}
\] 
in this setting, and finally (in Section \ref{se:comparison}) compare our construction to those in  \cite{GP20} and prove that, in a very natural sense, the relevant space involved sits between $\mathscr H^{1,2}(\ppi)$ and $\mathscr W^{1,2}(\ppi)$. The actual statement of our main result requires a bit of terminology, so we postpone it to Theorem \ref{thm:exuni}.

\bigskip

Let us give more details about our setting. Rather than investigating `generic collections of smooth curves', as test plans can be thought of, we focus on flows of `smooth' vector fields. Specifically, in the ${\sf RCD}$ setting the concept of Regular Lagrangian Flow (\cite{Ambrosio-Trevisan14}, see also \cite{Ambrosio04}, \cite{DiPerna-Lions89}) provides a reasonable counterpart for the classical Cauchy--Lipschitz theory and it gives meaning to the flow $(F_t^s)$ of a family $(w_t)$ of Sobolev vector fields with uniformly bounded divergence (see Theorem \ref{thm:AT} for the precise statement). A non-trivial regularity result - established in \cite{BS18} (see also \cite{BDS21}) - concerning such flows in the finite-dimensional case is that they are uniformly Lusin--Lipschitz, i.e.\ there is a Borel partition $(E_i)$ of a.a.\ the space such that the restriction of $F_0^t$ to $E_i$ is uniformly Lipschitz in $t$.  In the smooth setting a vector field $t\mapsto V_t\in T_{F_0^t(x)}M$ is a parallel transport along $t\mapsto  F_0^t(x)$ if and only if for $v_t:=\d F_t^0(V_t)\in T_xM$ we have 
\begin{equation}
\label{eq:dotvt}
\dot v_t=-\d F_t^0(\nabla_{\d F_0^t(v_t)}w_t),\qquad\text{ for a.e.\ }t.
\end{equation} In other words, the quantity  $\d F_t^0(\nabla_{\d F_0^t(v_t)}w_t)$ takes into account the time derivative of the distortion given by the differential of the flow. The advantage of working with the vectors $v_t$ is that they  all belong to the same tangent space, so that it is clear what the derivative $\dot v_t$ is and this advantage is particularly felt in our non-smooth setting, where it is unclear how to define the covariant derivative of a vector field of the form $t\mapsto  V_t\in T_{F_0^t(x)}M$.

We thus choose this approach and are able to make it work in the class of non-collapsed ${\sf RCD}$ spaces (introduced in \cite{GDP17} as a counterpart to non-collapsed Ricci-limit spaces introduced in \cite{Cheeger-Colding97I}, \cite{Cheeger-Colding97II}, \cite{Cheeger-Colding97III}). The reason for this further restriction is in the better regularity estimates available in this setting, that in particular allow to say that the norm of the differential of $F_t^s$ is of order $|s-t|$ (see Proposition \ref{prop:eliadanieleregularity} for the precise statement - in general ${\sf RCD}(K,N)$ spaces we only know that such differential is uniformly bounded in $t,s$). This key regularity property, obtained in the recent \cite{BDS21}, is needed if one wants to differentiate in time the differential of the flow and recover the covariant derivative of the underlying vector field $w_t$ as in  \eqref{eq:dotvt}.

\bigskip

To actually make this plan work a few non-trivial technical obstacles have to be dealt with, in particular in relation with the need of interchanging differentiation in time and differentiation of functions/vector fields. Usually, this sort of issues are managed through closure properties of the differentiation operator considered in conjunction with Hille's theorem about swapping (Bochner) integration and a closed operator. We do not deviate from this general plan, but have to rework some crucial steps. 

The first is about the definition of differentiation and the closure of such operator. We shall indeed need to work with functions of the kind $f\circ F_t^s$ for $f\in W^{1,2}(\X)$ and $(F_t^s)$ Regular Lagrangian Flow as discussed above. It is well known that one should not expect Sobolev regularity for this kind of functions, but at least a notion of differentiation can naturally be given using the Lusin--Lipschitz property of $F_t^s$, that in turn ensures that also $f\circ F_t^s$ is Lusin--Lipschitz. The problem with the concept of differential defined by locality for Lusin--Lipschitz maps is that is by no means a closed operator: what we need to do to get closure is to properly restrict the domain of the operator. The idea we use takes inspiration from the original work of Haj\l asz \cite{H96} about Sobolev maps in metric measure spaces and the estimates by Crippa--De Lellis \cite{CDL08}, revisited by Bru\'e--Semola \cite{BS18}: for given $\phi\in L^0(\X)$ non-negative and $R>0$, we can quantify Lusin--Lipschitz regularity by considering the space $H_{\phi,R}(\X)\subset L^0(\X)$ of functions $f$ such that for some non-negative $G\in L^2(\X)$ we have
\[
|f(y)-f(x)|\leq \sfd(x,y) (G(x)+G(y)) e^{\phi(x)+\phi(y)}
\]
for every $x,y$ outside some negligible subset of $\X$
and whose mutual distance does not exceed \(R\). We shall denote by $A_{\phi,R}(f)\subset L^2(\X)$ the set of $G$'s as above and endow $H_{\phi,R}(\X)$ with the natural norm
\[
\|f\|_{H_{\phi,R}(\X)}\coloneqq\sqrt{\|f\|_{L^2(\mm)}^2
+\inf_{G\in A_{\phi,R}(f)}\|G\|_{L^2(\mm)}^2}\quad\text{ for every }f\in H_{\phi,R}(\X)
\]
(see Definition \ref{def:hsob} for the precise definition). It it clear that functions in $H_{\phi,R}(\X)$ are Lusin--Lipschitz, moreover,  crucially, the differential restricted to bounded subsets of $H_{\phi,R}(\X)$ is a closed operator on $L^2$ (see Proposition \ref{prop:closure_tilde_d}). Here the role of the estimates in \cite{BDS21} is to ensure that for $f\in W^{1,2}(\X)$ the functions $f\circ F_0^t$ are uniformly bounded in  $H_{\phi,R}(\X)$  for some properly chosen $\phi$, see Lemma \ref{prop:fF_tbelongstoHphi}.

The second technical ingredient we need is about the very formulation of Hille's theorem and the underlying concept of integration used. Due to the low integrability properties of the differential of the flow, it is not natural to work with vector fields in $L^p$ for $p\geq 1$, but rather it is better to deal with $L^0$ vector fields.  In turn, since $L^0(T\X)$ is not a Banach space (in fact, it is not even locally convex) this creates the need of defining what the integral of a map $t\mapsto v_t\in L^0(T\X)$ is, or more generally of a map $t\mapsto v_t\in\H$ with $\H$ some $L^0$-Hilbert module: we do this in Section \ref{se:intmod} and the approach that we choose might be described as a sort of `pointwise Pettis integral'. More in detail, we first declare $t\mapsto f_t\in L^0(\X)$ to be integrable provided so is the map $t\mapsto f_t(x)\in\R$ for $\mm$-a.e.\ $x$, and then we say that $t\mapsto v_t\in \H$ is integrable provided so is the function $t\mapsto |v_t|\in L^0(\X)$. In this case, its integral can be defined noticing that for any $z\in \H$ the function $t\mapsto \la z,v_t\ra\in L^0(\X)$ is integrable and its integral depends $L^0$-linearly and continuously on $z$, so that Riesz's theorem for $L^0$-Hilbert modules gives the desired notion. 

To this concept of integration we need to attach a form of Hille's theorem. Its standard proof for closed operators $L:E\to F$, with $E,F$ Banach spaces, uses integration in the product space and concludes using the fact that Bochner integral commutes with projections. In turn, this latter property follows from the more general fact that Bochner integration commutes with linear and continuous operators, a fact that trivially follows from the definition of Bochner integral. This exact line of thought does not work in our setting, partly because we cannot work with `pointwise integration' as we are not assuming the closed operator $L$ to be $L^0$-linear. Still, the general approach does, our idea being to first prove that our notion of integration can be realized as limit of Riemann sums in analogy with a classical statement by Hahn, so that its commutation with projections into factors can trivially be established. To the best of our knowledge, previous results in this direction required the additional assumptions on the operator $L$ to be continuous, $L^0$-linear, and with uniformly bounded pointwise norm (see \cite{GM13}).

\bigskip

It is worth to point out that our assumptions cover the case of geodesics in the following sense. Let $(\mu_t)$ be a $W_2$-geodesic so that $\mu_0,\mu_1$ have bounded support and bounded density. Also, let $\eps\in(0,\frac12)$ and consider the restricted geodesic $\nu_t:=\mu_{(1-t)\eps+t(1-\eps)}$. Then there are vector fields $(w_t)$ satisfying our regularity assumptions (see Proposition \ref{prop:eliadanieleregularity} for the precise set of those) such that the associated flow $(F_t^s)$ satisfy $(F_t^s)_*\nu_t=\nu_s$ for every $t,s\in[0,1]$. Moreover, for $\nu_0$-a.e.\ $x\in\X$ the curve $t\mapsto F_0^t(x)$ is a geodesic and the map $F_0^1$ is the only optimal map from $\nu_0$ to $\nu_1$. The fact that these $(w_t)$ exist is a consequence of  the abstract Lewy--Stampacchia inequality \cite{Gigli-Mosconi14} and the estimates in \cite{Gigli14} (just let $w_t:=\nabla \eta_t$ with $\eta_t$ obtained from Kantorovich potentials via the double obstacle problem, see \cite[Theorem 3.13]{Gigli-Mosconi14}). In particular, considering the (only, by \cite{Ambrosio-Trevisan14}) lifting $\ppi$ of $(\nu_t)$, we have that $\ppi=(F_t^\cdot)_*\nu_t$ for every $t\in[0,1]$, where $F_t^\cdot:\X\to C([0,1],\X)$ is the map given by $x\mapsto\big(s\mapsto F_t^s(x)\big)$.
Therefore our results can be read as existence and uniqueness of parallel transport along $\ppi$-a.e.\ geodesic. It is unclear to us whether the initial restriction from $(\mu_t)$ to $(\nu_t)$ is truly necessary, but let us point out that it is well known in geometric analysis and metric geometry that this sort of restriction of geodesics are much more regular than `full' geodesics (see e.g.\ \cite{ColdingNaber12} and \cite{Petrunin98}).

Let us conclude this introduction noticing that the class of non-collapsed ${\sf RCD}$ spaces contains that of finite dimensional Alexandrov spaces with curvature bounded from below equipped with the appropriate Hausdorff measure, thus our construction provides a notion of parallel transport alternative to that in \cite{Petrunin98}. We obtain existence and uniqueness - in place of `only' existence (plus a related second order differentiation formula that we do not explore) - at the price of describing parallel transport not along a single geodesic, but rather along a.e.\ geodesic, in a sense. It is certainly natural to try to compare the two notions, and while we expect them to agree, we do not investigate in this direction. Let us just point out that a key difficulty in doing so is in the different set of techniques involved in the process, which can be summarized in Sobolev vs convex calculus.
\bigskip

\noindent\textbf{Acknowledgements.} The second named author thanks Prof.\ G.R.\ Mingione for some interesting conversations on topics relevant for this work. The third named author acknowledges the support by the Academy of Finland
(project number 314789) and by the Balzan project led by
Prof.\ L.\ Ambrosio.
\section{Preliminaries}
\subsection{Differential calculus on metric measure spaces}
We assume the reader to be familiar with first order Sobolev calculus on metric measure spaces (see e.g.\ \cite{Cheeger00}, \cite{Heinonen07}, \cite{HKST15}, \cite{AmbrosioGigliSavare11}) and second order calculus on ${\sf RCD}$ spaces (see \ \cite{Gigli14}, \cite{Gigli17}, \cite{GP19}) as well as with the notion of non-collapsed ${\sf RCD}$ space, ${\sf ncRCD}$ space in short (see \cite{GDP17} and the earlier \cite{Gigli12}, \cite{AmbrosioGigliSavare11-2}, \cite{Lott-Villani09}, \cite{Sturm06I,Sturm06II}).

Here we only recall those notions that will most frequently be used in the manuscript.

\bigskip

All metric measure spaces $(\X,\sfd,\mm)$ are assumed to be complete and separable as metric spaces and with the measure $\mm$ to be Borel, non-negative, non-zero, and finite on balls. Such a space is called locally doubling provided for any $x\in\X$ and $R>0$ there is a constant $C>0$ such that
\begin{equation}
\label{eq:defdoubl}
\mm(B_{2r}(x))\leq C\mm(B_r(x))\qquad\forall r\in(0,R)
\end{equation}
and locally uniformly doubling provided $C$ can be chosen independent of $x$.

Locally uniformly doubling spaces are regular enough to allow for estimates for the (restricted) maximal functions similar to those available in the Euclidean setting. We recall that given $f \in L^1_{loc}(\X,\mm)$ and  $\lambda > 0$ one  defines
\[
M_{\lambda} f(x):= \sup_{0 < r < \lambda} \frac1{\mm(B_r(x))}\int_{B_r(x)} |f|\, \d \mm.
\]
Then on a locally uniformly doubling space we have that:  for every $1<p\le \infty$ and $\lambda>0$ there exists a constant $C_{p,\lambda}$ such that
\begin{equation}
\label{eq:estimate_L^p-strong_locmaxfunct}
\|M_{\lambda} f \|_{L^p} \le C_{p,\lambda} \| f \|_{L^p} \text{ for every }f \in L^p(\X,\mm).
\end{equation}
See e.g.\ \cite[Theorem 3.5.6]{HKST15} for a proof (notice that in such reference the measure is assumed to be doubling, i.e.\ with the constant in \eqref{eq:defdoubl} to be independent of both $x$ and $R$, however as the argument in the proof shows, since we are considering the restricted maximal function, this is not really an issue).

We shall denote by $\lip(f):\X\to[0,\infty]$ the `local Lipschitz constant' of the function $f:\X\to\R$ defined as
\[
\lip(f)(x):=\lims_{y\to x}\frac{|f(y)-f(x)|}{\sfd(y,x)} 
\]
if $x$ is not isolated, $0$ otherwise. We say that $(\X,\sfd,\mm)$ satisfies a weak local $(1,1)$ Poincar\'{e} inequality, provided for every $R>0$ there exists $C_P(R)$ and $\lambda \ge 1$ such that for every $f:\X\to\R$ Lipschitz, $x \in \X$, $0<r <R$ we have
\begin{equation*}
\frac{1}{\mm(B_r(x))}\int_{B_r(x)} \left| f - f_{B_r(x)}  \right|\, \d \mm \le C_P r \frac{1}{\mm(B_{\lambda r}(x))}\int_{B_{\lambda r}(x)} \lip (f)\,\d \mm,
\end{equation*}
where $f_{B_r(x)}:=\frac{1}{\mm(B_r(x))}\int_{B_r(x)}f\,\d\mm$. 

We say that $(\X,\sfd,\mm)$ is a PI space if it is locally uniformly doubling and satisfies a weak local $(1,1)$ Poincar\'{e} inequality.  For us, this class of spaces is relevant because the results in \cite{Cheeger00} apply and because finite dimensional ${\sf RCD}$ spaces are PI spaces (see  \cite{Rajala12}). On PI spaces, the following maximal estimate holds (see \cite[Theorem 3.2]{HK00}): given $f \in W^{1,2}(\X)$, there exists a $\mm$-null set $N$ such that 
\begin{equation}
\label{eq:local_max_est_sobolev}
|f(x)-f(y)| \le C \big( M_{4R} (|Df|)(x) + M_{4R} (|Df|)(y) \big) \sfd(x,y) \qquad\forall x,y \in \X \setminus N,\,\sfd(x,y)\le R.
\end{equation}

As for what concerns Sobolev differential calculus, we shall rely on the theory of $L^0$-normed modules as developed in \cite{Gigli14}. A crucial feature of the differentiation operator $\d:W^{1,2}(\X)\to L^0(T^*\X)$ that will be frequently used is the locality, i.e.
\[
\d f=\d g,\qquad\mm-a.e.\ on\ \{f=g\},\qquad\forall f,g\in W^{1,2}(\X).
\]

\bigskip

On ${\sf RCD}$ spaces, we shall consider the following notion of 'test functions' (first introduced in \cite{Savare13}):
\[
\test\X:=\big\{f\in  W^{1,2}(\X):\ f,|\d f|\in L^{\infty}(\X),\ f \in D(\Delta),\  \Delta f\in  W^{1,2}(\X)\big\}
\]
A crucial regularity result concerning these functions is
\begin{equation}
\label{eq:leibgrad}
\begin{split}
f\in\test\X,\ w\in L^\infty\cap W^{1,2}_C(T\X)\quad\Rightarrow\quad \d f(w)&\in W^{1,2}(\X)\quad{\rm with }\\
\quad \d(\d f(w))&={\rm Hess}(f)(w,\cdot)+\la\nabla_\cdot w,\nabla f\ra,
\end{split}
\end{equation}
where in the above, following \cite{Gigli14}, we say that $w\in W^{1,2}_C(T\X)$ provided there is $\nabla w \in L^2(T^{\otimes 2}\X)$, which is easily seen to be unique, such that:
\[
\int h\, \nabla w:  (\nabla g_1\otimes \nabla g_2)\,\d\mm=\int-\la w,\nabla g_2\ra\,{\rm div}(h\nabla g_1)-h\,{\rm Hess}(g_2)(w,\nabla g_1)\,\d\mm,
\]
for every $g_1,g_2,h \in \test\X$.

\subsection{Regular Lagrangian flows on \texorpdfstring{\({\sf RCD}(K,\infty)\)}{RCD} spaces}
\label{sec:rlf} We recall in this section the basics of the theory of Regular Lagrangian Flows on ${\sf RCD}$ spaces as developed in \cite{Ambrosio-Trevisan14} (and modelled upon the analogous theory on Euclidean setting proposed in  \cite{Ambrosio04}, see also the original work \cite{DiPerna-Lions89}). The concept of Regular Lagrangian Flow provides a counterpart to the one of `flow of a vector field', as made rigorous in the following definition:

\begin{definition}[Regular Lagrangian Flow]
\label{def:rlf}
Let $(\X,\sfd,\mm)$ be a ${\sf RCD}(K,\infty)$ space, for some $K\in\R$ and $(w_t)\in L^1([0,1], L^2(T\X))$.  Then $F:[0,1]^2\times \X\to \X $  is said to be a \emph{Regular Lagrangian Flow} for $w$ provided it is Borel and the following properties are verified:
\begin{itemize}
\item[i)] For some  $C>0$  we have
\begin{equation}
\label{eq:boundcompr}
(F_t^s)_*\mm\leq C  \mm\qquad\forall t,s\in[0,1].
\end{equation}
\item[$\rm ii)$] For every $x\in\X$, $t\in[0,1]$ the curve $[0,1]\ni s \mapsto F_t^s(x)\in\X$
is continuous and satisfies $F_t^t(x)=x$.
\item[$\rm iii)$] Given $f\in W^{1,2}(\X)$ and $t\in[0,1]$, one has that for $\mm$-a.e.\ $x\in\X$
the map $[0,1]\ni s\mapsto(f\circ F_t^s)(x)\in\R$ belongs to $W^{1,1}(0,1)$ and satisfies
\begin{equation}\label{eq:RLF}
\frac{\d}{\d s}\,(f\circ F_t^s)(x)=\d f(w_s)\big(F_t^s(x)\big)
\quad\text{ for }\mathcal{L}^1\text{-a.e.\ }s\in[0,1].
\end{equation}
\end{itemize}
\end{definition}
We shall typically write $(F_t^s)$ in place of $F$ for Regular Lagrangian Flows.
The following crucial result establishes existence and uniqueness of RLF's (the assumptions are a bit suboptimal, but sufficient for our purposes):
\begin{theorem}\label{thm:AT}
Let $(\X,\sfd,\mm)$ be a ${\sf RCD}(K,\infty)$ space. Let $(w_t)\in L^1([0,1],W^{1,2}_{C}(T\X))$ be such that $w_t\in  D(\div)$ for a.e.\ $t\in[0,1]$ and $|w_t|,\div (w_t)\in L^\infty([0,1]\times\X)$.  

Then there exists a unique regular Lagrangian flow ${(F_t^s)}$ of  $w$. Uniqueness has to be intended as: if $(\tilde F_t^s)$ is another
Regular Lagrangian Flow, then for every $t\in[0,1]$ we have that for  $\mm$-a.e.\ point $x\in\X$ it holds  $\tilde {F}_t^s(x)=F_t^s(x)$ for every $s\in[0,1]$.
\end{theorem}
Notice that the uniqueness statement and the very definition of RLF, imply the following group(oid) property:
\begin{equation}
\label{eq:groupRLF}
\text{for every $t,s,r\in[0,1]$ we have}\qquad F_r^s\circ F_t^r=F_{t}^s\qquad\mm-a.e..
\end{equation}
It can be proved that for $\mm$-a.e.\ $x$ the metric speed  of the curve $s\mapsto F^s_t(x)$ is exactly $|w_s|(F_t^s(x))$, in particular, we have the uniform Lipschitz estimate: for every $t,s_1,s_2\in[0,1]$ we have
\begin{equation}
\label{eq:unifspFl}
\sfd(F_t^{s_2}(x),F_t^{s_1}(x))\leq |s_2-s_1|\|w\|_{L^\infty},\qquad\mm-a.e.\ x\in\X.
\end{equation}
A consequence of the continuity of $s\mapsto F_t^s(x)$ and of the bounded compression property \eqref{eq:boundcompr} is that
\begin{equation}
\label{eq:contlp}
f\in L^p(\X)\qquad\Rightarrow\qquad [0,1]^2\ni (t,s)\mapsto f\circ F_t^s\in L^p(\X)\text{ is continuous,}\qquad\forall p\in[1,\infty).
\end{equation}
Indeed, the bounded compression property gives the uniform estimate
\begin{equation}
\label{eq:dabc}
\|f\circ F_t^s\|_{L^p}^p\leq C\|f\|^p_{L^p}\qquad\forall t,s\in[0,1],
\end{equation}
and the continuity of $s\mapsto F_t^s(x)$ ensures that for $f$ Lipschitz with bounded support $s\mapsto f\circ F_t^s\in L^p$ is continuous for every $t\in[0,1]$. From the density of such Lipschitz functions in $L^p$ and \eqref{eq:dabc} we deduce that  $s\mapsto f\circ F_t^s\in L^p$ is continuous for every $t\in[0,1]$ and $f\in L^p(\X)$. Now notice that
\[
\begin{split}
\| f\circ F_{t'}^{s'}-f\circ F_t^s\|_{L^p}&\leq\|f\circ F_{t'}^{s'}- f\circ F_{t'}^s\|_{L^p}+\| f\circ F_{t'}^s- f\circ F_{t}^{s}\|_{L^p}\\
\text{(by \eqref{eq:dabc}, \eqref{eq:groupRLF})}\qquad&\leq C^{\frac1p} \|f\circ F_{s}^{s'}- f\|_{L^p}+\| (f\circ F_{t}^s)\circ F_{t'}^t- f\circ F_{t}^{s}\|_{L^p}\\
\text{(by \eqref{eq:dabc}, \eqref{eq:groupRLF})}\qquad&\leq C^{\frac1p} \|f\circ F_{s}^{s'}- f\|_{L^p}+C^{\frac{1}{p}}\| f\circ F_{t}^s- (f\circ F_{t}^{s})\circ F_{t}^{t'}\|_{L^p}
\end{split}
\]
so that the claim \eqref{eq:contlp} follows from what already proved.

In a similar way, if for some bounded set $B\subset\X$ we have that $\supp(w_t)\subset B$ for a.e.\ $t$ we have that
\begin{equation}
\label{eq:contl0}
f\in L^0(\X)\qquad\Rightarrow\qquad [0,1]^2\ni (t,s)\mapsto f\circ F_t^s\in L^0(\X)\text{ is continuous.}
\end{equation}
Indeed, recall that the topology of $L^0(\X)$ is metrized by the distance 
\[
\sfd_{L^0}(f,g):=\int1\wedge |f-g|\,\d\mm',
\]
where $\mm'$ is any Borel probability measure with $\mm\ll\mm'\ll\mm$. We thus pick $\mm'$ so that $\mm'\restr B=c\mm\restr B$ for some $c>0$ and notice that the assumption on $w_t$ ensures that $F_t^s$ is the identity outside $B$ for every $t,s\in[0,1]$. Hence the bounded compression property gives $(F_t^s)_*\mm'\leq C\mm'$ for any $t,s$ and therefore, in analogy with \eqref{eq:dabc}, we have
\[
\sfd_{L^0}(f\circ F_t^s,g\circ F_t^s)=\int1\wedge |f-g|\circ F_t^s\,\d\mm'=\int1\wedge |f-g|\,\d(F_t^s)_*\mm'\leq C\int1\wedge |f-g|\,\d\mm'=C\sfd_{L^0}(f,g)
\]
and the claim \eqref{eq:contl0} follows along the same lines as \eqref{eq:contlp}.

From \eqref{eq:contlp} for $p=2$ and by integrating \eqref{eq:RLF} we see that
\begin{equation}
\label{eq:int RLF}
f\circ F_t^{s_2}-f\circ F_t^{s_1}=\int_{s_1}^{s_2}\d f(w_r)\circ F_t^r\,\d r\qquad \mm-a.e.,\qquad\forall f\in W^{1,2}(\X)
\end{equation}
for any $t,s_1,s_2\in[0,1]$ with $s_1<s_2$, where the integral is intended in the pointwise a.e.\ sense (or, equivalently, in the Bochner sense). It is then also clear that for a.e.\ $s\in[0,1]$ we have
\begin{equation}
\label{eq:c1RLF}
\lim_{h\to 0}\frac{f\circ F_t^{s+h}-f\circ F^s_t}{h}=\d f(w_s)\circ F_t^s,\qquad\mm-a.e.,\qquad\forall f\in W^{1,2}(\X),
\end{equation}
the limit being in $L^2(\X)$ and by usual maximal-type arguments it is not hard to see that the exceptional set of times may be chosen independent of $f\in W^{1,2}(\X)$ (see e.g.\ the arguments in Proposition \ref{prop:mainreg}).

\bigskip

We now turn to the main regularity estimates on Regular Lagrangian Flows as established in \cite{BDS21} (see also the earlier \cite{BS18}). In order to state the result, let us recall a few definitions: given $f:\X\to\R\cup\{\pm\infty\}$ Borel and $x\in\X$ we define the approximate limsup of $f$ at $x$ as
\[
\aplims_{y\to x}f(y):=\inf\big\{\lambda\in\R\cup\{\pm\infty\}\ :\ x\text{ is a density point of }\{f\leq\lambda\}\big\}
\]
and analogously we define the approximate liminf $\aplimi_{y\to x}f(y)$. Then the approximate Lipschitz constant at $x$ of a Borel function $F:\X\to\Y$ with $(\Y,\sfd_\Y)$ metric space is defined as
\[
\aplip F(x):=\aplims_{y\to x}\frac{\sfd_\Y(F(y),F(x))}{\sfd(y,x)}.
\]
Finally, we recall that a Borel map $F:\X\to\Y$ is said Lusin--Lipschitz provided  there exists a sequence \((E_i)\)
of Borel sets in \(\X\) such that \(\mm\big(\X\setminus\bigcup_i E_i\big)=0\)
and \(F\restr{E_i}\) is Lipschitz for all \(i\in\N\).

\bigskip
We then have the following result (extracted mainly from  \cite[Theorem 1.6, Proposition 3.3]{BDS21}):
\begin{proposition}
\label{prop:eliadanieleregularity}
Let $(\X,\d,\mm)$ be a ${\sf ncRCD}(K,N)$ space with $K \in \mathbb{R}$,
$N<\infty$, and  $(w_t) \in L^2([0,1], W^{1,2}_{C}(T\X))$ be such that $|w_t|,\div (w_t)\in L^\infty([0,1]\times\X)$ and for some $\bar x\in \X$ and $R>0$ we have $\supp(w_t)\subset B_R(\bar x)$ for a.e.\ $t\in[0,1]$.

Then there exists a nonnegative function $(g_t)\in L^2([0,1], L^2(\X,\mm))$ such that for any $t,s\in[0,1]$ we have
\begin{equation}
\label{eq:estimate3_factor}
\begin{aligned}
& \aplip F_t^s(x) \le e^{ \int_{t\wedge s}^{t\vee s} g_r(F_t^{r}(x))\, \d r}  \qquad \mm-a.e.\ x\in\X
\end{aligned}
\end{equation}
and, for every $t\in[0,1]$, a nonnegative function $\bar g_t\in L^2_{loc}(\X,\mm)$ such that for any $s\in[0,1]$ we have
\begin{equation}
\label{eq:estimate4_factor}
\frac{\sfd(F_{t}^s(x), F_{t}^s(y))}{\sfd(x,y)} \le e^{{\bar g_t}(x)+{\bar g_t}(y)},\qquad\forall x,y\in\X.
\end{equation}
\end{proposition}
\begin{proof} The estimate \eqref{eq:estimate4_factor} for $t\leq s$ is the content of  \cite[Proposition 3.3]{BDS21}, the case $t\geq s$ then also follows noticing that $s\mapsto F_t^{1-s}$ is the Regular Lagrangian Flow of $(-w_t)$.

We pass to \eqref{eq:estimate3_factor} and notice that
 \cite[Theorem 1.6]{BDS21} states that for $t,s\in[0,1]$, $t\leq s$, for $\mm$-a.e.\ $x\in\X$ we have
\begin{equation}
\label{eq:thm16}
e^{-\int_t^s g_r( F_t^r(x))\,\d r}\leq 
\aplimi_{y\to x}\frac{\sfd(F_t^s(y),F_t^s(x))}{\sfd(y,x)}\leq \aplims_{y\to x}\frac{\sfd(F_t^s(y),F_t^s(x))}{\sfd(y,x)}\le e^{\int_t^s g_r( F_t^r(x))\,\d r},
\end{equation}
where the functions $g_r$ satisfy the bound $\|g_r\|_{L^2}\leq C(K,N,R)(\||\nabla w_r| \|_{L^2}+\|\div(w_r)\|_{L^\infty})$
(thus from the integrability assumptions on $(w_t)$ the integrability of $(g_t)$ as in the statement follows).  In particular, this gives \eqref{eq:estimate3_factor} for $t\leq s$. For the case $t\geq s$ we assume for the moment that for $\mm$-a.e.\ $x$ we can establish the `change of variable formula' marked with a star in the following computation
\[
\aplimi_{y\to x}\frac{\sfd(F_s^t(y),F_s^t(x))}{\sfd(y,x)}\stackrel*=\aplimi_{w\to F_s^t(x)}\frac{\sfd(w,F_s^t(x))}{\sfd(F_t^s(w),x)}=\frac1{\aplims_{w\to F_s^t(x)}\frac{\sfd(F_t^s(w),x)}{\sfd(w,F_s^t(x))}}.
\]
Then \eqref{eq:thm16} (that we apply with $t,s$ swapped) gives
\[
\aplims_{w\to F_s^t(x)}\frac{\sfd(F_t^s(w),x)}{\sfd(w,F_s^t(x))}\leq e^{\int_s^t g_r( F_s^r(x))\,\d r}\qquad\mm-a.e.\ x.
\]
Writing $x=F_t^s(z)$ and keeping in mind the  bounded compression property \eqref{eq:boundcompr} and the group property  \eqref{eq:groupRLF} we deduce that
\[
\aplims_{w\to z}\frac{\sfd(F_t^s(w),F_t^s(z))}{\sfd(w,z)}\leq e^{\int_s^t g_r( F_t^r(z))\,\d r}\qquad\mm-a.e.\ z,
\]
which is \eqref{eq:estimate3_factor} in the case $t\geq s$. Thus it remains to prove that the starred identity in the above holds for $\mm$-a.e.\ $x\in\X$. By the very definition of $\aplimi$, this will follow if we show that for $\mm$-a.e.\ $x\in\X$ we have that: for any $A\subset \X$ Borel we have that $F_s^t(x)$ is a Lebesgue point of $A$ if and only if $x$ is a Lebesgue point of $(F_s^t)^{-1}(A)$. Since \eqref{eq:estimate4_factor} ensures that $F_t^s$ and $F_s^t$ are Lusin--Lipschitz, \eqref{eq:groupRLF} that they are one the essential inverse of the other and recalling \eqref{eq:boundcompr}, the conclusion follows from Lemma \ref{le:changevar} below.
\end{proof}

\begin{lemma}\label{le:changevar} Let $(\X,\sfd,\mm)$ be a locally doubling space, $T,S:\X\to\X$ Borel maps, Lusin--Lipschitz, such that $T\circ S=\Id$ and $S\circ T=\Id$ $\mm$-a.e.\ and finally so that $T_*\mm\leq C\mm$ and $S_*\mm\leq C\mm$ for some $C>0$.

Then for $\mm$-a.e.\ $x\in\X$ the following holds: for any $A\subset\X$ Borel we have that $T(x)$ is a Lebesgue point of $A$ if and only if $x$ is a Lebesgue point of $T^{-1}(A)$.
\end{lemma}
\begin{proof}
The assumptions about essential invertibility and bounded compression give the existence of $X_1,X_2\subset \X$ Borel of full measure such that $T,S$ are invertible bijections from $X_1$ to $X_2$ and vice versa, respectively. Also, from the Lusin--Lipschitz regularity we see that $\mm$-a.a.\ $X_1$ can be covered by a countable number of Borel sets $B$ such that $T\restr B$ is Lipschitz and $S\restr{T(B)}$ is also Lipschitz. Fix such $B$, let $L$ be a bound on the Lipschitz constants of $T\restr B$  and $S\restr{T(B)}$ and let $B'\subset B$ be the set of $x$'s such that $x$ is a Lebesgue point of $B$ and $T(x)$ is a Lebesgue point of $T(B)=S^{-1}(B)\cap X_2$. Clearly, $\mm(B\setminus B')=0$, so the conclusion will follow if we show that any $x\in B'$ satisfies the required conditions. Thus fix $x\in B'$, let $A\subset \X$ be Borel and $r\in(0,1)$.  Assume that $x$ is a Lebesgue point of $T^{-1}(A)$. Since $T\restr B$ is $L$-Lipschitz, we have that $T(B_{r/L}(x)\cap B)\subset B_r(T(x))\cap T(B)$ and therefore
\begin{equation}
\label{eq:palla1}
\begin{split}
\mm\big(B_r(T(x))\big)&\geq\mm\big(B_r(T(x))\cap T(B)\big)\geq\mm\big(T(B_{r/L}(x)\cap B)\big)=\mm\big(S^{-1}(B_{r/L}(x)\cap B)\big)\\
&=S_*\mm\big(B_{r/L}(x)\cap B\big)\geq C^{-1}\mm\big(B_{r/L}(x)\cap B\big)\geq C^{-1} D^{-(2\log_2(L)+1)}\mm\big(B_{rL}(x)\cap B\big) ,
\end{split}
\end{equation}
where $D$ is the local doubling constant at $x$  and we used the bound $S_*\mm\geq C^{-1}\mm$ (which follows taking $S_*$ on both sides of $T_*\mm\leq C\mm$). On the other hand we have
\[
\begin{split}
\mm\big(B_r(T(x))\setminus A\big)\leq \mm\big(B_r(T(x))\setminus(A\cap T(B))\big)=\mm\big(B_r(T(x))\setminus T(B)\big)+\mm\restr{T(B)}\big(B_r(T(x))\setminus A\big)
\end{split}
\]
and since we assumed $T(x)$ to be a Lebesgue point for $T(B)$, when we divide the first addend in the rightmost side by $\mm(B_r(T(x)))$ and let $r\downarrow0$ it converges to 0. For the other addend we have the estimate
\[
\begin{split}
\mm\restr{T(B)}\big(B_r(T(x))\setminus A\big)&\leq C(T_*\mm)\restr{T(B)}\big(B_r(T(x))\setminus A\big)\\
&=C\mm\big(T^{-1}\big(T(B)\cap B_r(T(x))\big)\setminus T^{-1}(A)\big)\\
&=C\mm\big(S\big(T(B)\cap B_r(T(x))\big)\setminus T^{-1}(A)\big)\\
&\leq C\mm\big( B_{Lr}(x)\setminus T^{-1}(A)\big).
\end{split}
\]
Thus recalling \eqref{eq:palla1} we get
\[
\begin{split}
\lims_{r\downarrow0}\frac{\mm\big(B_r(T(x))\setminus A\big)}{\mm\big(B_r(T(x))\big)}\leq  C^2 D^{2\log_2(L)+1}\lims_{r\downarrow0}\frac{\mm\big( B_{Lr}(x)\setminus T^{-1}(A)\big)}{\mm(B_{Lr}(x))}\frac{\mm(B_{Lr}(x))}{\mm\big(B_{Lr}(x)\cap B\big)}=0 ,
\end{split}
\]
having used the assumption that $x$ is a Lebesgue point of $B$ and $T^{-1}(A)$. This proves that $T(x)$ is a Lebesgue point of $A$. The converse implication is proved analogously.
\end{proof}
\begin{remark}{\rm
The results in \cite{BDS21} are based on the slightly less stringent assumption that there is an $L^2$ control over only the symmetric part of the covariant derivative. We phrased the result in this weaker formulation because in any case we will need a control on the full covariant derivative later on (when discussing the properties of the convective derivative introduced in Definition \ref{def:Dt}).

Notice also that in Proposition \ref{prop:eliadanieleregularity} one needs to assume $L^2$ integrability in time, rather than the $L^1$ integrability which is sufficient in Theorem \ref{thm:AT}. We shall therefore make this assumption throughout the manuscript. In any case, notice that by a simple reparametrization argument one can always reduce to the case of vector fields bounded in $W^{1,2}_C$ at the only price of reparametrizing the flow in time.
}\fr\end{remark}

It is clear from \eqref{eq:estimate4_factor} that 
\begin{equation}
\label{eq:unifLip2}
\text{for any $c>0$ and $t\in[0,1]$ the restriction of $F^s_t$ to $\{\bar g_t\leq c\}$ is Lipschitz, uniformly in $s\in[0,1]$.}
\end{equation}

\section{Functional analytic tools}
\subsection{Differential of Lusin--Lipschitz maps}
\label{sec:diff}
In this section, inspired by some discussions in \cite{Gigli14}, \cite{GP16}  we develop a language for the differential of a Lusin--Lipschitz and invertible map $\varphi \colon \X \rightarrow \X$ of  bounded compression. The results will be applied to the flow maps $({F_t^s})$.

We start noticing that if $f:\X\to\R$ is Lusin--Lipschitz and $(E_j)$ is a Borel partition of $\mm_\X$-a.a.\ $\X$ such that $f\restr{E_j}$ is Lipschitz
for every $j$, then the formula
\[
\d f:=\sum_{j\in\N}\nchi_{E_j}\d g_j,\qquad\text{where $g_j\in\LIP(\X)$ is equal to $f$ on $E_j$ for every $j$}
\]
defines an element of $L^0(T\X)$ that, by locality of the differential, is independent of the particular functions $g_j$ and Borel sets $E_j$ as above. We shall refer to $\d f$ as the differential of $f$ and notice that this definition poses no ambiguity as, again by locality, for $f$ Sobolev and Lusin--Lipschitz the definition above produces the same differential of $f$ as defined in \cite{Gigli14}.

We then notice the following simple lemma:
\begin{lemma}
Let $(\X,\sfd,\mm)$ be locally uniformly doubling. Let $(\Y,\sfd_\Y)$ be a complete space and $\varphi:\X\to\Y$ a Lusin--Lipschitz map.
Let $(E_i)_{i\in\N}$ be a Borel partition of $\X$ up to $\mm$-null sets such that $\varphi\restr{E_i}$ is a Lipschitz map for every $i\in\N$. Then
\begin{equation}
\label{eq:characterization_aplip2}
\aplip \varphi=\sum_{i=1}^\infty \nchi_{E_i} {\rm lip}(\varphi \restr{E_i})\qquad\mm-a.e..
\end{equation}
\end{lemma}
\begin{proof}
Let $i\in\N$ be fixed. Let $x\in E_i$ be a density point of $E_i$; recall that $\mm$-a.e.\ point of $E_i$ has this property.
As observed, e.g., in the paragraph following \cite[Eq.\ (2.6)]{GT20}, the quantity $\aplip \varphi(x)$ is independent of the
behaviour of $\varphi$ outside $E_i$. Therefore, \cite[Proposition 2.5]{GT20} grants that $\aplip \varphi(x)=\lip(\varphi\restr{E_i})(x)$,
whence \eqref{eq:characterization_aplip2} follows.
\end{proof}
We come to the definition of differential of a Lusin--Lipschitz map. In what follows, it will be useful to notice that if $\varphi:\X\to\Y$ is Lusin--Lipschitz with $\varphi_*\mm_\X\ll\mm_\Y$ and $f:\Y\to\R$ is also Lusin--Lipschitz, then $f\circ\varphi:\X\to\R$ is  Lusin--Lipschitz as well. Indeed, let  $(F_i)$ (resp.\ $(E_j)$) be a Borel partition of $\mm_\Y$-a.a.\ $\Y$ (resp.\ $\mm_\X$-a.a.\ $\X$) such that $f\restr{F_i}$ (resp.\ $\varphi\restr{E_j}$) is Lipschitz for every $i$ (resp. $j$). Then since $\varphi_*\mm_\X\ll\mm_\Y$, we have that $(E_j\cap\varphi^{-1}(F_i))$ is a Borel partition of $\mm_\X$-a.a.\ $\X$ such that $f\circ\varphi\restr{E_j\cap\varphi^{-1}(F_i)}$ is Lipschitz for every $i,j$.

In particular, the right hand side in formula \eqref{eq:char_diff_LL} below makes sense.
\begin{theorem}[Differential of a Lusin--Lipschitz map]\label{thm:diff_LL}
Let \((\X,\sfd_\X,\mm_\X)\), \((\Y,\sfd_\Y,\mm_\Y)\) be metric measure spaces.
Suppose \((\X,\sfd_\X) \) is geodesic, \(\mm_\X\) is locally uniformly doubling, and \((\Y,\sfd_\Y,\mm_\Y)\) is a PI space.
Let \(\varphi\colon\X\to\Y\) be an essentially invertible Lusin--Lipschitz map such that
\(\varphi_*\mm_\X\ll\mm_\Y\) and \(\varphi^{-1}_*\mm_\Y\ll\mm_\X\), where  $\varphi^{-1}$ is an essential inverse of $\varphi$ (that is uniquely defined up to $\mm_\Y$-negligible sets).

Then there exists a unique
linear and continuous operator \(\d\varphi\colon L^0(T\X)\to L^0(T\Y)\) such that
\begin{equation}\label{eq:char_diff_LL}
\d f\big(\d\varphi(v)\big)\circ\varphi=\d (f\circ\varphi)(v)\quad\text{ holds }\mm_\X\text{-a.e.\ on }\X,
\end{equation}
for every  \(v\in L^0(T\X)\) and $f:\Y\to\R$ Lusin--Lipschitz.

Moreover, it holds that
\begin{equation}\label{eq:pn_diff_LL}
\big|\d\varphi(v)\big|\circ \varphi \leq\aplip \varphi\,|v|\;\;\;
\mm_\X\text{-a.e.\ on }\X,\quad\text{ for every }v\in L^0(T\X).
\end{equation}
\end{theorem}
\begin{proof} For $f:\Y\to\R$ Lusin--Lipschitz and $v\in L^0(T\X)$ we define $T_v(f)\in L^0(T\Y)$ as
\[
T_v(f):=\big(\d(f\circ \varphi)(v)\big)\circ\varphi^{-1}
\]
(notice that $\d(f\circ \varphi)(v)$ is in $L^0(\X)$, so the above defines a function in $L^0(\Y)$ thanks to the assumption $\varphi^{-1}_*\mm_\Y\ll\mm_\X$).
Notice that since, as previously remarked, $f\circ\varphi$ is Lusin--Lipschitz, the right hand side of the above is well defined. We claim that for any Borel partition \((E_j)_j\) of $\mm_\X$-a.a.\ \(\X\) such that \(\varphi\restr{E_j}\) is Lipschitz for all \(j\in\N\) we have
\begin{equation}
\label{eq:claimtvf}
|T_v(f)|\leq \Big(|v|\sum_{j\in\N}{\lip}(\varphi\restr{E_j})\nchi_{E_j}\Big)\circ\varphi^{-1}|\d f|\qquad\mm_\X-a.e..
\end{equation}
To see this we start noticing that $|T_v(f)|\leq |v|\circ\varphi^{-1}|\d(f\circ\varphi)|\circ\varphi^{-1}$, so the claim will follow if we show that
\begin{equation}
\label{eq:pertvf}
|\d(f\circ \varphi)|\leq |\d f|\circ\varphi\Big(\sum_{j\in\N}{ \lip}(\varphi\restr{E_j})\nchi_{E_j}\Big).
\end{equation}
To see this, let $(F_i)$ be a Borel partition of $\mm_\Y$-a.a.\ $\Y$ such that $f\restr{F_i}$ is Lipschitz for every $i$ and recall that $(E_j\cap\varphi^{-1}(F_i))$ is a Borel partition of $\mm_\X$-a.a.\ $\X$ such that $f\circ\varphi\restr{E_j\cap\varphi^{-1}(F_i)}$ is Lipschitz for every $i,j$. For every $i,j\in\N$ let $h_i:\Y\to\R$ be Lipschitz and equal to $f$ on $F_i$ and $g_{i,j}:\X\to\R$ be Lipschitz and equal to $f\circ\varphi$ on $E_j\cap\varphi^{-1}(F_i)$. 

Note that since $\mm_\X$ is locally doubling we have \(\lip(g_{i,j})=\lip(g_{i,j}\restr{E_{j}\cap \varphi^{-1}(F_i)})\) \(\mm_\X\)-a.e.\ on \(E_j\cap\varphi^{-1}(F_i)\) 
(see e.g.\ \cite{GT20}) and since \((\Y,\sfd_\Y,\mm_\Y)\) is PI we have  \(\lip(h_i)=|\d h_i|\) \(\mm_\Y\)-a.e.\ on \(\Y\) (see \cite{Cheeger00}). Then $\mm_\X$-a.e.\ on $E_j\cap\varphi^{-1}(F_i)$ we have
\[
\begin{split}
|\d(f\circ \varphi)|&=|\d g_{i,j}|\leq\lip(g_{i,j})=\lip(g_{i,j}\restr{E_j\cap\varphi^{-1}(F_i)})=\lip\big((f\circ\varphi)\restr{E_j\cap\varphi^{-1}(F_i)}\big)\\
&\leq \lip(\varphi\restr{E_j\cap\varphi^{-1}(F_i)})\lip (f\restr{\varphi(E_j\cap\varphi^{-1}(F_i))}) \circ \varphi\leq \lip(\varphi\restr{E_j})\lip (f\restr{F_i}) \circ \varphi\\
&\leq \lip(\varphi\restr{E_j})\lip(h_i) \circ \varphi=\lip(\varphi\restr{E_j})|\d h_i| \circ \varphi =\lip(\varphi\restr{E_j})|\d f| \circ \varphi
\end{split}
\]
whence \eqref{eq:pertvf} - and thus also \eqref{eq:claimtvf} - follows. From \eqref{eq:claimtvf} and the linearity of $T_v$ it follows that if $\d f=\d f'$ on some Borel set $E\subset \X$, then $T_v(f)=T_v(f')$ on $E$ as well. Therefore the operator $L_v:\{\text{differentials of Lusin--Lipschitz functions on $\Y$}\}\to L^0(\Y)$ defined by
\[
L_v(\d f):=T_v(f)
\]
is well defined and satisfies $L_v(\nchi_E\d f)=\nchi_EL_v(\d f)$. Also, since \eqref{eq:claimtvf} holds for any partition $(E_j)$ as above, from \eqref{eq:characterization_aplip2} we see that
\begin{equation}
\label{eq:lvdf}
|L_v(\d f)|\leq  \big(|v|\aplip\varphi\big)\circ\varphi^{-1}\,|\d f|.
\end{equation}
Now notice that   \(W^{1,2}(\Y)\) is reflexive (see  \cite{Cheeger00}), thus  \(\LIP(\Y)\cap W^{1,2}(\Y)\) is strongly dense in \(W^{1,2}(\Y)\) (see \cite{ACM14}), therefore  \(L^0(T^*\Y)\) is generated by \(\big\{\d f\,:\,f\in\LIP(\Y)\big\}\) and so the set $\{\text{differentials of Lusin--Lipschitz functions on $\Y$}\}$ is dense in  \(L^0(T^*\Y)\). Also, from \eqref{eq:lvdf} it is easy to see (see also the arguments in Section \ref{se:setting} that lead to \eqref{eq:unifO}) that $L_v$ is uniformly continuous, hence it can be uniquely extended to a continuous map, still denoted $L_v$, from $L^0(T^*\Y)$ to $L^0(\Y)$ and this extension satisfies
\begin{equation}
\label{eq:lvw}
|L_v(\omega)|\leq  \big(|v|\aplip\varphi\big)\circ\varphi^{-1}\,|\omega|\qquad\mm_\Y-a.e.,\ \forall\omega\in L^0(T^*\Y).
\end{equation}
It is clear from the previous considerations that $L_v$ is also $L^0(\Y)$-linear, and thus an element of $L^0(T^*\Y)^*$. We denote by \(\d\varphi(v)\) the element of \(L^0(T\Y)\) corresponding to \(L_v\in L^0(T^*\Y)^*\). 

It is clear that the map $L^0(T\X)\ni v\mapsto  \d\varphi(v)\in L^0(T\Y)$ is linear and that this assignment is the only one   satisfying \eqref{eq:char_diff_LL}. Finally,   \eqref{eq:pn_diff_LL} follows from  \eqref{eq:lvw}.
\end{proof}

The pointwise norm \(|\d\varphi|\in L^0(\X)\) of the
differential \(\d\varphi\) introduced in Theorem \ref{thm:diff_LL}
is defined as follows:
\begin{equation}\label{eq:def_ptse_norm_bar_d}
|\d\varphi|\coloneqq
\underset{\substack{v\in L^0(T\X): \\ |v|\leq 1\;\mm_\X\text{-a.e.}}}{\mm-{\rm ess\,sup}}
\big|\d\varphi(v)\big|\circ\varphi.
\end{equation}
\begin{proposition}[Basic properties of the differential]
With the same assumptions on $\X,\Y,\varphi$ as in Theorem \ref{thm:diff_LL} the following holds.

We have
\begin{equation}
\label{eq:estimate_differential_aplip}
|{\d}\varphi|\leq \aplip\, \varphi\quad \mm_{\X}-a.e.
\end{equation}
and   for $v \in L^0(T\X)$ the bound 
\begin{equation}
\label{eq:pointwise_norm_estimate}
|{\d} \varphi(v)| \circ \varphi \le |{\d} \varphi| |v| \quad \mm_\X-a.e..
\end{equation}
Moreover, for $f:\Y\to\R$ Lusin--Lipschitz, the function $f\circ \varphi$ is also Lusin--Lipschitz and we have
\begin{equation}\label{eq:techn_comp_cl}
\big|\d(f\circ \varphi)\big|\leq|\d f|\circ  \varphi\,|\d  \varphi|,\qquad \mm_\X-a.e..
\end{equation}
Also, for $v \in L^0(T\X)$ and $h\in L^0(\X)$ we have the identity
\begin{equation}
\label{eq:l0lindf}
\d \varphi (hv)=h\circ \varphi^{-1}\,\d \varphi(v).
\end{equation}
Finally, if $({\rm Z},\sfd_{\rm Z},\mm_{\rm Z})$ is another PI space and $\psi:\Y\to{\rm Z}$ is Lusin--Lipschitz, essentially invertible and such that $\psi_*\mm_\Y\ll\mm_{\rm Z}$, $\psi^{-1}_*\mm_{\rm Z}\ll\mm_\Y$, then  \(\psi\circ\varphi\colon\X\to{\rm Z}\)
is Lusin--Lipschitz and satisfies the identity
\begin{equation}\label{eq:chain_rule_diff_LL}
\d(\psi\circ\varphi)=\d\psi\circ\d\varphi
\end{equation}
and the bound
\begin{equation}\label{eq:ineq_chain_rule_diff_LL}
\big|\d(\psi\circ\varphi)\big|\leq|\d\psi|\circ\varphi\,|\d\varphi|
\quad\mm_\X-a.e..
\end{equation}
\end{proposition}
\begin{proof} The bound \eqref{eq:estimate_differential_aplip} follows directly from the definition and the estimate \eqref{eq:pn_diff_LL}, while \eqref{eq:l0lindf} is a direct consequence of the defining property  \eqref{eq:char_diff_LL}. For  \eqref{eq:pointwise_norm_estimate} we put $v':=|v|^{-1}v$, where $|v|^{-1}$ is intended to be 0 on $\{v=0\}$, notice that $|v'|\leq 1$ $\mm_\X$-a.e.\ and $v=|v|v'$, thus taking \eqref{eq:l0lindf} into account we get
\[
|{\d} \varphi(v)|\circ \varphi=\big|{\d} \varphi(|v|v')\big|\circ\varphi= |v| \,|\d\varphi(v')|\circ \varphi \leq|v| \,|\d\varphi|\qquad\mm_\X-a.e.,
\] 
having used the definition of $|\d \varphi|$ in the last inequality.

To prove  \eqref{eq:techn_comp_cl} we notice that the Lusin--Lipschitz regularity of $f\circ\varphi$ has already been established before Theorem \ref{thm:diff_LL}, then for every $v\in L^0(T\X)$ we have
\[
\begin{split}
|\d(f\circ\varphi)(v)|\stackrel{\eqref{eq:char_diff_LL}}=|\d f(\d\varphi(v))|\circ\varphi\leq |\d f|\circ\varphi |\d\varphi(v)|\circ\varphi\stackrel{\eqref{eq:pointwise_norm_estimate}}\leq  |\d f|\circ\varphi\, |\d\varphi|\,|v|
\end{split}
\]
and \eqref{eq:techn_comp_cl} follows from the arbitrariness of $v$. 

For the last claim, we notice that the fact that $\psi\circ\varphi$ is Lusin--Lipschitz can be proved as we did for the function $f\circ\varphi$ before Theorem \ref{thm:diff_LL}. Now notice that for $f:{\rm Z}\to\R$ Lusin--Lipschitz, the maps $f\circ\psi$ and $f\circ\psi\circ\varphi$ are Lusin--Lipschitz, therefore for any $v\in L^0(T\X)$ we have
\[
\d f(\d\psi(\d\varphi(v)))\circ(\psi\circ\varphi)\stackrel{\eqref{eq:char_diff_LL}}=\d (f\circ\psi)(\d\varphi(v))\circ\varphi\stackrel{\eqref{eq:char_diff_LL}}= \d (f\circ\psi\circ\varphi)(v).
\]
According to Theorem \ref{thm:diff_LL}, this is sufficient to prove \eqref{eq:chain_rule_diff_LL}. Finally, let $v\in L^0(T\X)$  and notice that
\[
\begin{split}
|\d(\psi\circ\varphi)(v)|\circ\psi\circ\varphi\stackrel{\eqref{eq:chain_rule_diff_LL}}=|\d\psi(\d\varphi(v))|\circ\psi\circ\varphi\stackrel{\eqref{eq:pointwise_norm_estimate}}\leq \big(|\d\psi||\d\varphi(v)|\big)\circ\varphi\stackrel{\eqref{eq:pointwise_norm_estimate}}\leq |\d\psi|\circ\varphi|\d\varphi||v|,
\end{split}
\]
thus \eqref{eq:ineq_chain_rule_diff_LL} follows from the very definition \eqref{eq:def_ptse_norm_bar_d}.
\end{proof}
\begin{remark}{\rm
In \cite{Gigli14} the concept of differential for a map $\varphi$ between metric measure spaces has been introduced under the assumptions that $\varphi$ is Lipschitz, essentially invertible and such that $\varphi_*\mm_\X\leq C\mm_\Y$ and $\varphi^{-1}_*\mm_\Y\leq C\mm_\X$ for some $C>0$ (in fact, the existence of the inverse was not really needed in \cite{Gigli14},
but in the general case one has to work with pullback modules
). In this case, $\d \varphi:L^2(T\X)\to L^2(T\Y)$ was defined as the only linear continuous operator such that
\begin{equation}
\label{eq:chardiffl2}
\d f\big(\d\varphi(v)\big)=\d(f\circ\varphi)(v)\circ\varphi^{-1}\;\;\;\mm_\Y\text{-a.e.,}
\quad\text{ for every }f\in W^{1,2}(\Y)\text{ and }v\in L^2(T\X).
\end{equation}
Then under the assumptions of Theorem \ref{thm:diff_LL} it is clear that the above defines the same object as the one given by \eqref{eq:char_diff_LL}. Indeed, from \eqref{eq:chardiffl2} and the fact that Lipschitz functions are locally Sobolev we see that  \eqref{eq:char_diff_LL} holds for $f$ Lipschitz, and then by locality for $f$ Lusin--Lipschitz. Conversely, once we know  \eqref{eq:char_diff_LL}  we have that \eqref{eq:chardiffl2} holds at least for $f$ Lipschitz, and the fact that it holds for $f$ Sobolev follows from the same density arguments used in proving Theorem \ref{thm:diff_LL}.

Finally, the fact that under the current assumptions $\d \varphi$ maps $L^2(T\X)$ to $L^2(T\Y)$ follows from the bound \eqref{eq:pn_diff_LL} and $\varphi_*\mm_\X\leq C\mm_\Y$.
}\fr\end{remark}%

Notice that under the assumptions of Proposition \ref{prop:eliadanieleregularity}, we know from \eqref{eq:unifLip2} that $F_t^s$ is Lusin--Lipschitz, from \eqref{eq:groupRLF} that it is essentially invertible, so that keeping into account the bounded compression property \eqref{eq:boundcompr} we see by  Theorem \ref{thm:diff_LL} that the differential $\d F_t^s$ of $F_t^s$ is a well defined map from $L^0(T\X)$ into itself. 

For us, the following estimate is of crucial importance:
\begin{proposition}
\label{coro:est_diff_flowmaps} With the same assumptions and notation as in  Proposition \ref{prop:eliadanieleregularity}, we have
\begin{equation}
\label{eq:estimate_differential_eliadaniele}
|\d F_{t}^s|\circ F_{ t'}^t \le \exp\Big(\int_{t\wedge s}^{t\vee s} g_r\circ F_{ t'}^{r}\, \d r\Big),\qquad\mm-a.e.\quad\forall  t', t,s\in[0,1].
\end{equation}
\end{proposition}
\begin{proof} We have
\[
\begin{split}
|\d F_{t}^s| (F_{ t'}^t(x)) \stackrel{\eqref{eq:pn_diff_LL}}\le (\aplip F_t^s)( F_{t'}^t(x))\stackrel{\eqref{eq:estimate3_factor}}\le  e^{ \int_{t\wedge s}^{t\vee s} g_r\circ F_t^{r}\, \d r} ( F_{t'}^t(x))\stackrel{\eqref{eq:groupRLF}}= e^{ \int_{t\wedge s}^{t\vee s} g_r(F_{t'}^{r}(x))\, \d r} 
\end{split}
\]for $\mm$-a.e.\ $x\in\X$, having used the fact that $(F_{t'}^t)_*\mm\ll\mm$ to justify the  post-composition with $F_{t'}^t$ in the above.
\end{proof}

\subsection{Weighted Haj{\l}asz--Sobolev space}
\label{sec:weight_hajlasz}
In \cite{H96}, a notion of Sobolev function defined over a metric measure space (in fact the first one) was studied. We study here, with a similar approach, a space of functions which satisfy a weaker condition.
\begin{definition}[The space \(H_{\phi,R}(\X)\)]\label{def:hsob}
Let \((\X,\sfd,\mm)\) be a metric measure space, $R>0$ and $\phi\in L^0(\X,\mm)$ non-negative be fixed. Put for brevity $F_{\phi}(x,y)\coloneqq e^{\phi(x)+\phi(y)}$.
Given any \(f\in L^2(\X)\), we say that a function
\(G\in L^2(\X)\), $G\geq 0$ is \emph{admissible for \(f\)} provided
there exists a \(\mm\)-negligible Borel set \(N\subseteq \X \)
such that
\begin{equation}\label{eq:def_H_phi}
\big|f(x)-f(y)\big|\leq F_{\phi}(x,y)\big(G(x)+G(y)\big)\sfd(x,y)
\quad\text{ for every }x,y\in \X \setminus N,\ \sfd(x,y)\leq R.
\end{equation}
We call \(A_{\phi,R}(f)\) the family of all admissible functions
for \(f\). Then the space \(H_{\phi,R}(\X)\) is given by
\[
H_{\phi,R}(\X)\coloneqq\big\{f\in L^2(\X)\;\big|\;A_{\phi,R}(f)\neq\emptyset\big\}.
\]
We say that \(H_{\phi,R}(\X)\) is the \emph{\(\phi\)-weighted Haj{\l}asz--Sobolev space} on \(\X\) at scale $R$ and  we endow it with the norm
\begin{equation}
\label{eq:normhphi}
\|f\|_{H_{\phi,R}(\X)}\coloneqq\sqrt{\|f\|_{L^2}^2
+\inf_{G\in A_{\phi,R}(f)}\|G\|_{L^2}^2}\quad\text{ for every }f\in H_{\phi,R}(\X).
\end{equation}
\end{definition}
Given any function \(f\in H_{\phi,R}(\X)\), it is easy to see that   \(A_{\phi,R}(f)\) is convex
and closed in \(L^2(\X)\), thus it admits a unique element of minimal
\(L^2\)-norm, that we call the optimal function for $f$.
\begin{proposition}\label{prop:basehphi}
$(H_{\phi,R}(\X),\|\cdot \|_{H_{\phi,R}(\X)})$ is a Banach space. Moreover,  if $f_n\weakto f$, $G_n\weakto G$ in $L^2(\X)$ and $G_n\in A_{\phi,R}(f_n)$ for every $n\in\N$, then $G\in A_{\phi,R}(f)$.
Finally, we have
\begin{equation}
\label{eq:hl2lsc}
f_n\stackrel{L^2}\weakto f\qquad\Rightarrow\qquad \|f\|_{H_{\phi,R}}\leq \limi_{n\to\infty}\|f_n\|_{H_{\phi,R}}
\end{equation}
(where as customary $\|f\|_{H_{\phi,R}}$ is set to be $+\infty$ if $f\notin {H_{\phi,R}}(\X)$).
\end{proposition}
\begin{proof}
The triangle inequality follows from the implication `$F\in A_{\phi,R}(f)$, $G\in A_{\phi,R}(g)$ imply $F+G\in A_{\phi,R}(f+g)$', which is easy to prove. The other properties of the norm are trivial.

We turn to the second statement and start noticing that by Mazur's lemma we can find, for every $n\in\N$, non-negative coefficients  $\{ \alpha^n_k \}_{k =n}^{N(n)}$ with $\sum_{k=n}^{N(n)} \alpha^n_k = 1$ such that \(\tilde{f}_n:=\sum_{k=n}^{N(n)} \alpha^n_k f_{k}\to f\) and \( \tilde{G}_n:=\sum_{k=n}^{N(n)} \alpha^n_k G_{k} \to G\)
strongly in \(L^2(\X)\). It is clear from our first claim that 
\(\sum_{k=n}^{N(n)} \alpha^n_k G_{k}\in
A_{\phi,R}\big(\sum_{k=n}^{N(n)} \alpha^n_k f_{k}\big)\) for all \(n\in\N\). Possibly taking a further subsequence,
we  also have that \(\tilde f_n\to f\) and \(\tilde G_n\to G\) pointwise
\(\mm\)-a.e.\ respectively, whence by letting \(n\to\infty\) for every $x,y$ outside of a $\mm$-null set in
the inequality \(\big|\tilde f_n(x)-\tilde f_n(y)\big|\leq
F_{\phi}(x,y)\big(\tilde G_n(x)+\tilde G_n(y)\big)\sfd(x,y)\), we get
\(G\in A_{\phi,R}(f)\), as claimed.

Now the $L^2$-lower semicontinuity of the $H_{\phi,R}$-norm stated in \eqref{eq:hl2lsc} is clear: we can assume that the $\limi$ is a finite limit, then we pick $G_n\in A_{\phi,R}(f_n)$ such that $\|f_n\|_{L^2}^2+\|G_n\|^2_{L^2}\leq \|f_n\|_{H_{\phi,R}}^2+\frac1n$ and observe that up to passing to a subsequence, we have $G_n\weakto G$ in $L^2$ for some $G\in L^2$. Then \eqref{eq:hl2lsc} follows by what already proved.

Finally, the completeness of $H_{\phi,R}(\X)$ is now a standard consequence of \eqref{eq:hl2lsc}: let $(f_n)$ be $H_{\phi,R}(\X)$-Cauchy and $f$ its $L^2$-limit (which exists because the $H_{\phi,R}$-norm is bigger than the $L^2$-norm and $L^2(\mm)$ is complete). Then we have
\[
\lims_{n\to\infty}\|f-f_n\|_{H_{\phi,R}(\X)}\stackrel{\eqref{eq:hl2lsc}}\leq \lims_{n\to\infty}\limi_{m\to\infty}\|f_m-f_n\|_{H_{\phi,R}(\X)}=0,
\]
where in the last step we used the fact that  $(f_n)$ is $H_{\phi,R}(\X)$-Cauchy.
\end{proof}
The following is easily verified:
\begin{proposition}\label{prop:HLL}
Every \(f \in H_{\phi,R}(\X)\) has the Lusin--Lipschitz property and 
\begin{equation}
\label{eq:bounddH}
|\d f|\leq 2e^{2\phi} G\qquad\mm-a.e..
\end{equation}
\end{proposition}
\begin{proof}
Let \(f\in H_{\phi,R}(\X)\) and \(G\in A_{\phi,R}(f)\) be fixed. Pick any \(\mm\)-null set
\(N\) satisfying \eqref{eq:def_H_phi}. Given any \(a,b\in \Q\cap(0,\infty)\), we define
\(E_{a,b}:=\{\phi\leq a\}\cap \{G\leq  b\}\setminus N \). Then  from \eqref{eq:def_H_phi} we see that 
\[
\big|f(x)-f(y)\big|\leq 2b e^{2a}\sfd(x,y)\quad\text{ for every }x,y\in E_{a,b}\text{ with }\sfd(x,y)\leq R,
\]
proving that \(f\restr{E_{a,b}}\) 
is locally Lipschitz and that $|\d f|\leq 2be^{2a}$ $\mm$-a.e.\ on $E_{a,b}$. The conclusion follows by the arbitrariness of $a,b$.
\end{proof}
\begin{remark}[Weighted normed modules]{\rm
Fix a Radon measure \(\mu\) on \((\X,\sfd)\) such that \(\mu\ll\mm\). Denote by
\(\pi_\mu\colon L^0(\mm)\to L^0(\mu)\) the the canonical projection map sending the $\mm$-a.e.\ equivalence class of a
Borel function to its $\mu$-a.e.\ equivalence class. Given a \(L^0(\mm)\)-normed \(L^0(\mm)\)-module \(\mathscr M^0\), we define
\begin{equation}\label{eq:def_M0_mu}
\mathscr M^0_\mu\coloneqq\mathscr M^0/\sim_\mu,
\quad\text{ where }v\sim_\mu w\text{ if and only if }\pi_\mu(|v-w|)=0
\text{ holds }\mu\text{-a.e.\ on }\X.
\end{equation}
The resulting set $\mathscr M^0_\mu$ can be endowed with a natural structure of $L^0(\mu)$-normed $L^0(\mu)$-module.

Moreover, given a $L^2(\mm)$-normed $L^\infty(\mm)$-module $\mathscr M$, we define
\[
\mathscr M_\mu:=\big\{v\in\mathscr M^0_\mu\,:\,|v|\in L^2(\mu)\big\},
\]
where \(\mathscr M^0\) stands for the $L^0(\mm)$-completion of $\mathscr M$. The space $\mathscr M_\mu$ inherits a natural
structure of $L^2(\mu)$-normed $L^\infty(\mu)$-module. One can readily check that
\begin{equation}\label{eq:Hilb_pass_to_quotient}
\mathscr M\text{ is Hilbert}\qquad\Longrightarrow\qquad\mathscr M_\mu\text{ is Hilbert.}
\end{equation}
When $\mathscr M$ is the cotangent module $L^2(T^*\X)$, we write $L^2(T^*\X,\mu)$ in place of $L^2(T^*\X)_\mu$.
\fr}\end{remark}
The following technical proposition will be of crucial importance in the study of the regularity properties of $s\mapsto f\circ F_t^s$ that will be performed in Section \ref{sec:reg_fcircF_t}:
\begin{proposition}[Closure of \(\d\) on bounded subsets of $H_{\phi,R}(\X)$]\label{prop:closure_tilde_d}
Let \((\X,\sfd,\mm)\) be an infinitesimally Hilbertian metric measure space,
\(\phi\in L^0(\X)\) non-negative, and $R>0$. Let \((f_n)_n\subset H_{\phi,R}(\X)\) be a bounded sequence such that $f_n\to f$
in \(L^2(\X)\) and \(\d f_n\to\omega\) in \(L^0(T^*\X)\), for some \(f\in L^2(\X)\) and \(\omega\in L^0(T^*\X)\) respectively. 

Then  \(f\in H_{\phi,R}(\X)\)  and \(\omega=\d f\).
\end{proposition}
\begin{proof} The fact that  \(f\in H_{\phi,R}(\X)\)  (with  \(\|f\|_{H_{\phi,R}(\X)}\leq\limi_n\|f_n\|_{H_{\phi,R}(\X)}\))  follows from Proposition \ref{prop:basehphi}. To prove that $\omega=\d f$, we start picking \((G_n)_n\) in \(L^2(\X)\) such
that \(G_n\in A_{\phi,R}(f_n)\) and \(\|f_n\|_{L^2(\X)}^2+\|G_n\|_{L^2(\X)}^2\leq\|f_n\|_{H_{\phi,R}(\X)}^2+1/n\) for every \(n\in\N\).  Then up to a not
relabeled subsequence, we have that \(G_n\rightharpoonup G\) weakly
in \(L^2(\mm)\), for some \(G\in L^2(\X)\). We apply Mazur's lemma once again to find, for every $n\in\N$, non-negative coefficients  $\{ \alpha^n_k \}_{k =n}^{N(n)}$ with $\sum_{k=n}^{N(n)} \alpha^n_k = 1$ such that \( \tilde{G}_n:=\sum_{k=n}^{N(n)} \alpha^n_k G_{k} \to G\)
strongly in \(L^2(\X)\). Putting 
\(\tilde{f}_n:=\sum_{k=n}^{N(n)} \alpha^n_k f_{k}\) it is clear that $\tilde f_n\to f$ in \(L^2(\X)\) and
\(\tilde G_n\in
A_{\phi,R}(\tilde f_n)\) for all \(n\in\N\). We claim that   \(\d\tilde f_n\weakto\omega\) in
\(L^2(T^*\X,\tilde\mm)\),
where $\tilde\mm:=e^{-4\phi}\mm$. To see this observe that \eqref{eq:bounddH} tells that 
\[
\|\d f\|_{L^2(T^*\X,\tilde \mm)}\leq  2\|G\|_{L^2(\X,\mm)}\qquad \text{ if }G\in A_{\phi,R}(f).
\]
Thus from our assumptions it follows that $(\d f_n)$ is a bounded sequence in $L^2(T^*\X,\tilde \mm)$, hence up to a non-relabeled subsequence it converges weakly in such space to some $\tilde\omega$;
here, we are using the fact that $L^2(T^*\X,\tilde\mm)$ is Hilbert and thus reflexive, cf.\ \eqref{eq:Hilb_pass_to_quotient}.

Since we also have $\d f_n\to \omega$ in $L^0(T^*\X)$ it is clear that $\tilde\omega=\omega$, showing in particular that the weak limit $\tilde\omega$ does not depend on the subsequence chosen. To get the claim notice that  since the sequence $(\d\tilde f_n)$ is made of convex combinations of the $\d f_n$'s, we also have that 
\begin{equation}
\label{eq:weakdconv}
\d\tilde f_n\weakto \omega\qquad\text{ in }L^2(T^*\X,\tilde \mm). 
\end{equation}
Possibly taking a further subsequence, we also have that \(\sum_{n=1}^\infty\|\tilde G_{n+1}-\tilde G_n\|_{L^2(\X)}<\infty\), whence \(H\coloneqq \tilde G_1+\sum_{n=1}^\infty|\tilde G_{n+1}-\tilde G_n|\) belongs
to \(L^2(\X,\mm)\). Since clearly $\tilde G_n\leq H$ $\mm$-a.e.\ for any \(n\in\N\), we deduce that  \(H\in A_{\phi,R}(\tilde f_n)\) for every \(n\in\N\) and thus we can find a \(\mm\)-null Borel set \(N\subset \X\)
such that
\begin{equation}\label{eq:closure_tilde_d_aux}
\big|\tilde f_n(x)-\tilde f_n(y)\big|\leq F_{\phi}(x,y)\big(H(x)+H(y)\big)\sfd(x,y)
\quad\text{ for all }n\in\N\text{ and }x,y\in \X\setminus N,\ \text{with }\sfd(x,y)\leq R.
\end{equation}
Let $(x_j)\subset\X$ be countable and dense and for $j,k\in\N$ let
\[
E_{j,k}:=\big(B_{R/2}(x_j)\cap\{H\leq k\}\cap\{\phi\leq k\}\big)\setminus N.
\]
Fix $j,k\in\N$ and notice that the bound   \eqref{eq:closure_tilde_d_aux} ensures that the functions $\tilde f_n$ are uniformly  Lipschitz on the bounded set $E_{j,k}$. Therefore, we can find a sequence $(g^{j,k}_n)\subset {\rm LIP}(\X)$ made of functions with uniformly bounded support such that $g^{j,k}_n=\tilde f_n\restr{E_{j,k}}$ for every \(n\in\N\), and $\sup_n{\rm Lip}(g^{j,k}_n)<+\infty$. This grants that
\((g^{j,k}_n)\) is bounded in \(W^{1,2}(\X,\sfd,\mm)\), so
that (up to a not relabeled subsequence) by the continuity of $\d:W^{1,2}(\X)\to L^2(T^*\X,\mm) \hookrightarrow L^2(T^*\X,\tilde\mm)$  we have
\( g^{j,k}_n\weakto  g^{j,k}\) weakly in \(L^2(\X,\mm)\) and
\(\d g^{j,k}_n\weakto \d g^{j,k}\) weakly in \(L^2(T^*\X,\tilde\mm)\),
for some \(g^{j,k}\in W^{1,2}(\X,\sfd,\mm)\). In particular, $\nchi_{E_{j,k}}\d g^{j,k}_n\weakto \nchi_{E_{j,k}}\d g^{j,k}$ weakly in \(L^2(T^*\X,\tilde\mm)\) and since the construction ensures   that
\(g^{j,k}= f\) \(\mm\)-a.e.\ on \(E_{j,k}\), we also know that
\(\d g^{j,k}=\d  f\) on $E_{j,k}$ and, similarly, that   \(\d g^{j,k}_n=\d \tilde f_n\) on $E_{j,k}$  for
every \(n\in\N\). We thus proved that $\nchi_{E_{j,k}}\d \tilde f_n\weakto \nchi_{E_{j,k}}\d  f$ weakly in \(L^2(T^*\X,\tilde\mm)\), which coupled with \eqref{eq:weakdconv} implies $\nchi_{E_{j,k}}\d  f=\nchi_{E_{j,k}}\omega$. Since the sets $E_{j,k}$ cover $\mm$-a.a.\ $\X$, by the arbitrariness of $j,k$ this is sufficient to conclude that $\d f=\omega$, as desired.
\end{proof}

\subsection{Integration of module-valued maps and related topics}\label{se:intmod}
The goal of this section is to study integration (and differentiation) of maps with values in a Hilbert module $\H$. We shall mostly apply this theory to the case $\H=L^0(T\X)$ in order to study the module $W^{1,2}_{fix}([0,1],L^0(T\X))$ (see Definition \ref{def:fixedspace}). From the conceptual point of view, the most important result here is perhaps the Hille-like Theorem \ref{thm:HilleH} below, that we will use in conjunction with the closure result  for the differential of functions in $H_{\phi,R}(\X)$ in Proposition \ref{prop:closure_tilde_d}.

We shall work with maps on $[0,1]$ with values in Hilbert modules, but several parts of the discussion below can be adapted to more general settings.

\bigskip

Before coming to general module-valued maps, we consider the case of maps taking values into $L^0(\X)$. Recall that the topology of $L^0(\X)$ is metrized by the complete and separable distance
\[
\sfd_{L^0}(f,g):=\int 1\wedge |f-g|\,\d \mm',
\]
where $\mm'\in \mathscr P(\X)$ is any Borel probability measure having the same negligible sets of $\mm$. Let us fix such $\mm'$, and thus the distance $\sfd_{L^0}$: the actual choice of $\mm'$ does not matter, but in establishing some inequalities it is useful to have it fixed (and for convenience we shall add some further requirement to $\mm'$ in Section \ref{se:setting}).

\begin{definition}[Some spaces of functions]\label{def:spacesf}
We shall consider:
\begin{itemize}
\item[i)] For $p\in[1,\infty]$ the space $L^p([0,1],L^0(\X))\subset L^0([0,1],L^0(\X))$ is the collection of functions $(f_t)$ such that for $\mm$-a.e.\ $x$ the function $t\mapsto f_t(x)$ is in $L^p(0,1)$.
\item[ii)] The space $W^{1,2}([0,1],L^0(\X))\subset L^0([0,1],L^0(\X))$ is the collection of  functions $(f_t)$ such that for $\mm$-a.e.\ $x$ the function $t\mapsto f_t(x)$ is in $W^{1,2}(0,1)$.
\item[iii)] The space $AC^2([0,1],L^0(\X))\subset C([0,1],L^0(\X))$ is the collection of  functions $(f_t)$ such that for $\mm$-a.e.\ $x$ the function $t\mapsto f_t(x)$ is in $W^{1,2}(0,1)$. 
\end{itemize}
\end{definition}
\begin{remark}[Comments on the notation]{\rm
The notation $L^0(\X,L^p([0,1]))$ and $L^0(\X,W^{1,2}([0,1]))$ would be more in line, as opposed to  $L^p([0,1],L^0(\X))$ and $W^{1,2}([0,1],L^0(\X))$, with the standard notation for Banach-valued maps: our choice is motivated by convenience in dealing with module-valued curves, where we will speak of $L^p([0,1],\H)$ and $W^{1,2}([0,1],\H)$.

Also, notice that the way we defined it makes $AC^2([0,1],L^0(\X))$ different from the usual space of absolutely continuous curves with values in the metric space $L^0(\X)$ (and the same holds for the space $AC^2([0,1],\H)$ defined below). 

Finally, let us stress that by $C([0,1],L^0(\X))$ (and $C([0,1],\H)$ below) we intend the standard space of continuous curves from $[0,1]$ to $L^0(\X)$ equipped with the usual `sup' distance. In particular, $C([0,1],L^0(\X))$ has nothing to do with Borel collections of continuous functions $t\mapsto f_t(x)$ and is not contained in $L^\infty([0,1],L^0(\X))$.
}\fr\end{remark}
For $(f_t)\in  L^0([0,1],L^0(\X))$ and $p\in[1,\infty]$ we define (up to equality for $\mm$-a.e.\ $x$) the map
\[
|(f_t)|_{L^p}(x):=\|f_\cdot(x)\|_{L^p(0,1)}
\]
and similarly
\[
|(f_t)|_{W^{1,2}}(x):=\|f_\cdot(x)\|_{W^{1,2}(0,1)}.
\]
Then clearly $L^p([0,1],L^0(\X))$ (resp.\ $W^{1,2}([0,1],L^0(\X))$) is the subspace of $L^0([0,1],L^0(\X))$ made of those functions for which $|(f_t)|_{L^p}$ (resp.\ $|(f_t)|_{W^{1,2}}$) is finite $\mm$-a.e.. In particular, the distances
\[
\sfd_{L^p}\big((f_t),(g_t)\big):=\sfd_{L^0}\big(|(f_t-g_t)|_{L^p},0\big)\qquad\text{and}\qquad\sfd_{W^{1,2}}\big((f_t),(g_t)\big):=\sfd_{L^0}\big(|(f_t-g_t)|_{W^{1,2}},0\big)
\]
are well defined on $L^p([0,1],L^0(\X))$ and  $W^{1,2}([0,1],L^0(\X))$ respectively. It is then easy to see that $L^p([0,1],L^0(\X))$ and  $W^{1,2}([0,1],L^0(\X))$ are $L^0(\X)$-normed modules when equipped with the above pointwise norms and with the product $g(t\mapsto f_t):=(t\mapsto gf_t)$. Here the only possibly non-trivial claim is  completeness: this follows from the completeness of $L^0([0,1],L^0(\X))\sim L^0([0,1]\times \X)$ and the lower semicontinuity of $L^p/W^{1,2}$-norms w.r.t.\ convergence a.e.. Indeed, the inequality
\[
\iint_0^1 |f_t-g_t|(x)\wedge 1\,\d t\,\d\mm(x)=\int 1\wedge\big(\int_0^11\wedge|f_t-g_t|(x)\,\d t\big)\,\d\mm(x)\leq\int 1\wedge |(f_t-g_t)|_{L^p}\,\d\mm
\]
shows that if $(f^n_t)\subset L^p([0,1],L^0(\X))$ is $\sfd_{L^p}$-Cauchy, then it is also Cauchy in $L^0([0,1]\times\X)$ and thus converges to some $(f_t)$ in such space. Thus some subsequence $(f^{n_k}_t)$ converges $(\mm\times\mathcal L^1)$-a.e., and thus for $\mm$-a.e.\ $x\in\X$ we have that $f^{n_k}_t(x)\to f_t(x)$ for a.e.\ $t\in[0,1]$. Then Fatou's theorem implies that $|(f^{m}-f)|_{L^p}\leq\limi_k|(f^{m}-f^{n_k})|_{L^p}$ $\mm$-a.e.\ and thus
\[
\sfd_{L^p}\big((f_t),(f^m_t)\big)=\int1\wedge|(f_t-f^m_t)|_{L^p}\,\d\mm\leq \limi_{k\to\infty}\int1\wedge|(f^{n_k}_t-f^m_t)|_{L^p}\,\d\mm=\limi_{k\to\infty}\sfd_{L^p}\big((f^{n_k}_t),(f^m_t)\big),
\]
so that completeness follows letting $m\to\infty$ and recalling that $(f^n_t)$ is $\sfd_{L^p}$-Cauchy. The argument for $W^{1,2}([0,1],L^0(\X))$ is analogous.

The space $AC^2([0,1],L^0(\X))$ is complete w.r.t.\ the distance
\[
\sfd_{AC^2}\big((f_t),(g_t)\big):=\sfd_{W^{1,2}}\big((f_t),(g_t)\big)+\sup_{t\in[0,1]}\sfd_{L^0}(f_t,g_t),
\]
as it is trivial to check. It is an algebraic module over $L^0(\X)$ with respect to the operation $g(t \mapsto f_t) = t \mapsto g f_t$, but it does not have the structure of a normed $L^0(\X)$ module. The same holds for $C([0,1],L^0(\X))$.
\bigskip

For $(f_t)\in L^1([0,1],L^0(\X))$ and $A\subset [0,1]$ Borel, Fubini's theorem ensures that the function
\[
\big(\int_Af_t\,\d t\big)(x):=\int_Af_t(x)\,\d t
\]
is a well-defined element of $L^0(\X)$ and it is clear that
\[
\big|\int_Af_t\,\d t\big|\leq\int_A|f_t|\,\d t\qquad\mm-a.e..
\]
In particular, for any $(f_t)\in L^1([0,1],L^0(\X))$ and $t,s\in[0,1]$, $t<s$ we have
\[
\big|\int_0^s f_r\,\d r-\int_0^t f_r\,\d r\big |\leq \int_t^s|f_r|\,\d r,\qquad\mm-a.e.,
\]
showing that $t\mapsto\int_0^t f_r\,\d r$ is continuous w.r.t.\ $\mm$-a.e.\ convergence and thus also w.r.t.\ $L^0(\X)$-convergence. 

Now notice that since the map assigning to a function in the classical space $W^{1,2}(0,1)$ its distributional derivative in $L^2([0,1])$ is continuous, we have that for  $(f_t)\in W^{1,2}([0,1],L^0(\X))$ the pointwisely defined distributional derivative, that we shall denote by $(\dot f_t)$, is an element of $L^2([0,1],L^0(\X))$. It is also clear by comparison with the classical case  that
\begin{equation}
\label{eq:charw12}
\begin{split}
&(f_t)\in W^{1,2}([0,1],L^0(\X))\quad\Leftrightarrow\quad\exists (g_t)\in L^2([0,1],L^0(\X))\text{ such that }\forall h\in(0,1)\text{ we have}\\
&f_{t+h}-f_t=\int_t^{t+h} g_r\,\d r\quad for\ a.e.\  t\in[0,1-h]\text{ and in this case }\dot f_t=g_t,\quad a.e.\ t,
\end{split}
\end{equation}
where the identity between functions are intended $\mm$-a.e.. Similarly, the continuity in $t,s$ of $\int_t^s g_r\,\d r$  gives
\begin{equation}
\label{eq:charAC}
\begin{split}
&(f_t)\in AC^2([0,1],L^0(\X))\quad\Leftrightarrow\quad\exists (g_t)\in L^2([0,1],L^0(\X))\text{ such that }\\
&f_s-f_t=\int_t^s g_r\,\d r\qquad\forall t,s\in[0,1],\ t<s\quad \text{and in this case }\dot f_t=g_t,\quad a.e.\ t.
\end{split}
\end{equation}
Also, still by looking at the classical one dimensional case, we have the following characterization of functions in $W^{1,2}([0,1],L^0(\X))$:
\begin{equation}
\label{eq:charw12bound}
\begin{split}
&(f_t)\in W^{1,2}([0,1],L^0(\X))\quad\Leftrightarrow\quad\exists (g_t)\in L^2([0,1],L^0(\X))\text{ such that}\\
&|f_{s}-f_t|\leq \int_t^{s} g_r\,\d r\quad for\ a.e.\  t,s\in[0,1],\ t<s\ \text{ and in this case }|\dot f_t|\leq g_t,\quad a.e.\ t,
\end{split}
\end{equation}
where again the inequalities between functions are intended $\mm$-a.e.. We also notice  the existence of a unique \emph{continuous representative} of elements in $W^{1,2}([0,1],L^0(\X))$:
\begin{equation}\label{eq:ex_cont_repr}
\begin{split}
&\text{for any }(f_t)\in W^{1,2}([0,1],L^0(\X))\text{ there is a unique }(\bar f_t)\in AC^2([0,1],L^0(\X))\\
&\text{such that }f_t=\bar f_t\text{ for a.e. }t,\ \mm-a.e..
\end{split}
\end{equation}
Indeed, uniqueness is clear. For existence, we simply pick for $\mm$-a.e.\ $x$ the continuous representative $\bar f_\cdot(x)$ of  $f_\cdot(x)\in W^{1,2}(0,1)$: the fact that $x\mapsto\bar f_t(x)$ is Borel can be proved by building upon the fact that the map from $W^{1,2}(0,1)$ to $C([0,1])$ sending a Sobolev function to its continuous representative is continuous. Then it is clear that $t\mapsto \bar f_t\in L^0(\X)$ is continuous w.r.t.\ a.e.\ convergence, and thus w.r.t.\ the $L^0$-topology. 

The existence of such representatives that are absolutely continuous for $\mm$-a.e.\ $x$ can also be used to prove that
\begin{equation}
\label{eq:charder}
(f_t)\in AC^2([0,1],L^0(\X))\quad\Rightarrow\quad\lim_{h\to 0}\frac{f_{t+h}-f_t}h=\dot f_t,\qquad in\ L^0(\X),\ for\ a.e.\ t\in[0,1].
\end{equation}
Indeed, for $\bar f_t$ as in the proof of \eqref{eq:ex_cont_repr}, the differentiability of functions in $AC^2([0,1],\R)$ and Fubini's theorem ensure that $\frac{\bar f_{t+h}-\bar f_t}h\to \dot f_t$ $\mm$-a.e.\ for a.e.\ $t$. Also, again Fubini's theorem grants that $\bar f_t=f_t$ $\mm$-a.e.\ for a.e.\ $t$, and since $(f_t),(\bar f_t)$ are both in $C([0,1],L^0(\X))$ (the first by assumption, the second because by construction it is continuous w.r.t.\ $\mm$-a.e.\ convergence), we deduce that  $\bar f_t=f_t$ $\mm$-a.e.\ for every $t\in[0,1]$, thus \eqref{eq:charder} follows.

We conclude the discussion about the space $W^{1,2}([0,1],L^0(\X))$ with two simple results: the first concerns stability property and the second is a sort of density criterion.
\begin{proposition}\label{prop:stabl0w12}
Let $(f^n_t)\in W^{1,2}([0,1],L^0(\X))$, $n\in\N$, be such that $(f^n_t)\to (f_t)$ and $(\dot f_t^n)\to (g_t)$ in $L^0([0,1],L^0(\X))$ for some $(f_t),(g_t)\in L^0([0,1],L^0(\X))$. Assume that
\begin{equation}
\label{eq:bw1201}
\lim_{C\to+\infty}m(C)=0,\qquad\text{ where }\qquad m(C):=\sup_{n\in\N}\mm'\big(\{|(f^n_t)|_{W^{1,2}}\geq C\}\big).
\end{equation}
Then $(f_t)\in W^{1,2}([0,1],L^0(\X))$ and $\dot f_t=g_t$ $\mm$-a.e.\ for a.e.\ $t$. 

Notice that in particular, condition \eqref{eq:bw1201} holds provided  $ |(f^n_t)|_{W^{1,2}}\leq g_n$ $\mm$-a.e.\ for some sequence $(g_n)$ having  a limit in $L^0(\X)$.
\end{proposition}
\begin{proof} Up to pass to a non-relabeled subsequence we have that for $\mm$-a.e.\ $x\in\X$ the functions $t\mapsto f_t^n(x), \dot  f_t^n(x)$ converge to $t\mapsto f_t(x),  g_t(x)$ for a.e.\ $t\in[0,1]$. By standard, and easy, results about Sobolev functions on $(0,1)$, to conclude that $t\mapsto f_t(x)$ belongs to $W^{1,2}(0,1)$ with derivative $t\mapsto g_t(x)$ it is sufficient to prove that $\limi_{n\to\infty}|(f_t^n)|_{W^{1,2}}(x)<\infty$. The fact that this holds for $\mm$-a.e.\ $x$ is a consequence of Borel--Cantelli's lemma and the assumption \eqref{eq:bw1201}. Indeed $\limi_{n\to\infty}|(f_t^n)|_{W^{1,2}}(x)<C$ if and only if $x\in \cap_{n\in\N}\cup_{i\geq n}\{|(f^i_t)|_{W^{1,2}}< C\}$ and since the sequence of sets $n\mapsto \cup_{i\geq n}\{|(f^n_t)|_{W^{1,2}}< C\}$ is decreasing we have
\[
\begin{split}
\mm'\Big(\bigcap_{n\in\N}\bigcup_{i\geq n}\big\{|(f^i_t)|_{W^{1,2}}< C\big\}\Big)&=\inf_{n\in\N}\mm'\Big(\bigcup_{i\geq n}\big\{|(f^i_t)|_{W^{1,2}}< C\big\}\Big)\\
&\geq \inf_{n\in\N}\mm'\big(\big\{|(f^n_t)|_{W^{1,2}}< C\big\}\big)\geq 1-m(C).
\end{split}
\]
Thus \eqref{eq:bw1201} ensures a.e.\ finiteness of the $\limi$, as claimed. 

For the last statement observe that condition \eqref{eq:bw1201} is satisfied by the sequence $(g_n)$ (rather trivially by the definition of local convergence in measure and/or  of distance $\sfd_{L^0}$).
\end{proof}
\begin{proposition}\label{prop:Afunc}
Let $\mathcal A\subset W^{1,2}([0,1],L^0(\X))$. Assume that $\mathcal A$:
\begin{itemize}
\item[o)] is a vector space,
\item[i)] is stable by the `restriction' operation, i.e.\  $(f_t)\in \mathcal A$ and $E\subset\X$ Borel  imply that $t\mapsto\nchi_Ef_t$ is in $\mathcal A$,
\item[ii)] is stable by multiplication by a $C^1$ function, i.e.\ $(f_t)\in \mathcal A$ and $\varphi\in C^1([0,1])$ imply that $t\mapsto\varphi(t)f_t$ is in $\mathcal A$,
\item[iii)] is closed in  the $W^{1,2}([0,1],L^0(\X))$-topology,
\item[iv)] contains the constant functions, i.e.\ for any $f\in L^0(\X)$ the map $t\mapsto f$ is in $\mathcal A$.
\end{itemize}
Then $\mathcal A=W^{1,2}([0,1],L^0(\X))$.
\end{proposition}
\begin{proof}
Let $(\varphi_n)\subset C^1([0,1])$ be countable and dense in $W^{1,2}(0,1)$. Fix $\eps>0$ and let $T:W^{1,2}(0,1)\to W^{1,2}(0,1)$ be the map sending $f$ to $\varphi_n$, where $n\in\N$ is the least index $j\in\N$ such that $\|f-\varphi_j\|_{W^{1,2}}\leq\eps$. It is clear that $T$ is well defined and Borel. Thus for  $(f_t)\in W^{1,2}([0,1],L^0(\X))$ the curve $t\mapsto T(f)_t$ defined by $T(f)_t(x):=T(f_\cdot(x))(t)$ is in $W^{1,2}([0,1],L^0(\X))$ and, by construction, its $\sfd_{W^{1,2}}$-distance from $(f_t)$ is $\leq \eps$. As $\eps>0$ is arbitrary, by the closure property $(iii)$ to conclude it is sufficient to show that $(T(f)_t)\in \mathcal A$. To see this, for every $n\in\N$ let $E_n\subset\X$ be the set of $x$'s such that $T(f_\cdot)(x)=\varphi_n$. Then we have $T(f)_t=\sum_n\nchi_{E_n}\varphi_n(t)$ (where the convergence of the partial sums is intended in the $W^{1,2}([0,1],L^0(\X))$-topology) and the properties $(o),(i),(ii),(iv)$ trivially ensure that $(T(f)_t)\in\mathcal A $, as desired.
\end{proof}

We pass to the vector case. Fix a separable, Hilbertian $L^0$-module $\H$. Recall that the space $L^0([0,1],\H)$, that we shall sometimes abbreviate in $L^0_\H$, is the space of Borel maps from $[0,1]$ to $\H$ identified up to equality for a.e.\ $t$. The space $L^0_\H$ is complete w.r.t.\ the distance 
\[
\sfd_{L^0_\H}\big((v_t),(z_t)\big):=\int_0^11\wedge\sfd_{L^0}(|v_t-z_t|,0)\,\d t.
\]
The space $C([0,1],\H)$ denotes, as usual, the space of continuous curves with values in $\H$ equipped with the `sup' distance.

We pass to the `vector versions' of the spaces introduced in Definition \ref{def:spacesf}:
\begin{definition}[Some spaces of vectors]\label{def:spacesv} We shall denote by:
\begin{itemize}
\item[i)] For $p\in[1,\infty]$ the space $L^p([0,1],\H)\subset L^0([0,1],\H)$, that we shall abbreviate in $L^p_\H$,  is the collection of vector fields $(v_t)$ such that the quantity 
\[
|(v_t)|_{L^p_\H}(x):=|(|v_t|)|_{L^p}(x)=\||v_\cdot|(x)\|_{L^p(0,1)}
\]
(which is well defined up to equality $\mm$-a.e.) is finite $\mm$-a.e..
\item[ii)] The space $W^{1,2}([0,1],\H)\subset L^0([0,1],\H)$ is the collection of  vector fields $(v_t)$ for which there is $(\dot v_t)\in L^2([0,1],\H)$ such that for any $z\in \H$ the curve $t\mapsto \la v_t,z\ra$ is in $W^{1,2}([0,1],L^0(\X))$ with 
\[
\partial_t  \la v_t,z\ra=  \la \dot v_t,z\ra\qquad \mm-a.e.,\ a.e.\ t.
\]
For $(v_t)\in W^{1,2}([0,1],\H)$ we define
\[
|(v_t)|_{W^{1,2}_\H}^2:= |(v_t)|_{L^2_\H}^2+ |(\dot v_t)|_{L^2_\H}^2\in L^0(\X).
\]
\item[iii)] The space $AC^2([0,1],\H)\subset C([0,1],\H)$ is the collection of    vector fields $(v_t)$ for which there is $(\dot v_t)\in L^2([0,1],\H)$ such that for any $z\in \H$ the curve $t\mapsto \la v_t,z\ra$ is in $W^{1,2}([0,1],L^0(\X))$ (and thus in $AC^2([0,1],L^0(\X))$)  with 
\[
\partial_t  \la v_t,z\ra=  \la \dot v_t,z\ra\qquad \mm-a.e.,\ a.e.\ t.
\]
\end{itemize}
\end{definition}
The three spaces defined above are naturally endowed with the respective distances
\[
\begin{split}
\sfd_{L^p_\H}\big((v_t),(z_t)\big)&:=\sfd_{L^0}(|(v_t-z_t)|_{L^p_\H},0),\\
\sfd_{W^{1,2}_\H}\big((v_t),(z_t)\big)&:=\sfd_{L^0}(|(v_t-z_t)|_{W^{1,2}_\H},0),\\
\sfd_{AC^2_\H}\big((v_t),(z_t)\big)&:=\sfd_{L^0}(|(v_t-z_t)|_{W^{1,2}_\H},0)+\sup_{t\in[0,1]}\sfd_{\H}(v_t,z_t).
\end{split}
\]
These are complete distances, as can be seen arguing as for the respective spaces of functions (to show that the derivative of the limit is the limit of the derivative in considering the spaces $W^{1,2}_\H$ and $AC^2_\H$ we use Proposition \ref{prop:stabl0w12}).

The spaces $L^p_\H,W^{1,2}_\H,AC^2_\H$ are also endowed with the product with $L^0(\X)$ functions defined as $g(t\mapsto v_t):=(t\mapsto gv_t)$ and it is clear that $L^p_\H,W^{1,2}_\H$ are $L^0$-normed modules.

\bigskip

Let us turn to the definition of integral of a vector field $(v_t)$ in $L^1_\H$: for $A\subset[0,1]$ Borel we want to define $\int_Av_t\,\d t$ as element of $\H$.  
To this aim, notice that for any $z\in\H$ the function $\la z,v_t\ra$ satisfies $|\la z,v_t\ra|\leq |z||v_t|$ $(\mm\times\mathcal L^1)$-a.e., thus  for $\mm$-a.e.\ $x\in\X$ we have that $t\mapsto\la z,v_t\ra(x)$ is in $L^1(0,1)$. Hence $\int_A\la z,v_t\ra\,\d t$ is a well defined function in $L^0(\X)$ and it is  clear that the assignment $z\mapsto \int_A\la z,v_t\ra\,\d t$ is linear and satisfies
\begin{equation}
\label{eq:perintvt}
\big|\int_A\la z,v_t\ra\,\d t\big|\leq \int_A|\la z,v_t\ra |\,\d t\leq |z|\int_A|v_t|\,\d t\leq|z||(v_t)|_{L^1_{\H}}.
\end{equation}
This is sufficient to establish that $\H\ni z\mapsto \int_A\la z,v_t\ra\,\d t\in L^0(\X)$ is $L^0(\X)$-linear and continuous and thus represented by - thanks to Riesz's theorem for modules -  an element of $\H$ that we shall denote by $\int_Av_t\,\d t$.   Notice that the bound \eqref{eq:perintvt} gives
\begin{equation}
\label{eq:normint}
\big|\int_Av_t\,\d t\big|\leq \int_A|v_t|\,\d t,\qquad\mm-a.e..
\end{equation}
This bound is sufficient to prove that
\begin{equation}
\label{eq:contint}
\text{for }(v_t)\in L^1_\H\text{ the map }t\mapsto\int_0^tv_s\,\d s\in \mathscr{H} \text{ is in }C([0,1],\H)
\end{equation}
(because $|\int_0^s v_r\,\d r-\int_0^tv_r\,\d r|\leq \int_t^s|v_r|\,\d r\to0$ $\mm$-a.e.\ as $s\to t)$.

Observe that a direct consequence of the definitions and of \eqref{eq:charw12} is that for $(v_t)\in W^{1,2}_\H$ and $h\in(0,1)$ it holds
\begin{equation}
\label{eq:intvt}
v_{t+h}-v_t=\int_t^{t+h}\dot v_r\,\d r,\qquad\mm-a.e.,\ a.e.\ t\in[0,1-h],
\end{equation}
and thus by \eqref{eq:normint} that $|v_t|\leq |v_s|+\int_0^1|\dot v_r|\,\d r$ is valid $\mm$-a.e.\ for a.e.\ $t,s$. Integrating this in $s$ we deduce that 
\begin{equation}
\label{eq:perlinfty1}
|(v_t)|_{L^\infty_\H}\leq |(v_t)|_{L^2_\H}+|(\dot v_t)|_{L^1_\H}\leq\sqrt 2|(v_t)|_{W^{1,2}_\H}\qquad\mm-a.e.,
\end{equation}
 showing in particular that $W^{1,2}_\H\subset L^\infty_\H$. Notice that by Fubini's theorem, an equivalent way of stating this bound is by saying that for a.e.\ $s\in[0,1]$ we have $|v_s|\leq  2|(v_t)|_{W^{1,2}_\H}$ $\mm$-a.e.. Now observe that if $(v_t)\in AC^2_\H$, from the continuity of the pointwise norm as map from $\H$ to $L^0(\X)$ we see that $t\mapsto |v_t|\in L^0(\X)$ is continuous, and thus
 \begin{equation}
\label{eq:perlinfty2}
(v_t)\in AC^2_\H\qquad\Rightarrow\qquad|v_t|\leq \sqrt 2|(v_s)|_{W^{1,2}_\H}\qquad\mm-a.e.,\ \forall t\in[0,1].
\end{equation}
Another direct consequence of the definitions and of \eqref{eq:charAC} is
\begin{equation}
\label{eq:charACvec}
\begin{split}
&(v_t)\in AC^2([0,1],\H)\quad\Leftrightarrow\quad\exists (z_t)\in L^2([0,1],\H)\text{ such that }\\
&v_s-v_t=\int_t^s z_r\,\d r\qquad\forall t,s\in[0,1],\ t<s\quad \text{and in this case }\dot v_t=z_t,\quad\mm-a.e.,\ a.e.\ t.
\end{split}
\end{equation}
To further investigate the properties of $W^{1,2}_\H$ it will be convenient to notice the following fact (reminiscent of the classical statement `weak convergence+convergence of norms $\Rightarrow$ strong convergence'):
\begin{equation}
\label{eq:l2dal0}
\left.
\begin{array}{l}
\la v^n_t,z\ra \to \la v_t,z\ra ,\quad\text{ in }L^0([0,1],L^0(\X)),\quad\forall z\in\H\\
|(v^n_t)|_{L^2_\H}
\leq |(v_t)|_{L^2_\H}<\infty,\quad\mm-a.e.
\end{array}
\right\}
\qquad\Rightarrow\qquad (v^n_t)\stackrel{L^2_\H}\to (v_t).
\end{equation}
To see this, notice that for $t\mapsto z_t\in \H$ piecewise constant,
i.e.\ of the form $z_t=\sum_{i=1}^N\nchi_{A_i}(t)z^i$ with $A_i\subseteq[0,1]$ Borel and $z^i\in\mathscr H$,
the first in \eqref{eq:l2dal0} easily gives $\la v^n_t,z_t\ra \to \la v_t,z_t\ra$ in $L^0([0,1],L^0(\X))$. Moreover,  the second in \eqref{eq:l2dal0} gives $|\la v_t^n,z_t\ra|_{L^2}\leq  |(v_t)|_{L^2_\H} |(z_t)|_{L^\infty_\H}$ $\mm$-a.e.\ for every $n\in\N$.
Notice that if $f^n(t)\to f(t)$  for a.e.\ $t$ and $\sup_n\|f^n\|_{L^2}<\infty$, then we have $\int_0^1 f^n(t)\,\d t\to\int_0^1 f(t)\,\d t$:
thanks to the reflexivity of $L^2(0,1)$, we have that $(f^n)_n$ has a $L^2(0,1)$-weakly converging subsequence, with limit $g$; the pointwise a.e.\ convergence
$f^n\to f$ ensures that $g=f$, thus the original sequence converges to $f$ weakly in $L^2(0,1)$; in particular (by testing against the constant function $1\in L^2(0,1)$)
we conclude that $\int_0^1 f^n(t)\,\d t\to\int_0^1 f(t)\,\d t$, as desired.

Using this fact in conjunction with what already mentioned and recalling that a sequence converges in $L^0$ if and only if any subsequence has a further sub-subsequence converging a.e., we deduce that $\int_0^1\la v_t^n,z_t\ra \,\d t\to\int_0^1\la v_t,z_t\ra\,\d t$ in $L^0(\X)$. Hence
\[
\begin{split}
\lims_{n\to\infty}\sfd_{L^2_\H}\big((v^n_t),(z_t)\big)&=
\lims_{n\to\infty}\sfd_{L^0}\bigg(\sqrt{|(v^n_t)|_{L^2_{\mathscr H}}^2+|(z_t)|_{L^2_{\mathscr H}}^2-2\int_0^1\la v^n_t,z_t\ra\,\d t},0\bigg)\\
&\leq\lims_{n\to\infty} \sfd_{L^0}\bigg(\sqrt{|(v_t)|_{L^2_{\mathscr H}}^2+|(z_t)|_{L^2_{\mathscr H}}^2-2\int_0^1\la v^n_t,z_t\ra\,\d t},0\bigg)\\
&= \sfd_{L^0}\bigg(\sqrt{|(v_t)|_{L^2_{\mathscr H}}^2+|(z_t)|_{L^2_{\mathscr H}}^2-2\int_0^1\la v_t,z_t\ra\,\d t},0\bigg)\\
&=\sfd_{L^2_\H}\big((v_t),(z_t)\big).
\end{split}
\]
Now notice that the set of $(z_t)$'s considered is dense in $L^2_\H$ (as it is easy to establish from the definition), thus from
\[
\lims_{n\to\infty}\sfd_{L^2_\H}\big((v^n_t),(v_t)\big)\leq \lims_{n\to\infty}\sfd_{L^2_\H}\big((v^n_t),(z_t)\big)+\sfd_{L^2_\H}\big((v_t),(z_t)\big)\leq 2\sfd_{L^2_\H}\big((v_t),(z_t)\big)
\]
we conclude letting $(z_t)\to (v_t)$ in $L^2_\H$.

Now let $(v_t)\in L^2_\H$ and for $t,\eps\in(0,1)$ define $v^\eps_t:=\eps^{-1}\int_t^{t+\eps}v_s\,\d s$, where $v_s$ is intended to be 0 if $s>1$. We claim that 
\begin{equation}
\label{eq:vepsv0}
(v^\eps_t)\quad\to\quad(v_t)\qquad\text{ in }L^2_\H,\qquad\text{as } \eps\downarrow0
\end{equation}
and we shall prove this using \eqref{eq:l2dal0}. Let $z\in\H$ and for $\mm$-a.e.\ $x\in\X$ apply Lebesgue's theorem to the $L^1(0,1)$-function $t\mapsto \la z,v_t\ra(x)$ to conclude that for a.e.\ $t$ we have $\la z,v^\eps_t\ra (x)\to \la z,v_t\ra(x)$: this proves the first condition in \eqref{eq:l2dal0}. The second follows from the inequality $|v^\eps_t|\leq \eps^{-1}\int_t^{t+\eps}|v_s|\,\d s$ as it implies  $|(v^\eps_t)|_{L^2_\H}\leq|(v_t)|_{L^2_\H}$  $\mm$-a.e.\ for every $\eps$, so the conclusion follows from  \eqref{eq:l2dal0}.

We also claim that
\begin{equation}
\label{eq:vepsreg}
(v^\eps_t)\in W^{1,2}_\H\qquad\text{ with }\qquad \dot v^\eps_t=\frac{v_{t+\eps}-v_t}\eps,\quad a.e.\ t,
\end{equation}
where again $v_{t+\eps}$ is intended to be 0 if $t+\eps>1$. To prove this, just notice that by the very definition of $W^{1,2}_\H$ and of the corresponding notion of derivative, it is sufficient to consider the scalar case $\H=L^0(\X)$. In turn, in this setting the conclusion is a direct consequence of the analogous well-known result for real valued functions on $[0,1]$. Notice that, in particular, from \eqref{eq:vepsreg} and \eqref{eq:vepsv0} we deduce that
\begin{equation}
\label{eq:w12densel2}
W^{1,2}_\H\quad\text{ is dense in }\quad L^2_\H.
\end{equation}
Applying \eqref{eq:vepsv0} to the derivative of a vector field in $W^{1,2}_\H$  by keeping \eqref{eq:intvt} in mind we obtain that
\begin{equation}
\label{eq:diffquot}
\frac{v_{t+h}-v_t}h\quad\to\quad \dot v_t\qquad in\ L^2_\H \qquad\text{as }h\to0,
\end{equation}
where $\frac{v_{t+h}-v_t}h$ is intended to be 0 if $t+h\notin [0,1]$.

A further consequence  is that
\begin{equation}
\label{eq:leibdifftime}
\begin{split}
(v_t),(z_t)\in W^{1,2}_\H\ (resp.\ AC^2_\H)\quad\Rightarrow\quad
\begin{array}{l} 
(\la v_t,z_t\ra)\in W^{1,2}([0,1],L^0(\X))\\ 
\qquad\qquad(resp.\ AC^2([0,1],L^0(\X)))\\
\text{ with }
\frac{\d}{\d t}\la v_t,z_t\ra =\la \dot v_t,z_t\ra +\la v_t,\dot z_t\ra.
\end{array}
\end{split}
\end{equation}
Indeed, the inequality
\[
\begin{split}
\big|\la v_{t+h},z_{t+h}\ra-\la v_t,z_t\ra\big|&\leq \big|\la v_{t+h},z_{t+h}-z_t\ra\big|-\big|\la v_{t+h}-v_t,z_t\ra\big|\\
\text{(by \eqref{eq:normint}, \eqref{eq:intvt}, \eqref{eq:perlinfty1})}\qquad&\leq |(v_s)|_{L^\infty_\H}\int_t^{t+h}|\dot z_r|\,\d r+ |(z_s)|_{L^\infty_\H}\int_t^{t+h}|\dot v_r|\,\d r\qquad\mm-a.e.,
\end{split}
\]
valid for every $h\in(0,1)$ and a.e.\ $t\in[0,1-h]$, shows that $(\la v_t,z_t\ra)\in W^{1,2}([0,1],L^0(\X))$ (recall \eqref{eq:charw12bound}). Then letting $h\to 0$ in
\[
\frac{\la v_{t+h},z_{t+h}\ra-\la v_t,z_t\ra}h=\la v_t,\tfrac{z_{t+h}-z_t}h\ra+\la \tfrac{v_{t+h}-v_t}h,z_t\ra+h\la \tfrac{v_{t+h}-v_t}h,\tfrac{z_{t+h}-z_t}h\ra
\]
and using  \eqref{eq:diffquot} we see that the right-hand side goes to $(\la \dot v_t,z_t\ra +\la v_t,\dot z_t\ra)$ in $L^0([0,1],L^0(\X))$. On the other hand, applying  \eqref{eq:diffquot} with $\H:=L^0(\X)$ and $(\la v_t,z_t\ra)$ in place of $(v_t)$ we see that the left-hand side converges to $(\partial_t\la v_t,z_t\ra)$ in $L^2([0,1],L^0(\X))$, thus also in $L^0([0,1],L^0(\X))$, so that \eqref{eq:leibdifftime} is proved.

Arguing along the same lines one uses to prove that $W^{1,2}(0,1)\hookrightarrow C([0,1])$, we can get existence of continuous representatives of elements of $W^{1,2}_\H$:
\begin{proposition}
Let $(v_t)\in W^{1,2}_\H$. Then there is a unique $(\bar v_t)\in C([0,1],\H)$ with $\bar v_t=v_t$ for a.e.\ $t\in[0,1]$.
\end{proposition}
\begin{proof} For $\eps\in(0,1)$ we define $v^\eps_t:=\eps^{-1}\int_t^{t+\eps}v_s\,\d s$ as before (here $v_s$ is intended to be 0 if $s>1$). Notice that 
\[
\dot v^\eps_t\stackrel{\eqref{eq:vepsreg}}=\frac{v_{t+\eps}-v_t}{\eps}\stackrel{\eqref{eq:intvt}}=\eps^{-1}\int_t^{t+\eps}\dot v_s\,\d s,
\]
therefore by \eqref{eq:vepsv0} we deduce that $(v^\eps_t),(\dot v^\eps_t)$ converge to $(v_t),(\dot v_t)$ respectively in $L^2_\H$ as $\eps\downarrow0$. In other words, $(v_t^\eps)\to (v_t)$ in $W^{1,2}_{\H}$ as $\eps\downarrow0$.

Now we claim that for every $\eps\in(0,1)$ the curve $t\mapsto v^\eps_t\in \H $ is continuous. Indeed, for any $t,s\in[0,1]$ with $0\leq s-t\leq \eps$, it is clear that we have
\[
|v^\eps_s-v^\eps_t|=\eps^{-1}\big|\int_{t+\eps}^{s+\eps}v_r\,\d r+\int_{t}^{s}v_r\,\d r\big|\leq\eps^{-1}\big(\int_{t+\eps}^{s+\eps}|v_r|\,\d r+\int_{t}^{s}|v_r|\,\d r\big)\qquad\mm-a.e.,
\]
and that the right-hand side goes to 0 $\mm$-a.e.\ both as $t\uparrow s$ and as $s\downarrow t$. Hence $|v_s^\eps- v^\eps_t|\to 0$ in the $\mm$-a.e.\ sense as $s\to t$, and thus a fortiori $v_s^\eps \to v_t^\eps$ in $\H$. 

We now claim  that the family of curves $\{(v^\eps_t)\}_{\eps\in(0,1)}$ is Cauchy in $C([0,1],\H)$ as $\eps\downarrow0$. Indeed, for $\eps,\eta\in(0,1)$ we have
\[
\begin{split}
\sup_{t\in[0,1]}\sfd_{\mathscr H}(v^\eps_t,v^\eta_t)=\sup_{t\in[0,1]}\sfd_{L^0}(|v^\eps_t-v^\eta_t|,0)\stackrel{\eqref{eq:perlinfty2}}\leq\sfd_{L^0}\big(|(v^\eps_t-v^\eta_t)|_{W^{1,2}_{\H}},0\big)=\sfd_{W^{1,2}_{\H}}\big((v^\eps_t),(v^\eta_t)\big).
\end{split}
\]
Since we proved that $(v^\eps_t)\stackrel{W^{1,2}_{\H}}\to(v_t)$, we know that the right-hand side of the above goes to 0 as $\eps,\eta\downarrow0$, hence our claim is proved. Let $(\bar v_t)$ be the limit of $(v^\eps_t)$ in $C([0,1],\H)$ as $\eps\downarrow0$: since we also know that $(v^\eps_t)\stackrel{L^2_{\H}}\to(v_t)$ and it is clear that $C([0,1],\H)$ continuously embeds in $L^2_{\H}$, we conclude that $(\bar v_t)$ and $(v_t)$ agree as elements of $L^2_\H$, i.e.\ that  $\bar v_t=v_t$ for a.e.\ $t$, as desired. 
\end{proof}
We shall make use of the following simple density-like result:
\begin{proposition}\label{prop:Atutto}
Let $\mathcal A\subset W^{1,2}([0,1],\H)$. Assume that $\mathcal A$:
\begin{itemize}
\item[o)] is a vector space,
\item[i)] is stable under  `restriction', i.e.\ if $(v_t)\in\mathcal A$  and $E\subset\X$ is Borel, then  $t\mapsto \nchi_{E}v_t$ belongs to $\mathcal A$,
\item[ii)] is stable under multiplication by functions in $C^1([0,1])$, i.e.\ if $(v_t)\in\mathcal A$ and $\varphi\in C^1([0,1])$ the map $t\mapsto \varphi(t)v_t$ belongs to $\mathcal A$,
\item[iii)] is closed in the $W^{1,2}_\H$ topology,
\item[iv)] contains the constant vector fields $t\mapsto \bar v$ for any given $\bar v\in \H$.
\end{itemize}
Then $\mathcal A=W^{1,2}([0,1],\H)$.
\end{proposition}
\begin{proof}
Let $(v_t)\in W^{1,2}_\H$ be arbitrary, $(e^n)\subset\H$ be a local Hilbert base (see \cite[Theorem 1.4.11]{Gigli14}) and for every $N\in\N$ let $v^N_t:=\sum_{n\leq N}\la  v_t,e^n\ra e^n$ and $\dot v^N_t:=\sum_{n\leq N}\la  \dot v_t,e^n\ra e^n$  for a.e.\ $t\in[0,1]$. By the properties of local Hilbert bases we have that $|v^N_t|\leq |v_t|$ $\mm$-a.e.\ and $v^N_t\to v_t$ in $\H$ as $N\to\infty$ for a.e.\ $t$ and similarly for $\dot v^N_t$. By   \eqref{eq:l2dal0} this is sufficient to deduce that $(v^N_t)\stackrel{L^2_\H}\to (v_t)$ and  $(\dot v^N_t)\stackrel{L^2_\H}\to (\dot v_t)$ as $N\to\infty$. Also, from the definition of $W^{1,2}_\H$ it is clear that $(v^N_t)\in W^{1,2}_\H$ with derivative $(\dot v^N_t)$ for every $N\in\N$, thus $(v^N_t)\stackrel{W^{1,2}_\H}\to (v_t)$. 

Therefore  to conclude that $(v_t)\in\mathcal A$ it is sufficient to prove that $(v^N_t)\in \mathcal A$ for every $N\in\N$. Since we assumed $\mathcal A$ to be a vector space, to prove this it is sufficient to show that for any $v\in\H$, the collection of those $(f_t)$'s in $W^{1,2}_{L^0}$ such that $t\mapsto f_tv$ is in $\mathcal A$ coincides with the whole $W^{1,2}_{L^0}$. This is a direct consequence of  Proposition \ref{prop:Afunc} and our assumptions.
\end{proof}
Our last goal for the  section is the study of an analogue of Hille's theorem in this context. The classical proof of this fact for Bochner integral of Banach-valued maps uses the fact that if $v_t$ belongs to some convex closed set for a.e.\ $t$, then so does its integral over $[0,1]$. In our setting, this is also true and the proof is based on the possibility of obtaining integration as limit of properly chosen Riemann sums.  For real valued functions on $[0,1]$, this last property is a classical statement of Hahn \cite{hahn14} (see also \cite{Hen68}); the proof we give is closely related to the ones in  \cite[Theorem I-2.8]{Doob90} and \cite[Lemma 4.4]{MS20}. In the course of the proof we shall use properties like additivity of the integral over disjoint sets and change-of-variable formulas that can be trivially obtained from the very definition of integration.
\begin{proposition}\label{prop:riem}
Let $(v_t)\in L^1([0,1],\H)$. Then for every sequence $\tau_k\downarrow0$ there is a subsequence, not relabeled, such that for a.e.\ $t\in[0,1]$ we have
\[
\lim_{k\to\infty}\tau_k\sum_{i=0}^{{[\tau_k^{-1}]-1}}v_{\tau_k(t+i)}=\int_0^1v_s\,\d s\qquad in\ \H,
\]
where $[\cdot]$ denotes the integer part.

In particular, if ${\rm C}\subset\H$ is convex and closed and $v_t\in {\rm C}$ for a.e.\ $t$, then $\int_0^1 v_t\,\d t\in {\rm C}$ as well.
\end{proposition}
\begin{proof} Fix a Borel representative of $(v_t)$ such that $v_t\in {\rm C}$ for every $t\in[0,1]$ and notice that $\frac1{[\tau^{-1}]}\sum_{i=0}^{{[\tau^{-1}]-1}}v_{\tau(t+i)}$ is a convex combination of the $v_t$'s, and thus belongs to ${\rm C}$. Since $\tau[\tau^{-1}]\to1$ as $\tau\downarrow0$, it is clear that the second statement is a consequence of the first one, so we focus on this.

Put $v_t=0$ for $t>1$ and notice that \eqref{eq:contint} gives that  $\int_{\tau t}^{1+\tau t}v_s\,\d s\to\int_0^1v_s\,\d s$ in $\H$ as $\tau\downarrow0$ for every $t\in[0,1]$. Thus to conclude it is sufficient to prove that
\begin{equation}
\label{eq:perriem}
\lim_{\tau\downarrow0}\int_0^1\sfd_{\H}(z_{\tau,t},0)\,\d t=0,\qquad\text{where}\qquad z_{\tau,t}:=\tau\sum_{i=0}^{[\tau^{-1}]-1}v_{\tau(t+i)}-\int_{\tau t}^{1+\tau t} v_s\,\d s.
\end{equation}
To this aim we start claiming that 
\begin{equation}
\label{eq:claimperriem}
\lim_{\tau\downarrow0}\int_0^1\int_0^1\big|v_{t}-v_{t+\tau s}\big|\,\d t\,\d s\quad\to\quad 0,\qquad in\ L^0(\X)\ as\ \tau\downarrow0.
\end{equation}
To prove this, we first notice that  the truncations $t\mapsto \nchi_{\{|v_t|\leq n\}}v_t=\varphi_n(|v_t|)v_t$, where $\varphi_n:\R\to\R$ is given by $\varphi_n:=\nchi_{[0,n]}$, belong to $L^\infty_\H\subset L^2_\H$ and trivially converge to $(v_t)$ in $L^1_\H$. Thus recalling \eqref{eq:w12densel2} we can find $(w^n_t)\subset W^{1,2}_\H$ converging to $(v_t)$ in $L^1_\H$. Now for every $n\in\N$ put $w_t^n=0$ for $t>1$ and notice that 
\[
\begin{split}
\int_0^1\int_0^1\big|v_{t}-v_{t+\tau s}\big|\,\d t\,\d s&\leq \int_0^1\int_0^1\big|w^n_{t}-w^n_{t+\tau s}\big|\,\d t\,\d s+\int_0^1\int_0^1\big|v_{t}-w_t^n\big|+\big|v_{t+\tau s}-w_{t+\tau s}^n\big|\,\d t\,\d s\\
&\leq \underbrace{\int_0^1\int_0^1\int_t^{t+\tau s}|\dot w_r|\,\d r\,\d t\,\d s}_{\to 0\quad\mm-a.e.\ as\ \tau\downarrow0}+2|(v_{t}-w_t^n)|_{L^1_\H}.
\end{split}
\]
Thus \eqref{eq:claimperriem} follows by first letting $\tau\downarrow0$ and then $n\to\infty$ in the above (in fact the argument easily gives that there is convergence $\mm$-a.e.\ in \eqref{eq:claimperriem}).

Now recall that $z_{\tau,t}$ is defined in \eqref{eq:perriem} and notice that
\begin{equation}
\label{eq:stimaz}
\begin{split}
\int_0^1|z_{\tau,t}|\,\d t&=\int_0^1\Big|\tau\sum_{i=0}^{[\tau^{-1}]-1}\Big(v_{\tau(t+i)}-\tau^{-1}\int_{\tau(t+i)}^{\tau(t+i+1)}v_s\,\d s\Big)\Big|\,\d t\\
\big(s=\tau(t+i+s')\big)\qquad&=\int_0^1\Big|\tau\sum_{i=0}^{[\tau^{-1}]-1}\Big(v_{\tau(t+i)}-\int_0^1v_{\tau(t+s'+i)}\,\d s'\Big)\Big|\,\d t\\
\big(\text{by \eqref{eq:normint}}\big)\qquad&\leq\int_0^1\int_0^1\tau\sum_{i=0}^{[\tau^{-1}]-1}\big|v_{\tau(t+i)}-v_{\tau(t+s'+i)}\big|\,\d t\,\d s'\\
\big(t'=\tau(t+i)\big)\qquad&=\int_0^1\int_0^1\big|v_{t'}-v_{t'+\tau s'}\big|\,\d t'\,\d s'\qquad\mm-a.e..
\end{split}
\end{equation}
Therefore
\[
\begin{split}
\int_0^1\sfd_{L^0}(|z_{\tau,t}|,0)\,\d t&=\int_0^1\int 1\wedge |z_{\tau,t}|\,\d\mm'\,\d t\leq\int 1\wedge\Big( \int_0^1 |z_{\tau,t}|\,\d t\Big)\,\d\mm'\quad\stackrel{\eqref{eq:claimperriem},\eqref{eq:stimaz}}\to\quad0,
\end{split}
\]
which is \eqref{eq:perriem}.
\end{proof}

It is now easy to establish the desired version of Hille's theorem:
\begin{theorem}[A Hille-type result]\label{thm:HilleH}
Let $\mathscr H_1,\mathscr H_2$ be separable Hilbert $L^0(\mm)$-modules, $V\subset \mathscr H_1$ a vector subspace, $\|\cdot\|_V:V\to\R^+$ a norm on $V$ and $L:V\to \mathscr H_2$ a linear map. Let $(v_t)\in L^1([0,1],\mathscr H_1)$ and assume that:
\begin{itemize}
\item[i)] $v_t\in V$ for a.e.\ $t\in[0,1]$ and $t\mapsto L(v_t)$ is in $L^1([0,1],\mathscr H_2)$,
\item[ii)] for any $C>0$ the set ${\rm Graph}(L)\cap (\{v\in V:\|v\|_V\leq C\}\times \mathscr H_2)\subset  \H_1 \times\H_2$  is closed,
\item[iii)] $\int_0^1\|v_t\|_V\,\d t<\infty$.
\end{itemize}
Then $\int_0^1v_t\,\d t\in V$ and $L(\int_0^1v_t\,\d t)=\int_0^1L(v_t)\,\d t$.
\end{theorem}
\begin{proof} Start noticing that for any $C>0$ we have 
\[
\{t\ :\ \|v_t\|_V\leq C\}=\big\{t\ :\ (v_t,L(v_t))\in{\rm Graph}(L)\cap (\{v\in V:\|v\|_V\leq C\}\times \mathscr H_2) \big\},
\]
thus from $(ii)$ and the assumed Borel regularity of $t\mapsto v_t,L(v_t)$ we deduce that $t\mapsto\|v_t\|_V$ is Borel, hence assumption $(iii)$ makes sense.

Let $\H:=\H_1\times\H_2$ be equipped with the pointwise norm $|(v,z)|^2:=|v|^2+|z|^2$ $\mm$-a.e.\ for every $v\in\H_1$, $z\in\H_2$. It is clear that $\H$ is a Hilbert $L^0(\X)$-normed module, when equipped with the obvious distance and product with $L^0(\X)$ functions defined componentwise. Also, directly by definition of integral, testing with elements of the form $(v,0)$ and $(0,z)$ we see that $\int_0^1(v_t,z_t)\,\d t=(\int_0^1 v_t\,\d t,\int_0^1 z_t\,\d t)$. 

Now observe that the assumptions $(v_t)\in L^1_{\H_1}$ and $(L(v_t))\in L^1_{\H_2}$ together with the trivial bound $|(v,z)|\leq |v|+|z|$ show that $t\mapsto (v_t,L(v_t))$ is in $L^1_\H$. We can therefore apply Proposition \ref{prop:riem} above to $t\mapsto (v_t,L(v_t))\in \H$ to find $\tau_k\downarrow0$ so that for a.e.\ $t\in[0,1]$ we have
\begin{equation}
\label{eq:limH}
\lim_{k\to\infty}\tau_k\sum_{i=0}^{{[\tau_k^{-1}]-1}}\big(v_{\tau_k(t+i)},L(v_{\tau_k(t+i)})\big)=\Big(\int_0^1v_s\,\d s,\int_0^1L(v_s)\,\d s\Big)\qquad in\ \H.
\end{equation}
By the original result of Hahn (or, which is the same, by Proposition \ref{prop:riem} with $\H:=\R$, that is a $L^0$-normed module over a Dirac mass) we also know that for a.e.\ $t\in[0,1]$ we have  
\[
\lims_{k\to\infty} \Big\|\tau_k\sum_{i=0}^{{[\tau_k^{-1}]-1}}v_{\tau_k(t+i)}\Big\|_V\leq\lim_{k\to\infty}\tau_k\sum_{i=0}^{{[\tau_k^{-1}]-1}}\|v_{\tau_k(t+i)}\|_V=\int_0^1\|v_s\|_V\,\d s.
\]
Hence there is $t\in[0,1]$ such that $C:=\sup_k\Big\|\tau_k\sum_{i=0}^{{[\tau_k^{-1}]-1}}v_{\tau_k(t+i)}\Big\|_V<\infty$ and \eqref{eq:limH} holds. Thus the argument of the limit in the left-hand side of \eqref{eq:limH} is in ${\rm Graph}(L)\cap (\{v\in V:\|v\|_V\leq C\}\times \mathscr H_2)$, and since by assumption this set is closed, it contains also the limit. In particular, this means that the right-hand side is in ${\rm Graph}(L)$, which is the conclusion.
\end{proof}

\section{An existence and uniqueness theory in \texorpdfstring{${\sf ncRCD}(K,N)$}{ncRCD} spaces}

\subsection{The setting}\label{se:setting}

Let us fix once and for all the assumptions and the notations that we shall use in the next sections: unless otherwise specified, our theorems will all be based on these:
\begin{itemize}
\item[-] $(\X,\sfd,\mm)$ is a ${\sf ncRCD}(K,N)$ space, $K\in\R$, $N\in\N$.
\item[-] $(w_t) \in L^2([0,1], W^{1,2}_{C}(T\X))$ is  such that $|w_t|,|\div (w_t)|\in L^\infty([0,1]\times\X)$ and for some $\bar x\in \X$ and $R>0$ we have $\supp(w_t)\subset B_R(\bar x)$ for a.e.\ $t\in[0,1]$.
\item[-] $(F_t^s)$ is the Regular Lagrangian Flow of $(w_t)$.
\item[-] $(g_t)\in L^2([0,1],L^2(\X))$ is the function associated to $(w_t)$ as in Proposition \ref{prop:eliadanieleregularity}.
\item[-] $\mm'$ is a Borel probability measure on $\X$ such that $\mm\ll\mm'\ll\mm$ and coinciding with $c\mm$ on $B_R(\bar x)$ for some $c>0$.
\item[-] The distance $\sfd_{L^0}$ on the space $L^0(\X)$ is given by
\[
\sfd_{L^0}(f,g):=\int1\wedge|f-g|\,\d\mm',\qquad\forall f,g\in L^0(\X).
\]
Similarly the distance $\sfd_{L^0(T^*\X)}$ on $L^0(T^*\X)$ is defined as
\[
\sfd_{L^0(T^*\X)}(\omega,\eta):=\int1\wedge|\omega-\eta|\,\d\mm',\qquad\forall \omega,\eta\in L^0(T^*\X)
\]
and analogously for the  distance $\sfd_{L^0(T\X)}$ on $L^0(T\X)$.
\item[-] $G$ is a non-negative function in $L^0(\X)$ acting as `generic constant'. Its actual value may change in the various instances where it appears, but it depends only on the structural data (i.e.\ the space and the vector field $(w_t)$), but not on the specific vector fields $(v_t),(V_t)$ we shall work with later on. 

Occasionally, we will need to work with such `generic function' depending on some additional parameter (often a time variable). In this case we shall write $G_i$, $i\in I$, to emphasize such dependence. In any case, whenever we do so we tacitly (or explicitly) assume that the $G_i$'s are `uniformly small in $L^0$', i.e.\ it will always be intended that they satisfy
\begin{equation}
\label{eq:perunifcont3}
\lim_{C\to+\infty}\sup_{i\in I}\mm'\big(\{|G_i|\geq C\}\big)=0.
\end{equation}
\item[-] $\Omega:\R^+\to\R^+$ is a generic modulus of continuity, i.e.\ a non-decreasing continuous function such that $\Omega(z)\downarrow0$ as $z\downarrow0$. Much like for the function $G$ its actual value may change in the various instances where it appears, but it depends only on the structural data of the problem.
\end{itemize}
Notice that the  choice of $\mm'$, the fact that the flow $(F_t^s)$ has bounded compression, and the fact that  $F_t^s$ is the identity on the complement of ${\rm supp}(w)$, ensure that
\begin{equation}
\label{eq:ftmmp}
(F_t^s)_*\mm'\leq C\mm',\qquad\forall t,s\in[0,1],
\end{equation}
for some $C>0$. In particular, this implies
\begin{equation}
\label{eq:equil0}
C^{-1}\sfd_{L^0}(f,g)\leq\sfd_{L^0}(f\circ F_t^s,g\circ F_t^s)\leq C\sfd_{L^0}(f,g),\qquad\forall t,s\in[0,1],\ \forall f,g\in L^0(\X).
\end{equation}
More generally, the topological vector spaces we are dealing with are metrized by translation invariant distances, therefore linear and continuous operators are uniformly continuous. To see this, let $\sfd_1,\sfd_2$ be the distances on two such spaces $V_1,V_2$ and $T:V_1\to V_2$ be the linear operator, then if $\sfd_1(v_n,z_n)\to 0$ we have $\sfd_1(v_n-z_n,0)\to 0$ and thus $\sfd_2(T(v_n)-T(z_n),0)=\sfd_2(T(v_n-z_n),0)\to 0$ and therefore $\sfd_2(T(v_n),T(z_n))\to 0$ (in fact this is not really due to distances, but to the fact that topological vector spaces, much like topological groups, are uniform spaces).

We can apply  this general principle to the product by a given function $G\in L^0(\X)$, which is a linear continuous map from $L^0(\X)$ to itself, to deduce 
\begin{equation}
\label{eq:unifO}
\sfd_{L^0}(Gf,Gg)\leq\Omega\big(\sfd_{L^0}(f,g)\big),\qquad\forall f,g\in L^0(\X).
\end{equation}
The same principle applied to the linear and continuous embedding of $L^p(\X)$ into $L^0(\X)$ gives
\begin{equation}
\label{eq:lpl0}
\sfd_{L^0}(f,g)\leq\Omega\big(\|f-g\|_{L^p}\big),\qquad\forall f,g\in L^p(\X).
\end{equation}
Clearly, the map $\Omega$ appearing in these depends on the fixed function $G$ and the chosen exponent $p$, but since such dependence is not important in our discussion, to keep the notation simpler we will not emphasize this fact.

It will be useful to notice that for $(f_s)\in L^2([0,1],L^2(\X))$, from the bounded compression  property  of the flow we deduce that $(f_s\circ F_t^s)\in L^2([0,1],L^2(\X))$ as well for any $t\in[0,1]$. Thus from Fubini's theorem we deduce that 
\begin{equation}
\label{eq:intflo2}
(f_s)\in L^2([0,1],L^2(\X))\qquad\Rightarrow\qquad \int_{0}^1|f_r|^2\circ F_t^r\,\d r<\infty,\quad\mm-a.e.,\ \forall t\in[0,1].
\end{equation}
In particular, applying this to $f_s:=|\nabla w_s|$ we deduce that
\begin{equation}
\label{eq:estnw2}
\int_0^1|\nabla w_t|^2\circ F_0^t\,\d t\leq G\qquad\mm-a.e.,
\end{equation}
according to our convention on $G$ described above.  Similarly, from \eqref{eq:intflo2} applied to $(g_s)$,  the trivial inequality $e^a\leq 1+ae^b$ valid for any $0\leq a\leq b$,  and the estimate \eqref{eq:estimate_differential_eliadaniele}, we deduce the key bound
\begin{equation}
\label{eq:estdFt3}
|\d F_t^s|\circ F_{t'}^t\leq 1+G_{t'}\int_{t\wedge s}^{t\vee s}g_r\circ F_{t'}^r\,\d r\qquad\mm-a.e.,\ \forall t',t,s\in[0,1],
\end{equation}
where $G_{t'}:=e^{\int_0^1g_r\circ F_{t'}^r\,\d r}$. We claim that the functions $G_{t'}$ so defined satisfy the uniform bound \eqref{eq:perunifcont3}, so that our notation is justified according to what we declared at the beginning of the section. To see this, notice that we have the uniform estimate
\[
\int\Big|\int_0^1g_r\circ F_{t'}^r\,\d r\Big|^2\,\d\mm\leq\int_0^1\int |g_r|^2\circ F_{t'}^r\,\d\mm\,\d r\stackrel{\eqref{eq:boundcompr}}\leq C \iint_0^1|g_r|^2\,\d r\,\d\mm,\qquad\forall t'\in[0,1],
\]
where $C$ is the compressibility constant of $F_t^s$; thus Chebyshev's inequality gives for every $M>0$ 
\[
\mm(\{G_{t'}>M\})=\mm\Big(\Big\{\big|\int_0^1g_r\circ F_{t'}^r\,\d r\big|^2>|\log M|^2\Big\}\Big)\leq\frac{C\iint_0^1|g_r|^2\,\d r\,\d\mm}{|\log M|^2},\qquad\forall t'\in[0,1]
\]
and by the absolute continuity of the integral it follows that
\begin{equation}
\label{eq:unifsuperlgt}
\lim_{M\to+\infty}\sup_{t'\in [0,1]}\mm'\big(\{G_{t'}\geq M\}\big)=0.
\end{equation}
Notice that \eqref{eq:estdFt3} trivially implies the weaker, but still useful, bound
\begin{equation}
\label{eq:estdFt4}
|\d F_t^s|\circ F_{t'}^t\leq G_{t'}\qquad\mm-a.e.,\ \forall t',t,s\in[0,1],
\end{equation}
for some $G_{t'}$'s that still satisfy \eqref{eq:perunifcont3}.

We shall  often use the bound  \eqref{eq:perunifcont3} in conjunction with the following simple lemma:
\begin{lemma}\label{le:unifcont}
Let $G_i\in L^0(\X)$, $i\in I$, be such that 
\begin{equation}
\label{eq:perunifcont2}
\lim_{C\to+\infty}\sup_{i\in I}\mm'\big(\{|G_i|\geq C\}\big)=0.
\end{equation}
Then  the operators $L^0(\X)\ni f\mapsto G_if\in L^0(\X)$ are uniformly continuous, i.e.\ for some modulus of continuity $\Omega$ independent of $i\in I$ we have
\begin{equation}
\label{eq:equihi}
\sfd_{L^0}(G_if,G_ig)\leq \Omega\big(\sfd_{L^0}(f,g)\big),\qquad\forall f,g\in L^0(\X),\ i\in I.
\end{equation}
\end{lemma}
\begin{proof}
Fix $\eps\in(0,1)$, let $C_\eps>1$ be such that $\mm'\big(\{|G_i|\geq C_\eps\}\big)\leq\eps$ for every $i\in I$ and put $\delta:=\frac{\eps^2}{C_\eps}$. We shall prove that
\begin{equation}
\label{eq:perequicont1}
\sfd_{L^0}(f_1,f_2)\leq\delta\qquad\Rightarrow\qquad \sfd_{L^0}(G_if_1,G_if_2)\leq 3\eps,\qquad\forall i\in I,
\end{equation}
that, by the arbitrariness of $\eps$, gives the claim. Put for brevity $f:=|f_1-f_2|$ and use Cavalieri's formula to get
\begin{equation}
\label{eq:deltaeps}
\delta\geq \sfd_{L^0}(f_1,f_2)=\int_0^1\mm'(\{f>t\})\,\d t\geq \int_0^{\frac\delta\eps}\mm'(\{f>t\})\,\d t\geq \tfrac\delta\eps\mm'\big(\big\{f>\tfrac\delta\eps \big\}\big).
\end{equation}
On the other hand, the trivial bound 
\[
\int_0^1\mm'(\{|G_i|f>t\})\,\d t\leq\mm'(\{|G_i|f>\eps\})+\int_0^\eps\mm'(\{|G_i|f>t\})\,\d t\leq \mm'(\{|G_i|f>\eps\})+\eps
\] and the inclusion $\{|G_i|f>\eps\}\subset \{|G_i|>C_\eps\}\cup\{f>\tfrac\eps{C_\eps}\}$ give
\[
 \sfd_{L^0}(G_if_1,G_if_2)=\int_0^1\mm'(\{|G_i|f>t\})\,\d t\leq\mm'(\{|G_i|>C_\eps\})+\mm'(\{f>\tfrac\eps{C_\eps}\})+\eps.
\]
Since $\frac\eps{C_\eps}=\frac\delta\eps$, the choice of $C_\eps$ and  \eqref{eq:deltaeps} give the claim \eqref{eq:perequicont1} and thus \eqref{eq:equihi}.
\end{proof}

\subsection{Time dependent vector fields at `fixed' and `variable' points}

We have seen in Section \ref{sec:diff} that for every $t\in\R$ the map $F_0^t$  admits a differential $\d F_0^t:L^0(T\X)\to L^0(T\X)$. We now want to consider all these maps as $t$ varies in $[0,1]$. 

\bigskip

In dealing with composition with the flow maps, the following simple lemma will  be useful:
\begin{lemma}\label{le:contcont} $[0,1]^2\ni (t,s)\mapsto f_{t}^s\in L^0(\X)$ is continuous if and only if $[0,1]^2\ni (t,s)\mapsto f_{t}^{s}\circ F_t^s\in L^0(\X)$ is continuous.
\end{lemma}
\begin{proof} Since $(F_{t}^s)^{-1}=F_s^t$, it is sufficient to prove the `only if'. Thus suppose that $t,s\mapsto f_{t}^s$ is continuous in $L^0(\X)$ and notice that
\[
\begin{split}
\sfd_{L^0}\big(f_{t'}^{s'}\circ F_{t'}^{s'},f_t^s\circ F_t^s\big)
&\leq \sfd_{L^0}\big(f_{t'}^{s'}\circ F_{t'}^{s'},f_t^s\circ F_{t'}^{s'}\big)+ \sfd_{L^0}\big(f_t^s\circ F_{t'}^{s'},f_t^s\circ F_{t}^{s}\big)\\
&\leq C\,\sfd_{L^0}\big(f_{t'}^{s'},f_t^s\big)+ \sfd_{L^0}\big(f_t^s\circ F_{t'}^{s'},f_t^s\circ F_{t}^{s}\big)
\end{split}
\]
having used the uniform bound \eqref{eq:equil0}. The conclusion follows recalling \eqref{eq:contl0}.
\end{proof}
The  regularity in time of $\d F_t^s$ will be obtained by duality starting from the following  result:
\begin{lemma}\label{lem:contd}
Let $f:\X\to\R$ be Lusin--Lipschitz. Then $[0,1]^2\ni (t,s)\mapsto \d (f\circ F_t^s)\in L^0(T^*\X)$ is continuous.  
\end{lemma}
\begin{proof}  We start claiming that for any $f$ Lusin--Lipschitz we have
\begin{equation}
\label{eq:nuovoclaim}
\lim_{s'\to s}\d (f\circ F_s^{s'})(v)=\d f(v)=\lim_{t'\to t}\d (f\circ F_{t'}^{t})(v),\qquad in\ L^0(\X),\ \forall v\in L^0(T\X),\ \forall t,s\in[0,1].
\end{equation}
We start with the first limit and notice that by definition of differential for Lusin--Lipschitz maps we have $\d((\nchi_Ef)\circ F_s^{s'})=\nchi_E\circ F_s^{s'}\,\d(f\circ F_s^{s'})$, thus taking into account \eqref{eq:contl0} and using the fact that the flow is the identity outside a bounded set, it is easy to see that the first in \eqref{eq:nuovoclaim} will follow in the general case if we prove it just for $f$ Lipschitz. Thus let this be the case and recall from \eqref{eq:unifLip2} that there is  a Borel partition $(E_i)$ of $\mm$-a.a.\ $\X$ made of bounded sets such that $\sup_{s'\in[0,1]}\Lip(F_s^{s'}\restr {E_i})<\infty$ for every $i\in\N$. Thus    $\sup_{s'\in[0,1]}\Lip((f\circ F_s^{s'})\restr {E_i})<\infty$ and by $L^0$-linearity and continuity, the first in \eqref{eq:nuovoclaim} will follow if we show that  for any $i\in\N$ we have $ \nchi_{E_i}\d(f\circ F_s^{s'})\weakto \d f$ in the weak topology of $L^2(T^*\X)$  as $s'\to s$.

Let  $(s'_n)\subset[0,1]$ be converging to $s$ and, for every $n\in\N\cup\{\infty\}$, let $g_n$ be uniformly Lipschitz with uniformly bounded support such that $g_n\restr {E_i}=(f\circ F_{s}^{s'_n})\restr {E_i}$ (it is easy to see that they can be found). Then $(g_n)$ is a bounded sequence in $W^{1,2}(\X)$, hence up to pass to a non-relabeled subsequence we can assume that it weakly converges to some $g\in W^{1,2}(\X)$. Clearly, $g=f$ on ${E_i}$, thus from the locality of the differential we see that
\[
\nchi_{E_i} \d (f\circ F_{s}^{s'_n})=\nchi_{E_i} \d g_n\stackrel{L^2(T^*\X)}\weakto \nchi_{E_i}\d g=\nchi_{E_i} \d  f.
\]
As this result does not depend on the subsequence chosen, the first in \eqref{eq:nuovoclaim} follows. 

We turn to the second and start noticing that arguing as for Lemma \ref{le:contcont} above it is sufficient to prove that
\begin{equation}
\label{eq:st11}
\d(f\circ F_{t'}^t)(v)\circ F_t^{t'}\to \d f(v)\qquad in \ L^0(\X),\ \forall v\in L^0(T\X),\ \forall t\in[0,1].
\end{equation}
Then taking into account the uniform bound
\[
|\d(f\circ F_{t'}^t)(v)\circ F_t^{t'}|\stackrel{\eqref{eq:techn_comp_cl}}\leq |v|\circ F_t^{t'}|\d f||\d F_{t'}^t|\circ F_t^{t'}\stackrel{\eqref{eq:estdFt4}}\leq  G_t |v|\circ F_t^{t'}|\d f|,
\]
 the fact that $|v|\circ F_t^{t'}\to |v|$ in $L^0(\X)$ as $t'\to t$ (recall \eqref{eq:contl0}), and the density of differentials of Lusin--Lipschitz functions in $L^0(T^*\X)$ (already noticed in the proof of Theorem \ref{thm:diff_LL}), we see that to prove \eqref{eq:st11}, and thus the second in \eqref{eq:nuovoclaim}, it is sufficient to prove that
\begin{equation}
\label{eq:st12}
\la \d(f\circ F_{t'}^t),\d h\ra \circ F_t^{t'}\to \la \d f,\d h\ra \qquad in\ L^0(\X),\ \forall h:\X\to\R\text{ Lusin--Lipschitz},\ \forall t\in[0,1].
\end{equation}
Now we put for brevity $\hat g_t^{t'}:=\int_{t\vee t'}^{t\wedge t'} g_r\circ F_t^r\,\d r\in L^0(\X)$ and notice that, rather trivially, we have
\begin{equation}
\label{eq:ghat}
\lim_{t'\to t}\hat g_t^{t'}= 0,\qquad in\ L^0(\X)
\end{equation}
(as for any sequence $t_n'\to t$ we have $\mm$-a.e.\ convergence).
The fact that $(1+a)^{-2} \geq (1-a)^2$ for any $a\in\R$ grants, together with \eqref{eq:estdFt3}, that
\begin{equation}
\label{eq:estbelow}
-|\d F_t^{t'}|^{-2} \le -(1+G_t \hat g_t^{t'})^{-2} \le -(1-G_t \hat g_t^{t'})^{2}.
\end{equation}
Also notice that  \eqref{eq:techn_comp_cl} with $f\circ F_{t'}^t$ in place of $f$ and $F_t^{t'}$ in place of $\varphi$ gives $|\d (f\circ F_{t'}^t)|\circ F_t^{t'}\geq |\d f| |\d F_t^{t'}|^{-1}$.
Similarly, one can show that $|\d h|\circ F_t^{t'}\geq |\d(h\circ F_t^{t'})| |\d F_t^{t'}|^{-1}$.
Using these bound in conjunction with \eqref{eq:techn_comp_cl} gives
\[
\begin{split}
\la \d(f\circ F_{t'}^t),\d h\ra  \circ F_t^{t'}&=\tfrac12 \left[ |\d(f\circ F_{t'}^t+h)|^2 \circ F_t^{t'}-(|\d(f\circ F_{t'}^t)|^2+|\d h|^2)\circ F_t^{t'} \right]\\
&\leq \tfrac12 \left[|\d(f+h\circ F^{t'}_t)|^2\,|\d  F_{t'}^t|^2\circ F_t^{t'}-|\d F_t^{t'}|^{-2}(|\d f |^2+|\d (h\circ F_{t}^{t'})|^2)\right]\\
\text{(by \eqref{eq:estdFt3} and \eqref{eq:estbelow})}\qquad&\leq \tfrac12\left[\big(1+G_t\hat g_t^{t'}\big)^2 |\d(f+h\circ F^{t'}_t)|^2-(1-G_t \hat g_t^{t'})^{2}(|\d f |^2+|\d (h\circ F_{t}^{t'})|^2)\right]\\
&= \big(1+G_t\hat g_t^{t'}\big)^2\la \d f,\d(h\circ F^{t'}_t)\ra+2 ( |\d f|^2+ |\d (h \circ F_t^{t'})|^2 )\,G_t \hat g_t^{t'}\\
&\leq \big(1+G_t\hat g_t^{t'}\big)^2\la \d f,\d(h\circ F^{t'}_t)\ra+2 ( |\d f|^2+ |\d h|^2 \circ F_t^{t'} G_t^2 )\,G_t \hat g_t^{t'}.
\end{split}
\]
Now notice that since we already proved the first in \eqref{eq:nuovoclaim}, as $t'\to t$ we have that  $\la \d f,\d(h\circ F^{t'}_t)\ra\to \la\d f,\d h\ra$ in $L^0(\X)$. Moreover, from \eqref{eq:contl0}, we have that $|\d h|^2 \circ F_t^{t'} \to |\d h|^2$ in $L^0(\X)$ as $t' \to t$. Thus taking \eqref{eq:ghat} into account  we just proved that
\[
\big(\la \d(f\circ F_{t'}^t),\d h\ra  \circ F_t^{t'}- \la\d f,\d h\ra\big)^+\to0\qquad in\ L^0(\X).
\]
A very similar argument gives that the negative part of the same quantity as above goes to 0 in  $ L^0(\X)$, hence the claim \eqref{eq:st12}, and thus the second in \eqref{eq:nuovoclaim}, follows.

Now we claim that
\begin{equation}
\label{eq:nuovoclaim2}
\lim_{s'\to s}\d (f\circ F_s^{s'})=\d f=\lim_{t'\to t}\d (f\circ F_{t'}^{t}),\qquad in\ L^0(T^*\X),\  \forall t,s\in[0,1].
\end{equation}
Indeed
\[
\begin{split}
|\d(f\circ F_{t'}^{t})-\d f|^2\circ F_t^{t'}&\leq |\d f|^2 |\d F_{t'}^t|^2\circ F_t^{t'}+|\d f|^2\circ F_t^{t'}-2\la \d (f\circ F_{t'}^t),\d f\ra\circ F_t^{t'}\\
\text{(by \eqref{eq:estdFt3})}\qquad&\leq  |\d f|^2 \big(1+G_t\,\hat g_t^{t'}\big)^2+|\d f|^2\circ F_t^{t'}-2\la \d (f\circ F_{t'}^t),\d f\ra\circ F_t^{t'}\\
\text{(by Lemma \ref{le:contcont}, \eqref{eq:st12}, and \eqref{eq:ghat})}\qquad&\qquad \to\quad  |\d f|^2+|\d f|^2-2|\d f|^2=0,
\end{split}
\]
the convergence being in $L^0(\X)$ as $t'\to t$. The second in \eqref{eq:nuovoclaim2} follows using again Lemma \ref{le:contcont} to conclude that $|\d(f\circ F_{t'}^{t})-\d f|^2\to 0$ in $L^0(\X)$. The first in \eqref{eq:nuovoclaim2} is proved analogously.

Now observe that from the group property \eqref{eq:groupRLF} and the chain rule \eqref{eq:chain_rule_diff_LL} we get
\[
\begin{split}
|\d(f\circ F_{t'}^{s'})-\d (f\circ F_t^s)|&\leq|\d(f\circ F_{s}^{s'}\circ F_{t'}^{s})-\d (f\circ F_{t'}^s)|+|\d(f\circ F_{t}^{s}\circ F_{t'}^{t})-\d (f\circ F_t^s)|\\
&\leq |\d(f\circ F_{s}^{s'})-\d f|\circ F_{t'}^s\,|\d F_{t'}^s|+|\d(f\circ F_{t}^{s}\circ F_{t'}^{t})-\d (f\circ F_t^s)|\\
\text{(by \eqref{eq:estdFt4})}\qquad&\leq G_{t'}\,|\d(f\circ F_{s}^{s'})-\d f|\circ F_{t'}^s+|\d(f\circ F_{t}^{s}\circ F_{t'}^{t})-\d (f\circ F_t^s)|.
\end{split}
\]
Hence
\[
\begin{split}
\sfd_{L^0(T^*\X)}(\d(f\circ F_{t'}^{s'}),\d (f\circ F_t^s))&=\sfd_{L^0}(|\d(f\circ F_{t'}^{s'})-\d (f\circ F_t^s)|,0)\\
&\leq \sfd_{L^0}( G_{t'}\,|\d(f\circ F_{s}^{s'})-\d f|\circ F_{t'}^s,0)\\
&\qquad\qquad+\sfd_{L^0}(|\d(f\circ F_{t}^{s}\circ F_{t'}^{t})-\d (f\circ F_t^s)|,0)\\
\text{(by \eqref{eq:unifsuperlgt}, Lemma \ref{le:unifcont} and \eqref{eq:equil0})}\qquad&\leq \Omega\big(\sfd_{L^0}( |\d(f\circ F_{s}^{s'})-\d f|,0)\big)\\
&\qquad\qquad+\sfd_{L^0}(|\d(f\circ F_{t}^{s}\circ F_{t'}^{t})-\d (f\circ F_t^s)|,0)\\
&=  \Omega\big(\sfd_{L^0(T^*\X)}( \d(f\circ F_{s}^{s'}),\d f)\big)\\
&\qquad\qquad+\sfd_{L^0(T^*\X)}(\d((f\circ F_{t}^{s})\circ F_{t'}^{t}),\d (f\circ F_t^s))
\end{split}
\]
and the conclusion follows from  \eqref{eq:nuovoclaim2} and the fact that $f\circ F_t^s$ is Lusin--Lipschitz.
\end{proof}
We then have the following basic regularity result:
\begin{proposition}\label{prop:dFtot2}
Let $(v_t)\in C([0,1],L^0(T\X))$. Then $t\mapsto V_t:=\d F_0^t(v_t)\in L^0(T\X)$ is also continuous. Moreover, the assignment $(v_t)\mapsto (V_t)$ from $C([0,1],L^0(T\X))$ to itself is continuous, invertible, with continuous inverse.

Similarly, if $(v_t)\in L^0([0,1],L^0(T\X))$ the map $t\mapsto V_t:=\d F_0^t(v_t)\in L^0(T\X)$ is an element of $L^0([0,1],L^0(T\X))$ and the assignment $(v_t)\mapsto (V_t)$ from $L^0([0,1],L^0(T\X))$ to itself is continuous, invertible, with continuous inverse.
\end{proposition}
\begin{proof}\ \\
\noindent{\sc Step 1.} We claim that
\begin{equation}
\label{eq:constcont}
\forall v\in L^0(T\X)\text{ the map }[0,1]^2\ni (t,s)\mapsto \d F_t^s(v)\in L^0(T\X)\text{ is continuous}
\end{equation}
and we shall prove this by arguing as in Lemma \ref{lem:contd} above.  Fix $v\in L^0(T\X)$. From Lemma \ref{lem:contd} we see that for $f$ Lusin--Lipschitz  the map $[0,1]^2 \ni(t,s)\mapsto \d (f\circ F_t^s)(v)= \d f(\d F_t^s(v))\circ F_t^s\in L^0(\X)$ is continuous. Taking into account Lemma \ref{le:contcont} we deduce that also $[0,1]^2 \ni (t,s)\mapsto \d f(\d F_t^s(v))\in L^0(\X)$ is continuous.  We now claim that
\begin{equation}
\label{eq:costwcont}
[0,1]^2\ni (t,s)\mapsto\la \d F_t^s(v),z\ra\in L^0(\X)\quad\text{ is continuous for every $z\in L^0(T\X)$}
\end{equation}
 and since differentials of Lusin--Lipschitz maps are dense in $L^0(T^*\X)$, this follows by what already established and the  uniform estimate
\[
\begin{split}
\sfd_{L^0}\big(\la \d F_t^s(v),z\ra\,,\,\la \d F_t^s(v),z'\ra\big)&\leq\sfd_{L^0}\big(|\d F_t^s(v)||z-z'|,0\big)\\
\text{(by \eqref{eq:pointwise_norm_estimate},\eqref{eq:estdFt4})}\qquad&\leq\sfd_{L^0}\big(G_t|v|\circ F_s^t|z-z'|,0\big)\\
\text{(by Lemma \ref{le:unifcont} and \eqref{eq:unifsuperlgt})}\qquad&\leq\Omega\big(\sfd_{L^0}\big(|v|\circ F_s^t|z-z'|,0\big)\big)\\
\text{(by Lemma \ref{le:unifcont})}\qquad&\leq\Omega\big(\sfd_{L^0}\big(|z-z'|,0\big)\big),
\end{split}
\]
valid for any $t,s\in[0,1]$ (the resulting modulus of continuity will depend on $|v|$, but this is not an issue to get \eqref{eq:costwcont} - notice also that in applying Lemma \ref{le:unifcont} in the last step we used the fact that $\{|v|\circ F_s^t>c\}=(F_s^t)^{-1}(\{|v|>c\})$ and \eqref{eq:ftmmp}). Now we claim that
\begin{equation}
\label{eq:contttp}
\lim_{s'\to s}\d F_{s}^{s'}(v)=\lim_{t'\to t}\d F_{t'}^t(v)=v\qquad in\ L^0(T\X).
\end{equation}
For $t,t'\in[0,1]$ we put, as in Lemma \ref{lem:contd} above, $\hat g_t^{t'}:=\int_{t\vee t'}^{t\wedge t'} g_r\circ F_t^r\,\d r\in L^0(\X)$ and recall that \eqref{eq:ghat} holds.  Thus 
\[
\begin{split}
|\d F_{t'}^t(v)-v|^2&\leq\big(|v|\,|\d F_{t'}^t|\big)^2\circ  F_t^{t'}+|v|^2-2\la v,\d F_{t'}^t(v)\ra\\
\text{(by \eqref{eq:estdFt3})}\qquad&\leq |v|^2\circ  F_t^{t'}\,(1+G_t\hat g_t^{t'})^2+|v|^2-2\la v,\d F_{t'}^t(v)\ra\\
\text{(by Lemma \ref{le:contcont} and \eqref{eq:ghat}, \eqref{eq:costwcont})}\qquad&\qquad \to\quad  |v|^2+ |v|^2-2|v|^2=0.
\end{split}
\]
This proves the second in \eqref{eq:contttp}. The first follows by similar arguments.

Now recall the chain rule \eqref{eq:chain_rule_diff_LL} to get that  
\[
\begin{split}
|\d F_{t'}^{s'}(v)-\d F_t^s(v)|&\leq|\d F_{t}^{s'}(\d F_{t'}^t(v)-v)|+|\d F_{s}^{s'}(\d F_t^s(v))-\d F_t^s(v)|\\
\text{(by \eqref{eq:pointwise_norm_estimate},\eqref{eq:estdFt4})}\qquad&\leq  \big(G_t|\d F_{t'}^t(v)-v|\big)\circ F_{s'}^t+|\d F_{s}^{s'}(\d F_t^s(v))-\d F_t^s(v)|.
\end{split}
\]
Therefore
\[
\begin{split}
\sfd_{L^0(T^*\X)}(\d F_{t'}^{s'}(v),\d F_t^s(v))&=\sfd_{L^0}(|\d F_{t'}^{s'}(v)-\d F_t^s(v)|,0)\\
&\leq \sfd_{L^0}\big(\big(G_t|\d F_{t'}^t(v)-v|\big)\circ F_{s'}^t,0\big)+\sfd_{L^0}\big(|\d F_{s}^{s'}(\d F_t^s(v))-\d F_t^s(v)|,0\big)\\
\text{(by \eqref{eq:unifsuperlgt}, Lemma \ref{le:unifcont} and \eqref{eq:equil0})}\quad&\leq \Omega\big(\sfd_{L^0}(|\d F_{t'}^t(v)-v|,0)\big)+\sfd_{L^0}\big(|\d F_{s}^{s'}(\d F_t^s(v))-\d F_t^s(v)|,0\big)\\
&= \Omega\big(\sfd_{L^0(T\X)}(\d F_{t'}^t(v),v)\big)+\sfd_{L^0(T\X)}\big(\d F_{s}^{s'}(\d F_t^s(v)),\d F_t^s(v)\big).
\end{split}
\]
Thus our claim \eqref{eq:constcont} follows from \eqref{eq:contttp}.

\noindent{\sc Step 2.} Let $(v_t)\in C([0,1],L^0(T\X))$ and notice that 
\[
\begin{split}
|\d F_0^s(v_s)-\d F_0^t(v_t)|&\leq |\d F_0^s(v_s-v_t)|+|\d F_0^s(v_t)-\d F_0^t(v_t)|\\
\text{(by  \eqref{eq:estdFt4})}\qquad&\leq (G|v_s-v_t|)\circ F_{s}^0+ |\d F_0^s(v_t)-\d F_0^t(v_t)|,\qquad\mm-a.e.
\end{split}
\]
for every $t,s\in[0,1]$. Thus, as before, \eqref{eq:equil0}, \eqref{eq:unifO}, and the continuity of $(v_t)$ yield $(G|v_s-v_t|)\circ F_{s}^0\to 0$ in $L^0(\X)$ as $s\to t$, hence recalling \eqref{eq:constcont} the continuity of $t\mapsto \d F_0^t(v_t)\in L^0(T\X)$ follows.

The continuity of the assignment $(v_t)\mapsto (\d F_0^t(v_t))$ will follow if we show that for some modulus of continuity  $\Omega$ we have
\begin{equation}
\label{eq:unifcontdF}
\sfd_{L^0(T\X)}\big(\d F_0^t(v),\d F_0^t(z)\big)\leq\Omega\big(\sfd_{L^0(T\X)}(v,z)\big),\qquad\forall t\in[0,1],\ v,z\in L^0(T\X).
\end{equation}
To see this,  notice that \eqref{eq:equil0} and \eqref{eq:estdFt4} give
\[
\begin{split}
\sfd_{L^0(T\X)}(\d F_0^t(v),\d F_0^t(z))&=\sfd_{L^0}(|\d F_0^t(v-z)|,0)\leq   \sfd_{L^0}((G|v-z|)\circ F_t^0,0)\\
\text{(by \eqref{eq:equil0},\eqref{eq:unifO})}\qquad&\leq \Omega\big( \sfd_{L^0}(|v-z|,0)\big)= \Omega\big( \sfd_{L^0(T\X)}(v,z)\big).
\end{split}
\]

Observe now that the estimate \eqref{eq:unifcontdF} also tells that the map $(v_t)\mapsto (\d F_0^t(v_t))$ is uniformly continuous as a map from $C([0,1],L^0(T\X))$ with the $L^0([0,1],L^0(T\X))$-topology to $L^0([0,1],L^0(T\X))$. By the density of $C([0,1],L^0(T\X))$ in $L^0([0,1],L^0(T\X))$ we deduce that $(v_t)\mapsto (\d F_0^t(v_t))$ can be (uniquely) extended to a linear and continuous map from $L^0([0,1],L^0(T\X))$ to itself, and again \eqref{eq:unifcontdF} also ensures that such extension is still given by $(v_t)\mapsto (\d F_0^t(v_t))$. 

Using the result \eqref{eq:constcont} obtained in Step 1, we can prove that $(V_t)\mapsto (\d F_t^0(V_t))$ has the same continuity properties we established for $(v_t)\mapsto (\d F_0^t(v_t))$ and since these two maps are one the inverse of the other (by the chain rule \eqref{eq:chain_rule_diff_LL}), the proof is finished.
\end{proof}

The source and target spaces of the map $(v_t)\mapsto (\d F_0^t(v_t))$, albeit formally equal,  have different roles in the theory and will be treated differently: we shall think at the source space as a family of vector fields defined at fixed base points, i.e.\ $(v_t)$ in the source space will be thought of as a collection of maps of the form $t\mapsto v_t(x)``\in T_x\X"$. On the other hand, we shall think at elements $(V_t)$ of the target space as a family of vector fields defined along flow lines, i.e.\ $(V_t)$ in the target space will  be thought of as a collection of maps of the form $t\mapsto V_t(F_0^t(x))``\in T_{F_0^t(x)}\X"$.

This distinction is particularly relevant in the case of the spaces $L^0$ and  to help keeping this in mind we shall denote the source space as $ L^0_{fix}([0,1],L^0(T\X))$ (in  short $L^0_{fix}$) and the target as  $ L^0_{var}([0,1],L^0(T\X))$ (in short $L^0_{var}$). The spaces $L^0_{fix}$ and $L^0_{var}$ are equipped with two different structures as modules over $L^0(\X)$:
\begin{equation}
\label{eq:l0fixprod}
\text{for $(v_t)\in L^0_{fix}$ and $f\in L^0(\X)$ we define $f(v_t)\in L^0_{fix}$ as $t\mapsto fv_t$}
\end{equation}
and
\begin{equation}
\label{eq:l0varprod}
\text{for $(V_t)\in L^0_{var}$ and $f\in L^0(\X)$ we define $f\times (V_t)\in L^0_{var}$ as $t\mapsto f\circ F_{t}^0V_t$.}
\end{equation}
Notice that by \eqref{eq:l0lindf} we see that these two products are conjugated via $\d F$ (here intended as the map sending $t\mapsto v_t$ to $t\mapsto \d F_0^t(v_t)$), i.e.
\[
 \d F\big(f(v_t)\big)=f\times \d F\big((v_t)\big).
\]
Notice also that $L^0_{fix}$ and $L^0_{var}$ are \emph{not} $L^0(\X)$-normed modules, as there is not really a $L^0(\X)$-valued pointwise norm on them (they could be endowed with the structure of $L^0([0,1]\times\X)$-normed modules, but we are not interested in doing this).

\bigskip

In what follows we shall typically use lowercase letters $(v_t)$ to denote an element of $L^0_{fix}$ and uppercase ones $(V_t)$ for elements of $L^0_{var}$. Also, typically $(v_t)$  and $(V_t)$ are related by
\begin{equation}
\label{eq:vtVt}
V_t=\d F_0^t(v_t)\quad \qquad\qquad\text{ and equivalently }\qquad\qquad v_t=\d F_{t}^0(V_t).
\end{equation}

We are now ready to discuss time integrability/regularity for vector fields in $L^0_{fix}$ and the related concept of derivative in time. In fact, given the discussion made in Section \ref{se:intmod}, and in particular Definition \ref{def:spacesv}, the following definitions are quite natural:
\begin{definition}[Some spaces of vectors at `fixed points']\label{def:fixedspace}  We shall denote by
\begin{itemize}
\item[i)] $L^p_{fix}([0,1],L^0(T\X))\subset L^0_{fix}([0,1],L^0(T\X))$ (or simply  $L^p_{fix}$), $p\in[1,\infty]$, the space $L^p_\H$ for $\H:=L^0(T\X)$.
\item[ii)] $W^{1,2}_{fix}([0,1],L^0(T\X))\subset L^0_{fix}([0,1],L^0(T\X))$ (or simply  $W^{1,2}_{fix}$) the space $W^{1,2}_\H$  for $\H:=L^0(T\X)$.
\item[iii)] $AC^2_{fix}([0,1],L^0(T\X))\subset C([0,1],L^0(T\X))$ (or simply  $AC^{2}_{fix}$) the space $AC^2_\H$ for  $\H:=L^0(T\X)$.
\end{itemize}
\end{definition}
Recall from Section \ref{se:intmod} that the elements $(v_t)$ of $W^{1,2}_{fix}$ come with a natural notion of derivative $(\dot v_t)\in L^2_{fix}$, that the space $L^p_{fix}$ can be characterized as the subspace of $L^0_{fix}$ of those $(v_t)$'s for which the pointwise norm
\[
|(v_t)|_{L^p_{fix}}(x):=|(|v_t|)|_{L^p}(x)=\||v_\cdot|(x)\|_{L^p(0,1)}
\]
is finite $\mm$-a.e., and that $W^{1,2}_{fix}$ comes with the pointwise norm
\[
|(v_t)|_{W^{1,2}_{fix}}^2:=|(v_t)|_{L^2_{fix}}^2+|(\dot v_t)|_{L^2_{fix}}^2.
\]
Also, $L^p_{fix},W^{1,2}_{fix},AC^2_{fix}$ are all complete w.r.t.\ the corresponding distances
\[
\begin{split}
\sfd_{L^p_{fix}}\big((v_t),(z_t)\big)&:=\sfd_{L^0}(|(v_t-z_t)|_{L^p_{fix}},0),\\
\sfd_{W^{1,2}_{fix}}\big((v_t),(z_t)\big)&:=\sfd_{L^0}(|(v_t-z_t)|_{W^{1,2}_{fix}},0),\\
\sfd_{AC^2_{fix}}\big((v_t),(z_t)\big)&:=\sfd_{L^0}(|(v_t-z_t)|_{W^{1,2}_{fix}},0)+\sup_{t\in[0,1]}\sfd_{L^0(T\X)}(v_t,z_t).
\end{split}
\]
Finally, it is clear that the product defined in \eqref{eq:l0fixprod} gives $L^p_{fix},W^{1,2}_{fix},AC^2_{fix}$ the structure of module over $L^0(\X)$ and that with the pointwise norms defined above the spaces $L^p_{fix}$  and $W^{1,2}_{fix}$ are $L^0(\X)$-normed modules.

\bigskip

We now turn to the corresponding notions for vector fields in $L^0_{var}$. In order to justify the definitions we are going to give in a moment, let us illustrate the situation in the case of  a smooth manifold  $M$ and flow of a smooth vector field. In this case, a vector field $(v_t)\in L^0_{fix}$ belongs to $L^p_{fix}$ if and only if for a.e.\ $x$ the curve  $t\mapsto v_t(x)\in T_xM\sim \R^d$ is in $L^p([0,1],\R^d)$. Equivalently, this is the same as to ask that $t\mapsto |v_t|(x)\in\R$ is in  $L^p(0,1)$. 

Now, we have said that we want to think of elements $(V_t)\in L^0_{var}$ as collections of vector fields defined along the flow lines, i.e.\ as the collection, for a.e.\ $x\in M$, of the vector fields $t\mapsto V_t(F_0^t(x))\in T_{F_0^t(x)}M$. Therefore, and by analogy with the `fixed' case, their natural pointwise $L^p$-norm should be given by  the $L^p$ norm of $t\mapsto |V_t|(F_t(x))$.

For the case of Sobolev vector fields, the derivative $\dot v_t(x)$, which is computed in the fixed tangent space $T_xM$, should be replaced by the covariant derivative of the vector field $t\mapsto V_t(F_0^t(x))\in T_{F_t(x)}M$ along the curve $t\mapsto F_t(x)$. By direct computation, if  $(t,x)\mapsto v_t(x)$ is smooth in $t,x$ and $V_t(x):=\d F_0^t(v_t)(x)$ it is not hard to check that such covariant derivative is given by 
\[
\nabla_{\dot F_0^t(x)}V_t(F_0^t(x))=(\d F_0^t(\dot v_t)+\nabla_{V_t}w_t)(F_0^t(x)).
\]
Notice also that when dealing with vector fields defined on the whole manifold (rather than along a single flow line) one typically speaks of `convective' derivative, rather than `covariant' one (think e.g.\ to the setting of fluid dynamics).

We now turn to the actual definitions; the subsequent discussion will make clear the link with what just said.
\begin{definition}[Some spaces of vectors at `variable points']\label{def:Dt}  We shall denote by
\begin{itemize}
\item[i)] $L^p_{var}([0,1],L^0(T\X))\subset L^0_{var}([0,1],L^0(T\X))$ (or simply  $L^p_{var}$), $p\in[1,\infty]$, the space of vector fields $(V_t)$ of the form $V_t=\d F_0^t(v_t)$ for a.e.\ $t$ for some $(v_t)\in L^p_{fix}$. We also define $\mm$-a.e.\ the quantity
\[
|(V_t)|_{L^p_{var}}:=\||V_t|\circ F_0^t\|_{L^p(0,1)}.
\]
\item[ii)] $W^{1,2}_{var}([0,1],L^0(T\X))\subset L^0_{var}([0,1],L^0(T\X))$ (or simply  $W^{1,2}_{var}$) the space of vector fields $(V_t)$ of the form $V_t=\d F_0^t(v_t)$ for a.e.\ $t$ for some $(v_t)\in W^{1,2}_{fix}$. In this case  we also define  the convective derivative $(D_tV_t)\in L^0_{var}([0,1],L^0(T\X))$ as
\begin{equation}
\label{eq:defDt}
D_tV_t:=\d F_0^t(\dot v_t)+\nabla_{V_t}w_t,\qquad a.e.\ t\in[0,1].
\end{equation}
Also, we consider the $\mm$-a.e.\ defined quantity
\[
|(V_t)|_{W^{1,2}_{var}}^2:=|(V_t)|^2_{L^2_{var}}+|(D_tV_t)|^2_{L^2_{var}}=\int_0^1|V_t|^2\circ F_0^t+|D_tV_t|^2\circ F_0^t\,\d t.
\]
\item[iii)] $AC^2_{var}([0,1],L^0(T\X))\subset C([0,1],L^0(T\X))$ (or simply  $AC^{2}_{var}$) the space of vector fields $(V_t)$ of the form $V_t=\d F_0^t(v_t)$ for every $t\in[0,1]$ for some $(v_t)\in AC^{2}_{fix}$.
\end{itemize}
\end{definition}
Notice that Proposition \ref{prop:dFtot2} above ensures that $(D_tV_t)$ is indeed an element of $L^0_{var}([0,1],L^0(T\X))$ and that if $(v_t)\in AC^2_{fix}\subset C([0,1],L^0(T\X))$, then $t\mapsto \d F_0^t(v_t)$ is also an element of $C([0,1],L^0(T\X))$, so that $AC^2_{var}$ is actually a subset of $C([0,1],L^0(T\X))$. It is also clear that the spaces defined above are vector spaces and that  $(V_t)\mapsto (D_tV_t)$ is linear.

Let us remark that by  \eqref{eq:pointwise_norm_estimate}, \eqref{eq:estdFt4} it follows that for any $v\in L^0(T\X)$ and $t\in[0,1]$, putting $V:=\d F_0^t(v)\in L^0(T\X)$, so that $v=\d F_t^0(V)$ by the group property \eqref{eq:groupRLF} and the chain rule \eqref{eq:chain_rule_diff_LL}, we have
\begin{equation}
\label{eq:normVtvt2}
|v|\leq G |V|\circ F_0^t\qquad\text{and}\qquad |V|\circ F_0^t\leq G|v|\qquad\mm-a.e.,
\end{equation}
for some non-negative function $G\in L^0(\X)$ depending only on $\X$ and $(w_t)$ (and in particular independent of $v,t$). It follows that  $(V_t)\in L^0_{var}$ belongs to $L^p_{var}$ if and only if  the quantity $|(V_t)|_{L^p_{var}}$ is finite $\mm$-a.e.. Then arguing as in the `fixed' case, it is easy to see that the distance
\[
\sfd_{L^p_{var}}\big((V_t),(Z_t)\big):=\sfd_{L^0}(|(V_t-Z_t)|_{L^p_{var}},0)
\]
is lower semicontinuous w.r.t.\ $L^0_{var}$-convergence and thus - by the completeness of $L^0_{var}$ - that it is a complete distance on $L^p_{var}$. Alternatively, completeness of $L^p_{var}$ can be established noticing that for $(v_t)$ and $(V_t)$ as in \eqref{eq:vtVt}, the uniform bound \eqref{eq:normVtvt2} gives
\begin{equation}
\label{eq:compLp}
|(v_t)|_{L^p_{fix}}\leq G|(V_t)|_{L^p_{var}}\qquad\text{ and }\qquad|(V_t)|_{L^p_{var}}\leq G|(v_t)|_{L^p_{fix}}
\end{equation}
$\mm$-a.e.. Thus for $(v_t),(z_t)\in L^p_{fix}$ and the corresponding $(V_t),(Z_t)\in L^p_{var}$ as in  \eqref{eq:vtVt} we have
\begin{equation}
\label{eq:Lpsame}
\begin{split}
\sfd_{L^p_{fix}}\big((v_t),(z_t)\big)&=\sfd_{L^0}\big(|(v_t-z_t)|_{L^p_{fix}},0\big)\\
\text{(by \eqref{eq:compLp})}\qquad&\leq \sfd_{L^0}\big(G|(V_t-Z_t)|_{L^p_{var}},0\big)\\
\text{(by \eqref{eq:unifO})}\qquad&\leq \Omega\big(\sfd_{L^0}\big(|(V_t-Z_t)|_{L^p_{var}},0\big)\big)= \Omega\big(\sfd_{L^p_{var}}\big((V_t),(Z_t)\big)\big) 
\end{split}
\end{equation}
and analogously it holds $\sfd_{L^p_{var}}\big((V_t),(Z_t)\big)\leq\Omega\big( \sfd_{L^p_{fix}}\big((v_t),(z_t)\big)\big)$. This proves that $(v^n_t)\subset L^p_{fix}$ is Cauchy if and only if the corresponding sequence $(V^n_t)\subset L^p_{var}$ is Cauchy, thus the completeness of $L^p_{var}$ follows from that of $L^p_{fix}$.

\bigskip

Analogous considerations are  in place for $W^{1,2}_{var}$, but are a bit harder to prove.  The key point is an analogue of \eqref{eq:compLp} in the Sobolev case, which is established in the following lemma:
\begin{lemma}\label{le:covtot} Let $(v_t)\in W^{1,2}_{fix}$ and $(V_t)\in W^{1,2}_{var}$ be as in \eqref{eq:vtVt}.  Then
\begin{equation}
\label{eq:covtot}
\begin{split}
|(V_t)|_{W^{1,2}_{var}}&\leq G|(v_t)|_{W^{1,2}_{fix}}\qquad\text{and}\qquad 
|(v_t)|_{W^{1,2}_{fix}}\leq G|(V_t)|_{W^{1,2}_{var}}\qquad\mm-a.e.
\end{split}
\end{equation}
and
\begin{equation}
\label{eq:Vtinfty}
|(V_t)|_{L^\infty_{var}}\leq G|(V_t)|_{W^{1,2}_{var}}\qquad\mm-a.e..
\end{equation}
\end{lemma}
\begin{proof} From the very definition \eqref{eq:defDt} we see that $\mm$-a.e.\ and for a.e.\ $t$ it holds
\[
|D_tV_t|\circ F_0^t\stackrel{\eqref{eq:estdFt4}}\leq G |\dot v_t|+|\nabla w_t|\circ F_0^t|V_t|\stackrel{\eqref{eq:normVtvt2}}\leq G \big(|\dot v_t|+|\nabla w_t|\circ F_0^t|v_t|) \stackrel{\eqref{eq:perlinfty1}}\leq G(|\dot v_t|+|\nabla w_t|\circ F_0^t |(v_s)|_{W^{1,2}_{fix}}).
\]
Squaring and integrating in $t$ we obtain
\[
|(D_tV_t)|_{L^2_{var}}^2\leq G\Big(|(\dot v_t)|^2_{L^2_{fix}}+|( v_t)|^2_{W^{1,2}_{fix}}\,|(|\nabla w_t|\circ F_0^t)|_{L^2}^2\Big)\stackrel{\eqref{eq:estnw2}}\leq G|( v_t)|^2_{W^{1,2}_{fix}},
\]
so that the first in \eqref{eq:covtot} follows taking also \eqref{eq:compLp} into account.

For the second in \eqref{eq:covtot} we notice that   \eqref{eq:defDt} gives $\dot v_t=\d F_t^0(D_tV_t-\nabla_{V_t}w_t)$, thus
\begin{equation}
\label{eq:vdotbound}
|\dot v_t|\stackrel{\eqref{eq:pointwise_norm_estimate} }\leq \Big(|\d F_{t}^0|\big(|D_tV_t|+|V_t|\,|\nabla w_t|\big)\Big)\circ F_0^t\stackrel{\eqref{eq:estdFt4}}\leq G\Big(|D_tV_t|\circ F_0^t+|V_t|\circ F_0^t\,|\nabla w_t|\circ F_0^t\Big)
\end{equation}
$\mm$-a.e.\ for a.e.\ $t\in[0,1]$. Integrating in $t$ we deduce that  $\mm$-a.e.\ we have
\[
|(\dot v_t)|_{L^1_{fix}}\leq  G\Big(|(D_tV_t)|_{L^2_{var}}+|(V_t)|_{L^2_{var}}\,|(|\nabla w_t|\circ F_0^t)|_{L^2}\,\Big)\stackrel{\eqref{eq:estnw2}}\leq G|(V_t)|_{W^{1,2}_{var}}.
\]
Using this bound in conjunction with the first inequality in \eqref{eq:perlinfty1} and  \eqref{eq:compLp} we obtain
\[
|(V_t)|_{L^\infty_{var}}\leq G|(v_t)|_{L^\infty_{fix}}\leq G\big(|(V_t)|_{L^2_{var}}+|(V_t)|_{W^{1,2}_{var}}\big)\leq G|(V_t)|_{W^{1,2}_{var}}\qquad\mm-a.e.,
\]
which is \eqref{eq:Vtinfty}. Plugging this inequality in \eqref{eq:vdotbound} we see that 
\[
|\dot v_t| \leq G\big(|D_tV_t|\circ F_0^t+|(V_s)|_{W_{var}^{1,2}}|\nabla w_t|\circ F_0^t\big)\qquad\text{$\mm$-a.e.\   for a.e.\ $t\in[0,1]$.}
\] 
Squaring, integrating in $t$  and recalling \eqref{eq:estnw2} we see that  $ |(\dot v_t)|_{L^2_{fix}}^2\leq G|(V_t)|^2_{W_{var}^{1,2}}$ holds $\mm$-a.e.. Then the second in \eqref{eq:covtot} follows taking \eqref{eq:compLp} into account.
\end{proof}
This lemma ensures in particular that $|(V_t)|_{W^{1,2}_{var}}<\infty$ $\mm$-a.e.\ for any $(V_t)\in W^{1,2}_{var}$. Hence the following is a well defined distance on $W^{1,2}_{var}$:
\[
\sfd_{W^{1,2}_{var}}\big((V_t),(Z_t)\big):=\sfd_{L^0}(|(V_t-Z_t)|_{W^{1,2}_{var}},0).
\]
Then starting from \eqref{eq:covtot}, arguing as we did in \eqref{eq:Lpsame} from \eqref{eq:compLp}, we see that for $(v_t),(z_t)\in W^{1,2}_{fix}$ and the corresponding $(V_t),(Z_t)\in W^{1,2}_{var}$ as in \eqref{eq:vtVt} we have
\[
\sfd_{W^{1,2}_{fix}}\big((v_t),(z_t)\big)\leq\Omega\big(\sfd_{W^{1,2}_{var}}\big((V_t),(Z_t)\big)\big)\qquad\text{ and }\qquad\sfd_{W^{1,2}_{var}}\big((V_t),(Z_t)\big)\leq\Omega\big(\sfd_{W^{1,2}_{fix}}\big((v_t),(z_t)\big)\big).
\]
Therefore, much like in the $L^p$ case, $(v_t^n)\subset W^{1,2}_{fix}$ is Cauchy if and only if the corresponding sequence $(V^n_t)\subset W^{1,2}_{var}$ is Cauchy. Hence  completeness of $W^{1,2}_{var}$ follows from that of $W^{1,2}_{fix}$.

It is then clear that $AC^2_{var}$ is also a complete space.

\subsection{Calculus with the convective derivative}\label{sec:reg_fcircF_t}
The main goal of this section is to establish appropriate calculus rules for the convective derivative that mimic those available in the smooth setting. 

In particular, we shall ultimately prove that  for $(V_t),(W_t)\in AC^2_{var}$ we have $\big(\la V_t,W_t\ra\circ F_0^t\big)\in AC^2_{L^0(\X)}$ with derivative $ \la D_tV_t,W_t\ra \circ F_0^t+ \la V_t,D_tW_t\ra \circ F_0^t$.
This can be read as \emph{compatibility with the metric} of our convective/covariant derivative and will be crucial to obtain uniqueness of parallel transport and preservation of norm.

The proof of this fact  will be obtained following roughly the ideas in \cite{DPG16}, but due to much lower regularity we have at disposal now, things are now more involved. The idea is to  establish regularity for $t\mapsto \la V_t,W_t\ra\circ F_0^t$ by duality, i.e.\ we   first study the regularity of $t\mapsto \d (f\circ F_0^t)$ (Proposition \ref{prop:mainreg}) and later  that of $t\mapsto \la V_t,W_t\ra\circ F_0^t$ (Theorem \ref{thm:leibconv}).

To pursue this program, and in particular in the  first step, we shall use the closure  of the differential on bounded subsets of the Haj\l asz--Sobolev space as discussed in Proposition \ref{prop:closure_tilde_d}. In turn, this will be possible thanks to the following lemma:
\begin{lemma}
\label{prop:fF_tbelongstoHphi} Let $\bar g:=\bar g_0\in L^2_{loc}(\X)$ be given by Proposition \ref{prop:eliadanieleregularity} for $t=0$, $f\in W^{1,2}(\X)$, and $R>0$. Then $f\circ F_0^t\in H_{{\bar g},R}(\X)$ for every $t\in[0,1]$ and there is $C>0$ such that
\begin{equation}
\label{eq:contemb}
\|f\circ F_0^t\|_{H_{{\bar g},R}}\leq C\|f\|_{W^{1,2}},\qquad\forall t\in[0,1].
\end{equation}
\end{lemma}
\begin{proof} Since ${\sf RCD}(K,N)$ spaces are PI spaces, we know from \eqref{eq:local_max_est_sobolev} that for every $R'>0$ there exists a $\mm$-negligible set $N$ and $C>0$ such that
\begin{equation}
\label{eq:damax}
|f(x)-f(y)| \le C \big( M_{4R'}(|D f|)(x) + M_{4R'}(|D f|)(y) \big)\sfd(x,y) 
\end{equation}
holds  for every $x,y \in \X \setminus N$ with  $\sfd(x,y)\le R'$.

Also, recalling \eqref{eq:unifspFl} we see that possibly enlarging $N$, keeping it $\mm$-negligible, for $L:=\|w_t\|_{L^\infty_{t,x}}$ we have
\[
\sfd(F_0^t(x),F_0^t(y))\leq \sfd(x,y)+2tL,\qquad\forall x,y\in \X\setminus N,\ \forall t\in[0,1].
\]
Put  $R':=R+2L$ and $H:=M_{4R'}(|D f|)$ (notice that $H\in L^2(\X,\mm)$ by \eqref{eq:estimate_L^p-strong_locmaxfunct}). Then for every $x,y\in \X\setminus (F_0^t)^{-1}(N)$ with $\sfd(x,y)\leq R$, from \eqref{eq:damax} and \eqref{eq:estimate4_factor} we have  that
\[
\begin{split}
|f(F_0^t(x))-f(F_0^t(y))|&\leq C \big(H(F_0^t(x)) +H(F_0^t(y)) \big)\sfd(F_0^t(x),F_0^t(y)) \\
&\leq C \big(H(F_0^t(x)) + H(F_0^t(y)) \big)e^{\bar g(x)+\bar g(y)}\sfd(x,y)
\end{split}
\]
holds for any $t\in[0,1]$. In other words, since $(F_0^t)^{-1}(N)$ is $\mm$-negligible, we proved that $H\circ F_0^t\in A_{\bar g,R}(f)$. Then the conclusion follows recalling that since $(F_0^t)$ has  bounded compression, we have that $\|H\circ F_0^t\|_{L^2}\leq C'\|H\|_{L^2}$ for some $C'>0$ and every $t\in[0,1]$. 
\end{proof}
We can now study the regularity of $t\mapsto \d (f\circ F_0^t)$:
\begin{proposition}\label{prop:mainreg}
Let  $f\in\test\X$.  Then for every $t\in[0,1]$ the map $s\mapsto \d(f\circ F_t^s)$ is in $AC^2([0,1],L^0(T^*\X))$ and there is a Borel negligible set $N\subset[0,1]$, independent of $f$ and $ t\in[0,1]$, such that for every $s\in[0,1]\setminus N$ we have
\begin{equation}
\label{eq:tuttit}
\lim_{h\to 0}\frac{\d(f\circ F_{ t}^{s+h})-\d (f\circ F_{ t}^s)}{h}= \d(\d f(w_t)\circ F_{ t}^s)\qquad\text{ in $L^0(T^*\X)$.}
\end{equation}
\end{proposition}
\begin{proof} \ \\
{\sc Step 1: basic integrability estimates.} We know from Lemma \ref{lem:contd} that $s\mapsto \d (f\circ F_t^s)\in L^0(T^*\X)$ is continuous for any $t\in[0,1]$ and since $\d f(w)\in W^{1,2}(\X)$ for any $w\in L^\infty\cap W^{1,2}_C(T\X)$ (recall \eqref{eq:leibgrad}), the same lemma ensures that for any such $w$ we have that $s\mapsto \d(\d f(w)\circ F_t^s)\in L^0(T^*\X)$ is continuous. Then a simple approximation argument shows that  $s\mapsto \d(\d f(w_s)\circ F_t^s)$ is Borel. 

Now observe that the bound $|\d f(w_t)|\leq\|\d f\|_{L^\infty}|w_t|$, the assumption $(w_t)\in L^2([0,1],L^2(T\X))$ and \eqref{eq:intflo2} ensure that $s\mapsto \d f(w_s)\circ F_t^s$ is in $L^2([0,1],L^0(\X))$. Similarly,  \eqref{eq:techn_comp_cl} and \eqref{eq:estdFt4}  give
\begin{equation}
\label{eq:piuvolte}
|\d(\d f(w_s)\circ F_t^s)|\leq G_t\,|\d (\d f(w_s))|\circ F_t^s\stackrel{\eqref{eq:leibgrad}}\leq G_t\,\big(|{\rm Hess}(f)|\|w_s\|_{L^\infty}+|\nabla w_s||\d f| \big)\circ F_t^s ,
\end{equation}
$\mm$-a.e.\  for every $t,s\in[0,1]$. It is clear from our assumptions on $(w_t)$ and $f$ that $s\mapsto (|{\rm Hess}(f)|\|w_s\|_{L^\infty}+|\nabla w_s||\d f| )$ is in $L^2([0,1],L^2(\X))$, thus from \eqref{eq:intflo2} and the above we see that $s\mapsto \d(\d f(w_s)\circ F_t^s)$ is in $L^2([0,1],L^0(T^*\X))$. 

{\sc Step 2: Sobolev regularity and computation of the derivative.} We now claim that  for any $s_1,s_2\in[0,1]$, $s_1<s_2$ we have
\begin{equation}
\label{eq:conHille}
\d (f\circ F_t^{s_2})-\d (f\circ F_t^{s_1})=\int_{s_1}^{s_2}  \d(\d f(w_r)\circ F_t^r)\,\d r.
\end{equation}
Notice that by \eqref{eq:charAC}  this proves that $s\mapsto \d(f\circ F_t^s)$ is in $AC^2([0,1],L^0(T^*\X))$ with derivative given by   $ \d(\d f(w_s)\circ F_t^s)$. Recalling \eqref{eq:int RLF}, we see that \eqref{eq:conHille} above will follow if we show that for any $s_1,s_2$ as before it holds
\begin{equation}
\label{eq:conHille2}
 \d\Big(\int_{s_1}^{s_2} \d f(w_r)\circ F_t^r\,\d r\Big)=\int_{s_1}^{s_2}  \d(\d f(w_r)\circ F_t^r)\,\d r.
 \end{equation}
To  prove this latter identity we apply Theorem \ref{thm:HilleH} with $\H_1:=L^0(\X)$, $\H_2:=L^0(T^*\X)$, $V:= H_{\bar g,1}(\X)$ equipped with its norm as defined in \eqref{eq:normhphi} (here $\bar g=\bar g_0$ is given by Proposition \ref{prop:eliadanieleregularity}), $L:=\d$, and $v_s:= \d f(w_s)\circ F_t^s$. We already observed that $(v_s)\in  L^2([0,1],\H_1)$ and $(L(v_s))\in L^2([0,1],\H_2)$.   Also,  the trivial bound $\|\d f(w_s)\|_{L^2}\leq \|w_s\|_{L^\infty_{s,x}}\|\d f\|_{L^2}$, the inequality   
\[
\|\d (\d f(w_s))\|_{L^2}\leq \|w_s\|_{L^\infty_{s,x}}\|{\rm Hess}( f)\|_{L^2}+\|\d f\|_{L^\infty}\||\nabla w_s|\|_{L^2},
\]
and again the assumption $(w_t)\in L^2([0,1],W^{1,2}_C(T\X))$ with $|w_t|\in L^\infty([0,1]\times\X)$, give that $s\mapsto \d f(w_s)$ is in $L^1([0,1],W^{1,2}(\X))$. Therefore by Lemma \ref{prop:fF_tbelongstoHphi} above we deduce that $\int_0^1\|v_s\|_{H_{\bar g,1}}\,\d s<\infty$. Finally, Proposition \ref{prop:closure_tilde_d} ensures that assumption $(ii)$ in Theorem  \ref{thm:HilleH} is satisfied, thus we can apply this proposition and deduce that \eqref{eq:conHille2}, and thus \eqref{eq:conHille}, holds.

Now for $f\in\test\X$ and $t\in[0,1]$ let $N(f,t)\subset[0,1]$ be the set of $s\in[0,1]$ for which the limiting relation
\[
\lim_{h\to 0}\frac{\d(f\circ F_t^{s+h})-\d (f\circ F_t^s)}{h}=\d(\d f(w_s)\circ F_t^s),\qquad\text{ in }L^0(T^*\X)
\]
does \emph{not} hold. We claim that $N(f,t)$ does not depend on $t$. To see this, let $t,t'\in[0,1]$ and observe that from the bound $|\d (g\circ F_t^{t'})|\leq G_t |\d  g|\circ F_t^{t'}$ (which follows from \eqref{eq:estdFt4}) and Lemma \ref{le:unifcont} we deduce
\[
\sfd_{L^0}(|\d (g\circ F_t^{t'})|,0)\leq\Omega\big(\sfd_{L^0}(|\d g|\circ F_t^{t'},0)\big)\stackrel{\eqref{eq:equil0}}\leq\Omega\big(\sfd_{L^0}(|\d g|,0)\big).
\]
Applying this to $g:=\frac{f\circ F_{t'}^{s+h}-f\circ F_{t'}^{s}-\d f(w_s)\circ F_{t'}^s}{h}$ we deduce that $N(f,t')\supset N(f,t)$ and the claim follows by the arbitrariness of $t,t'$.

We shall put $N(f):=N(f,0)=N(f,t)$ for any $t\in[0,1]$.  Notice that $N(f)$ is a Borel negligible subset of $[0,1]$ and that for every $t\notin N(f)$ we have $t\notin N(f,t)$ and thus \eqref{eq:tuttit} holds.

{\sc Step 3: the exceptional set can be chosen independent of $f$.} To conclude it is sufficient to prove that for some $N\subset [0,1]$ Borel and negligible we have $N(f)\subset N$ for every $f\in\test\X$.  To this aim we employ a standard idea based on uniform continuity of difference quotients; specifically, we  start claiming that for a.e.\ $t\in[0,1]$ the linear operators $T_{t}^s:\test\X\to L^0(T^*\X)$ defined as
\[
T_{t}^s(f):=\frac1{|s-t|}\int_{t\wedge s}^{t\vee s}\d(\d f(w_r)\circ F_0^r)\,\d r
\]
are uniformly continuous in $s\in[0,1]\setminus\{t\}$ if we equip $\test\X$ with the $W^{2,2}$-norm. To see this notice that
\begin{equation}
\label{eq:st1}
\begin{split}
|T_t^s(f)|&\leq \frac1{|s-t|}\int_{t\wedge s}^{t\vee s}|\d(\d f(w_r)\circ F_0^r)|\,\d r\\
\text{(by \eqref{eq:piuvolte})}\qquad& \leq G\,\frac1{|s-t|}\int_{t\wedge s}^{t\vee s}\,\big(|{\rm Hess}(f)|\|w_r\|_{L^\infty}+|\nabla w_r||\d f| \big)\circ F_0^r \,\d r \qquad\mm-a.e.
\end{split}
\end{equation}
and that
\[
\begin{split}
&\sfd_{L^0}\Big(\frac1{|s-t|}\int_{t\wedge s}^{t\vee s}\,\big(|{\rm Hess}(f)|\|w_r\|_{L^\infty}+|\nabla w_r||\d f| \big)\circ F_0^r \,\d r,0 \Big)\\
&\leq\sfd_{L^0}\Big(\frac{\|w_r\|_{L^\infty_{r,x}}}{|s-t|}\int_{t\wedge s}^{t\vee s}\,|{\rm Hess}(f)| \circ F_0^r \,\d r,0 \Big)+\sfd_{L^0}\Big(\frac1{|s-t|}\int_{t\wedge s}^{t\vee s}\big(|\nabla w_r||\d f| \big)\circ F_0^r \,\d r,0 \Big)\\
&\stackrel{*}\leq\Omega\Big( \Big\|\frac{\|w_r\|_{L^\infty_{r,x}}}{|s-t|}\int_{t\wedge s}^{t\vee s}\,|{\rm Hess}(f)| \circ F_0^r \,\d r \Big\|_{L^2}+\Big\|\frac1{|s-t|}\int_{t\wedge s}^{t\vee s}\big(|\nabla w_r||\d f| \big)\circ F_0^r \,\d r\Big\|_{L^1}\Big)\\
&\stackrel{**}\leq\Omega\Big( \frac{\|w_r\|_{L^\infty_{r,x}}}{|s-t|}\int_{t\wedge s}^{t\vee s}\,\big\||{\rm Hess}(f)| \big\|_{L^2} \,\d r+\frac1{|s-t|}\int_{t\wedge s}^{t\vee s}\big\||\nabla w_r||\d f| \big\|_{L^1} \,\d r\Big)\\
&\leq \Omega\Big(\|w_r\|_{L^\infty_{r,x}} \| |{\rm Hess}(f)|\|_{L^2}+\||\d f|\|_{L^2}{\sqrt{\frac1{|s-t|}\int_{t\wedge s}^{t\vee s}\||\nabla w_r|\|_{L^2}^2\,\d r}}\Big),
\end{split}
\]
having used \eqref{eq:lpl0} in $*$ and \eqref{eq:dabc} in $**$.

Now observe that since by assumption we have that $t\mapsto \||\nabla w_t|\|_{L^2}$  is in $L^2(0,1)$, the Hardy--Littlewood maximal inequality grants that $M(t):=\sup_{s\neq t}\sqrt{\frac1{|s-t|}\int_{t\wedge s}^{t\vee s}\||\nabla w_r|\|_{L^2}^2\,\d r}<\infty$ for a.e. \ $t$, therefore  for $f,f'\in\test\X$, putting for brevity $g:=f-f'$ the above gives
\begin{equation}
\label{eq:unifTts}
\begin{split}
\sfd_{L^0(T^*\X)}(T_{t}^s(f),T_{t}^s(f'))&=\sfd_{L^0}(|T_{t}^s(g)|,0)\\
\text{(by \eqref{eq:st1})}\qquad&\leq \sfd_{L^0}\Big(G\,\frac1{|s-t|}\int_{t\wedge s}^{t\vee s}\,\big(|{\rm Hess}(g)|\|w_r\|_{L^\infty}+|\nabla w_r||\d g| \big)\circ F_0^r \,\d r,0 \Big)\\
\text{(by \eqref{eq:unifO})}\qquad&\leq \Omega\Big(\sfd_{L^0}\Big(\frac1{|s-t|}\int_{t\wedge s}^{t\vee s}\,\big(|{\rm Hess}(g)|\|w_r\|_{L^\infty}+|\nabla w_r||\d g| \big)\circ F_0^r \,\d r,0 \Big)\Big)\\
&\leq \Omega\Big(\|w_r\|_{L^\infty_{r,x}} \| |{\rm Hess}(g)|\|_{L^2}+\||\d g|\|_{L^2}M(t)\Big)\\
&\leq \Omega\Big(\big(M(t)+1\big)\|f-f'\|_{W^{2,2}}\Big),
\end{split}
\end{equation}
proving the desired uniform continuity for a.e.\ $t\in[0,1]$. Similarly, for $f,f'\in\test\X$ and $t\in[0,1]$ such that $w_t\in L^\infty\cap W^{1,2}_C(T\X)$, starting from
\[
\begin{split}
|\d(\d f(w_t)\circ F_0^t)-\d(\d f'(w_t)\circ F_0^t)|&=|\d(\d (f-f')(w_t)\circ F_0^t)|\\
&\leq G\big(|{\rm Hess}(f-f')|\|w_t\|_{L^\infty}+|\nabla w_t|\,|\d(f-f')|\big),
\end{split}
\]
and arguing as in \eqref{eq:unifTts}, we get
\begin{equation}
\label{eq:unifdd}
\begin{split}
\sfd_{L^0(T^*\X)}\big(\d(\d f(w_t)\circ F_0^t)&,\d(\d f'(w_t)\circ F_0^t)\big)\leq \Omega\big(\big(\|w_t\|_{L^\infty}+\|w_t\|_{W^{1,2}_C}\big)\|f-f'\|_{W^{2,2}}\big).
\end{split}
\end{equation}
Now let $(f_{n})\subset\test\X$ be dense in $\test\X$ with respect to the separable norm of $W^{2,2}(\X)$. Define  $N\subset[0,1]$ as
\[
N:=\bigcup_n N(f_n)\cup \big\{t\ :\ M(t)=+\infty \,\text{or} \ w_t\notin L^\infty\cap W^{1,2}_C(T\X)\big\}
\]
and notice that it is  Borel and negligible. Let $t\notin N$,  $f\in\test\X$  and notice that \eqref{eq:unifTts} and \eqref{eq:unifdd} ensure that for any $n\in\N$ we have
\[
\begin{split}
&\sfd_{L^0(T^*\X)}\big(T_t^s(f),\d (\d f(w_t)\circ F_0^t)\big)\\
&\leq   \Omega\Big(\big(M(t)+1+\|w_t\|_{L^\infty}+\|w_t\|_{W^{1,2}_C}\big)\|f-f_n\|_{W^{2,2}}\Big)+\sfd_{L^0(T^*\X)}\big(T_t^s(f_n),\d (\d f_n(w_t)\circ F_0^t)\big),
\end{split}
\]
thus letting first $s\to t$ and then taking the infimum over $n\in\N$ we conclude that
\[
\lim_{s\to t}T_t^s(f)=\d (\d f(w_t)\circ F_0^t),\qquad in\ L^0(T^*\X).
\]
By the definition of $T_t^s(f)$ and \eqref{eq:conHille}, this proves that $t\notin N(f,t)=N(f)$, i.e.\ concludes the proof.
\end{proof}

Before coming to the main result of the section, we need the following statement about equi-continuity of quadratic forms on $L^0(T\X)$:
\begin{proposition}\label{prop:Qt}
Let $B_i:L^0(T\X)^2\to L^0(\X)$ be a family of $L^0(\X)$-bilinear, symmetric, and continuous maps indexed by $i\in I$, and let $Q_i:L^0(T\X)\to L^0(\X)$ be the associated quadratic forms, i.e.\ $Q_i(v):=B_i(v,v)$. Let $(G_i)_{i\in I}\subset L^0(\X) $ be a family of non-negative functions satisfying \eqref{eq:perunifcont2} and assume that  
\[
|Q_i(v)|\leq G_i|v|^2,\qquad\mm-a.e.,\ \forall v\in L^0(T\X),\ \forall i\in I.
\]
Then the $Q_i$'s are locally equicontinuous, i.e.\
\[
\lim_{v'\stackrel{L^0}\to v}\sup_{i\in I}\sfd_{L^0}(Q_i(v'),Q_i(v))=0,\qquad\forall v\in L^0(T\X).
\]
\end{proposition}
\begin{proof}  Let $\tilde B_i(v,z):=B_{i}(v,z)+G_i\la v,z\ra$ and $\tilde Q_i(v):=\tilde B_i(v,v)$. Then 
\[
0\leq \tilde Q_i(v,v)\leq 2G_i|v|^2,\qquad\mm-a.e.,\ \forall v\in L^0(T\X),\ \forall i\in I.
\]
By the pointwise Cauchy--Schwarz inequality (that is easily seen to hold also in this setting), we deduce $|B_i(v,z)|\leq \sqrt{Q_i(v)Q_i(z)}$ $\mm$-a.e.\ for every $v,z\in L^0(T\X)$ and $i\in I$. Thus for arbitrary $v,v'\in L^0(T\X)$, putting $z:=v'-v$ we have
\[
|Q_i(v')-Q_i(v)|=|Q_i(z)+2B_i(z,v)|\leq 2G_i(|z|^2+|z||v|),\qquad\mm-a.e.,\ \forall i\in I
\]
and therefore, recalling that we are in a position to apply Lemma \ref{le:unifcont}, we get
\[
\begin{split}
\sfd_{L^0}(Q_i(v'),Q_i(v))&\leq \sfd_{L^0}\big(2G_i(|z|^2+|z||v|)\big)\stackrel{\eqref{eq:equihi}}\leq\Omega\big( \sfd_{L^0}\big(|z|^2+|z||v|\big)\big).
\end{split}
\]
To conclude, notice that $1\wedge |z|^2\leq 1\wedge |z|$ and thus
\[
\sfd_{L^0}\big(|z|^2+|z||v|,0\big)\leq \sfd_{L^0}(|z|^2,0)+\sfd_{L^0}(|z||v|,0)\leq  \sfd_{L^0}(|z|,0)+\Omega(\sfd_{L^0}(|z|,0))=\Omega(\sfd_{L^0(T\X)}(v',v)),
\]
where the $\Omega$ appearing here depends on $v$ (it is the one in \eqref{eq:unifO} for $G:=|v|$), but does not depend on $v',z,i$. The conclusion follows.
\end{proof}

\begin{theorem}\label{thm:leibconv} Let $(V_t),(Z_t)\in W^{1,2}_{var}([0,1],L^0(T\X))$. Then $t\mapsto \la V_t,Z_t\ra\circ F_0^t$ belongs to $W^{1,2}([0,1],L^0(\X))$ and for a.e.\ $t\in[0,1]$ we have
\begin{equation}
\label{eq:leibconv}
\frac{\d}{\d t}\big(\la V_t,Z_t\ra\circ F_0^t\big)= \la D_tV_t,Z_t\ra \circ F_0^t+ \la V_t,D_tZ_t\ra \circ F_0^t,\qquad  \mm-a.e..
\end{equation}
If moreover $(V_t),(Z_t)\in AC^{2}_{var}([0,1],L^0(T\X))$, then $t\mapsto \la V_t,Z_t\ra\circ F_0^t$ belongs to $AC^{2}([0,1],L^0(\X))$.
\end{theorem}
\begin{proof}\ \\
\noindent{\sc Step 1: structure of the proof.} The case of absolutely continuous vector fields follows from the Sobolev one provided we show that for $(V_t),(Z_t)\in C([0,1],L^0(T\X))$ the map $t\mapsto \la V_t,Z_t\ra\circ F_0^t\in L^0(\X)$  is continuous. This follows noticing that, rather trivially, the map $t\mapsto \la V_t,Z_t\ra\in L^0(\X)$  is continuous, then recalling Lemma \ref{le:contcont}. We thus focus on the Sobolev case.

For $(z_t)\in W^{1,2}_{fix}$ denote by $\mathcal A_{(z_t)}\subset W^{1,2}_{fix}$ the collection of those $(v_t)$'s for which the conclusion of the theorem holds for $V_t:=\d F_0^t(v_t)$ and $Z_t:=\d F_0^t(z_t)$. Our goal is to prove that $\mathcal A_{(z_t)}= W^{1,2}_{fix}$ for any $(z_t)\in W^{1,2}_{fix}$ and we shall do so by applying Proposition \ref{prop:Atutto}.

It is clear that $\mathcal A_{(z_t)}$ is a vector space, i.e.\ $(o)$ holds, and - recalling \eqref{eq:l0lindf} - that the `restriction' property $(i)$ holds as well. Also, by direct computation we can verify that $(ii)$ holds. We pass to the stability property $(iii)$, thus let $(v^n_t)\stackrel{W^{1,2}_{fix}}\to(v^\infty_t)$,  and assume that the conclusion of the theorem  holds for $n<\infty$ and the choices $V_t=V^n_t:=\d F_0^t(v^n_t)$ and $Z_t:=\d F_0^t(z_t)$. Our goal is to prove that the conclusion also  holds for $V_t=V^\infty_t:=\d F_0^t(v^\infty_t)$. 

Since $(v^n_t)\to (v^\infty_t)$ and $(\dot v^n_t)\to (\dot v^\infty_t)$ in $L^2_{fix}\hookrightarrow L^0_{fix}$, Proposition \ref{prop:dFtot2} and the very definition \eqref{eq:defDt} tell that $(V^n_t)\to (V^\infty_t)$ and  $(D_tV^n_t)\to (D_tV^\infty_t)$ in $L^0_{var}$. It follows that
\begin{equation}
\label{eq:convinl0}
\begin{array}{ccc}
\la V^n_t,Z_t\ra \circ F_0^t\quad&\to&\quad \la V^\infty_t, Z_t\ra\circ F_0^t\\
\la D_tV^n_t, Z_t\ra\circ F_0^t+\la V^n_t,D_tZ_t\ra\circ F_0^t\quad&\to&\quad \la D_tV^\infty_t, Z_t\ra\circ F_0^t+\la V^\infty_t,D_tZ_t\ra\circ F_0^t
\end{array}
\end{equation}
in $L^0([0,1],L^0(\X))$. Now notice that 
\[
|(\la V^n_t,Z_t\ra \circ F_0^t)|_{L^2}\leq |(V^n_t)|_{L^2_{var}}|(Z_t)|_{L^\infty_{var}}\stackrel{\eqref{eq:Vtinfty}}\leq G |(V^n_t)|_{W^{1,2}_{var}}|(Z_t)|_{W^{1,2}_{var}},\qquad\mm-a.e.
\]
and 
\[
\begin{split}
|(\la D_tV^n_t, Z_t\ra\circ F_0^t+\la V^n_t,D_tZ_t\ra\circ F_0^t)|_{L^2}&\leq  |(D_tV^n_t)|_{L^2_{var}}|(Z_t)|_{L^\infty_{var}}+ |(V^n_t)|_{L^\infty_{var}}|(D_tZ_t)|_{L^2_{var}}\\
\text{(by \eqref{eq:Vtinfty})}\qquad&\leq G |(V^n_t)|_{W^{1,2}_{var}}|(Z_t)|_{W^{1,2}_{var}},\qquad\mm-a.e.,
\end{split}
\]
i.e., recalling also \eqref{eq:covtot}, we have 
\[
|(\la V^n_t,Z_t\ra \circ F_0^t)|_{W^{1,2}}\leq G |(v^n_t)|_{W^{1,2}_{fix}}|(Z_t)|_{W^{1,2}_{var}},\qquad\mm-a.e..
\]
Since $|(v^n_t)|_{W^{1,2}_{fix}}\to |(v^\infty_t)|_{W^{1,2}_{fix}}$ (as a simple consequence  of $(v^n_t)\stackrel{W^{1,2}_{fix}}\to (v^\infty_t)$), thanks to Proposition \ref{prop:stabl0w12} the conclusion follows from this latter estimate and \eqref{eq:convinl0}.

It remains to prove that $\mathcal A_{(z_t)}$ has property $(iv)$ in Proposition \ref{prop:Atutto}. Suppose that we know that this is the case  for $(z_t)$ constant. Then the argument above shows that for $(z_t)$ constant we have $\mathcal A_{(z_t)}=W^{1,2}_{fix}$. Therefore - by the symmetry in $(V_t),(Z_t)$ of the statement - we know that for $(z_t)\in W^{1,2}_{fix}$ arbitrary, the set $\mathcal A_{(z_t)}$ contains the constant vector fields. Hence again the argument above shows that $\mathcal A_{(z_t)}=W^{1,2}_{fix}$, which is  the conclusion.

We thus showed that it is sufficient to prove that for $v,z\in L^0(T\X)$, the conclusion of the theorem holds for $V_t:=\d F_0^t(v)$ and $Z_t:=\d F_0^t(z)$. By polarization, it is actually sufficient to consider the case $v=z$.

\noindent{\sc Step 2: key point.} To conclude the proof we need to   prove that for $v\in L^0(T\X)$, putting $V_t:=\d F_0^t(v)$ we have that $t\mapsto |V_t|^2\circ F_0^t$ is in $W^{1,2}([0,1],L^0(\X))$ and for a.e.\ $t\in[0,1]$ we have 
\begin{equation}
\label{eq:formulaleib}
\frac{\d}{\d t}\tfrac12|V_t|^2\circ F_0^t=\la V_t,D_t V_t\ra\circ F_0^t,\qquad\mm-a.e..
\end{equation}
{\sc Step 2a: Sobolev regularity.}
For the Sobolev regularity we start recalling from \eqref{eq:compLp} that $|(V_t)|_{L^\infty_{var}}<\infty$ $\mm$-a.e.\ and  that the group property \eqref{eq:groupRLF} and the chain rule \eqref{eq:chain_rule_diff_LL} give $V_s=\d F_t^s(V_t)$ for any $t,s\in[0,1]$. Then for any $t,s\in[0,1]$  we have
\[
\begin{split}
|V_s|^2\circ F_0^s-|V_t|^2\circ F_0^t&=(|V_s|\circ F_0^s+|V_t|\circ F_0^t)(|V_s|\circ F_0^s-|V_t|\circ F_0^t)\\
&\leq2 |(V_r)|_{L^\infty_{var}}\big(|\d F_t^s(V_t)|\circ F_0^s-|V_t|\circ F_0^t\big)\\
\text{(by \eqref{eq:pointwise_norm_estimate}, \eqref{eq:estdFt3})}\qquad&\leq 2 |(V_r)|^2_{L^\infty_{var}} G\int_{t\wedge s}^{t\vee s}g_r\circ F_0^r\,\d r
\end{split}
\]
$\mm$-a.e.. Swapping the roles of $t,s$ and recalling \eqref{eq:compLp} we get
\begin{equation}
\label{eq:persobreg2}
\big||V_s|^2\circ F_0^s-|V_t|^2\circ F_0^t\big|\leq  |v|^2   G\int_t^sg_r\circ F_0^r\,\d r,\qquad\mm-a.e.,\ \forall t,s\in[0,1],\ t\leq s.
\end{equation}
Since $(t \mapsto g_t \circ F_0^t) \in L^2([0,1],L^0(\X))$ as a consequence of the fact that $(g_t)\in L^2([0,1],L^0(\X))$ (by Proposition \ref{prop:eliadanieleregularity}) and \eqref{eq:intflo2}, this last estimate, thanks to the characterization in \eqref{eq:charw12bound}, is sufficient to deduce that $t\mapsto |V_t|^2\circ F_0^t$ is in $W^{1,2}([0,1],L^0(\X))$.

Now recall that \eqref{eq:charder}   gives that $\frac{|V_{t+h}|^2\circ F_0^{t+h}-|V_t|^2\circ F_0^t}{h}\to \frac{\d}{\d t}(|V_t|^2\circ F_0^t)$ in $L^0(\X)$ as $h\to 0$ for a.e.\ $t\in[0,1]$. Therefore to conclude that \eqref{eq:formulaleib} holds it is sufficient to prove that for a.e.\ $t\in[0,1]$ we have
\begin{equation}
\label{eq:perleibl0}
\begin{split}
\lims_{h\downarrow0}\sfd_{L^0}\Big(\Big(\frac{|V_{t+h}|^2\circ F_0^{t+h}-|V_t|^2\circ F_0^t}{2h}-\la D_t V_t,V_t\ra\circ F_0^t\Big)^-,0\Big)&=0,\\
\lims_{h\uparrow0}\sfd_{L^0}\Big(\Big(\frac{|V_{t+h}|^2\circ F_0^{t+h}-|V_t|^2\circ F_0^t}{2h}-\la D_t V_t,V_t\ra\circ F_0^t\Big)^+,0\Big)&=0,
\end{split}
\end{equation}
as these would imply that for a.e.\ $t$ we have $\sfd_{L^0}\big(\big(\frac{\d}{\d t}(\frac{1}{2}|V_t|^2\circ F_0^t)-\la D_t V_t,V_t\ra\circ F_0^t\big)^\mp,0\big)=0$, i.e.\ that  $\frac{\d}{\d t}(\frac{1}{2}|V_t|^2\circ F_0^t)=\la D_t V_t,V_t\ra\circ F_0^t$.

Now define 
\[
Q_t^s(v):=\frac{|\d F_t^s(v)|^2\circ F_t^s-|v|^2}{2(s-t)}
\]
and observe that the identity $V_s=\d F_t^s(V_t)$ that we already noticed gives $Q_t^s(V_t)=\frac{|V_{s}|^2\circ F_t^s-|V_t|^2}{2(s-t)}$. Moreover, since $D_tV_t=\nabla_{V_t}w_t$ by the very definition \eqref{eq:defDt} of convective derivative, we have $\la D_tV_t,V_t\ra=\la \nabla w_t,V_t\otimes V_t\ra$. Therefore recalling also \eqref{eq:equil0}, the claim \eqref{eq:perleibl0} will follow if we show that for a.e.\ $t\in[0,1]$ we have that for any $v\in L^0(T\X)$ it holds
\begin{equation}
\label{eq:perleibl02}
\begin{split}
\lims_{s\uparrow t}\sfd_{L^0}\big(\big(Q_t^s(v)-\la \nabla w_t,v\otimes v\ra\big)^+,0\big)=0=\lims_{s\downarrow t}\sfd_{L^0}\big(\big(Q_t^s(v)-\la \nabla w_t,v\otimes v\ra\big)^-,0\big).
\end{split}
\end{equation}
{\sc Step 2b: convergence properties on a dense set.} Let $N\subset[0,1]$ be the Borel negligible set given by Proposition \ref{prop:mainreg} and up to enlarge it, by keeping it negligible, assume that it contains the exceptional set of $s$'s for which \eqref{eq:c1RLF} does not hold. We claim that there is a  dense set $\mathcal D\subset L^0(T\X)$ for which the claim \eqref{eq:perleibl02} holds for any $v\in\mathcal D$ and $t\in[0,1]\setminus N$.

Fix $t\in[0,1]\setminus N$ and let  $v=\nabla f$ for $f\in\test\X$. Then, since  for any $s\in[0,1]$ we have
\[
\tfrac12|\d F_t^s(\nabla f)|^2\circ F_t^s\geq \d f(\d F_t^s(\nabla f))\circ F_t^s-\tfrac12|\d f|^2\circ F_t^s,\qquad\mm-a.e.
\]
with equality for $s=t$, we deduce that
\[
Q_{t}^s(\nabla f)\geq \underbrace{\frac{\d f(\d F_t^s(\nabla f))\circ F_t^s-|\d f|^2}{s-t}-\frac{\tfrac12|\d f|^2\circ F_t^s-\tfrac12|\d f|^2}{s-t}}_{=:RHS(s)},\qquad\mm-a.e.\ for\ s>t,
\]
 with opposite inequality for $s<t$. Using the identity $\d f(\d F_t^s(\nabla f))\circ F_t^s=\d(f\circ F_t^s)(\nabla f)$ and Proposition \ref{prop:mainreg} to handle the first addend, and \eqref{eq:c1RLF} and the fact that $|\d f|^2\in W^{1,2}(\X)$ (recall \eqref{eq:leibgrad}) for the second, we see that
 \[
\lim_{s\to t}RHS(s)= \d(\d f(w_t))(\nabla f)-\d(\tfrac12|\d f|^2)(w_t)=\la\nabla w_t,\nabla f\otimes\nabla f\ra,\qquad\text{in }L^0(\X).
 \]
Thus from the trivial implication $a\geq b\Rightarrow (a-c)^-\leq |b-c|$ valid for real numbers $a,b,c$, we get
\[
\lims_{s\downarrow t}\sfd_{L^0}\big(\big(Q_t^s(v)-\la \nabla w_t,v\otimes v\ra\big)^-,0\big)\leq \lims_{s\downarrow t}\sfd_{L^0}\big(\big|RHS(s)-\la \nabla w_t,v\otimes v\ra \big|,0\big)= 0,
\]
which is the second in \eqref{eq:perleibl02}. The first follows from the same arguments starting from the  implication $a\leq b\Rightarrow (a-c)^+\leq |b-c|$.

Now observe that if $(E_i)$ is a finite Borel partition of $\X$ and $(v_i)\subset L^0(T\X)$, for $v:=\sum_i\nchi_{E_i}v_i$ we have  $Q_t^s(v)=\sum_i\nchi_{E_i}Q_t^s(v_i)$ (recall \eqref{eq:l0lindf}). Similarly, $\la\nabla w_t,v\otimes v\ra=\sum_i\nchi_{E_i}\la \nabla w_t,v_i\otimes v_i\ra$. Thus if $v_i=\nabla f_i$ for $f_i\in{\rm Test}(\X)$  and every $i$, using what already established we get
\[
\begin{split}
\lims_{s\downarrow t}\sfd_{L^0}\big(\big(Q_t^s(v)-\la \nabla w_t,v\otimes v\ra\big)^-,0\big)&=\lims_{s\downarrow t}\sfd_{L^0}\Big(\sum_i\nchi_{E_i}\big(Q_t^s(\nabla f_i)-\la \nabla w_t,\nabla f_i\otimes \nabla f_i\ra\big)^-,0\Big)\\
&\leq\sum_i\lims_{s\downarrow t}\sfd_{L^0}\Big(\nchi_{E_i}\big(Q_t^s(\nabla f_i)-\la \nabla w_t,\nabla f_i\otimes \nabla f_i\ra\big)^-,0\Big)\\
&\leq\sum_i\lims_{s\downarrow t}\sfd_{L^0}\Big(\big(Q_t^s(\nabla f_i)-\la \nabla w_t,\nabla f_i\otimes \nabla f_i\ra\big)^-,0\Big)=0,
\end{split}
\]
which is the second in \eqref{eq:perleibl02}. The first follows along similar lines.

In summary, we proved that \eqref{eq:perleibl02} holds for any $v$ in the dense set $\mathcal D\subset L^0(T\X)$ made of those vectors of the form $v:=\sum_i\nchi_{E_i}\nabla f_i$, where the sum is finite and $E_i,f_i$ are as above.\\ 
{\sc Step 2c: equicontinuity and conclusion.}  Put for brevity
\[
\tilde Q_t^s(v):=Q_t^s(v)-\la\nabla w_t,v\otimes v\ra\qquad\forall v\in L^0(T\X), t,s\in[0,1],\ t\neq s.
\]
We  claim that for a.e.\ $t\in[0,1]$ we have that for any given $v\in L^0(T\X)$  the $\tilde Q_t^s$'s are equicontinuous at $v$, i.e.\ that
 \begin{equation}
\label{eq:unifcontQt}
\lim_{v'\stackrel{L^0}\to v}\sup_{s\in[0,1]\setminus\{t\}}\sfd_{L^0}\big(\tilde Q_t^s(v'),\tilde Q_t^s(v)\big)=0,
\end{equation}
 and to prove this we are going to apply Proposition \ref{prop:Qt} above. Start observing that arguing as for  \eqref{eq:persobreg2} we get
 \begin{equation}
\label{eq:qts}
 |Q_t^s(v)|\leq |v|^2 G_t\int_{t\wedge s}^{t\vee s}g_r\circ F_t^r\,\d r,\qquad\mm-a.e.,\ \forall t,s\in[0,1],\ s\neq t,\ v\in L^0(T\X),
\end{equation}
then  notice that for any  $s\in[0,1]\setminus\{t\}$,  $\tilde Q_t^s$ is a quadratic form induced by  a bilinear form that we shall denote by $B_t^s:L^0(T\X)^2\to L^0(\X)$. By \eqref{eq:l0lindf} we see that $B_t^s$ is $L^0(\X)$-bilinear and \eqref{eq:qts}  tells that
 \begin{equation}
\label{eq:perequicont}
 |\tilde Q_t^s(v)|\leq |\nabla w_t||v|^2+ G_t |v|^2\,\underbrace{\frac1{|s-t|}\int_{s\wedge t}^{s\vee t}g_r\circ F_t^r\,\d r}_{=:\hat g_t^s},\qquad\mm-a.e.,\ \forall t,s\in[0,1],\ s\neq t,\ v\in L^0(T\X).
\end{equation}
 Now put 
 \[
 M(t):=\sup_{s\in[0,1]\setminus \{t\}}\frac  1{|s-t|}\int_{s\wedge t}^{s\vee t}\|g_r^2\|_{L^2}^2\,\d r
 \]
 and notice that since $\int_0^1\|g_r^2\|_{L^2}^2\,\d r<\infty$ (by Proposition \ref{prop:eliadanieleregularity}), the Hardy--Littlewood maximal inequality ensures that $M(t)<\infty$ for a.e.\ $t\in[0,1]$. Then for every  $s\in[0,1]\setminus\{t\}$ we have
 \begin{equation}
\label{eq:boundghatts}
\|\hat g_t^s\|_{L^2}^2\leq\frac1{|s-t|}\int_{s\wedge t}^{s\vee t}\int g_r^2\circ F_t^r\,\d\mm\,\d r\stackrel{\eqref{eq:boundcompr}}\leq \frac  C{|s-t|}\int_{s\wedge t}^{s\vee t}\int g_r^2\,\d\mm\,\d r\leq CM(t),
\end{equation}
thus  Chebyshev's inequality gives the uniform control
 \begin{equation}
\label{eq:percontl0}
 \mm(\{|\hat g_t^s|\geq c\})= \mm(\{|\hat g_t^s|^2\geq c^2\})\leq\frac{\|\hat g_t^s\|_{L^2}^2}{c^2}\leq \frac{CM(t)}{c^2}\qquad\forall s\in[0,1]\setminus\{t\}.
\end{equation}
 Now notice that the set  $N':=N\cup\{t:M(t)=\infty\}\cup\{t:w_t\notin W^{1,2}_C(T\X)\}$  is Borel and negligible and fix $t\in[0,1]\setminus N'$. From \eqref{eq:percontl0}, $\mm'\ll\mm$, and the absolute continuity of the integral, we see that the functions $G_t^s:=|\nabla w_t|+G_t\hat g_t^s$, parametrized by $s\in[0,1]\setminus\{t\}$, satisfy \eqref{eq:perunifcont2} (with $i=s$). Thus the claim \eqref{eq:unifcontQt} follows from Proposition \ref{prop:Qt} and the bound \eqref{eq:perequicont}.
 
 Now let $v\in L^0(T\X)$ be arbitrary and $v'\in\mathcal D$. Using the trivial bound $a^-\leq b^-+|a-b|$ valid for any $a,b\in\R$, we get 
\[
\begin{split}
\lims_{s\downarrow t}\sfd_{L^0}\big((\tilde Q_t^s(v))^-,0\big)&\leq \lims_{s\downarrow t}\sfd_{L^0}\big((\tilde Q_t^s(v'))^-,0\big)+\lims_{s\downarrow t}\sfd_{L^0}\big(|\tilde Q_t^s(v)-\tilde Q_t^s(v')|,0\big) \\
\text{(by {\sc Step 2b} and $v'\in\mathcal D$)}\qquad&\leq\sup_{s\in[0,1]\setminus \{t\}}\sfd_{L^0}\big(\tilde Q_t^s(v),\tilde Q_t^s(v')\big).
\end{split}
\]
Taking the limit as $v'\to v$, $v'\in \mathcal D$ and using \eqref{eq:unifcontQt}, we get the second in \eqref{eq:perleibl02}. The first is proved analogously, by exploiting the fact that, for every $a,b \in \R$, $a^+ \leq b^+ + |a-b|$.
\end{proof}
\begin{remark}{\rm
In the last step of the proof we showed that the functions $\hat g_t^s$ defined in \eqref{eq:perequicont} satisfy the bound \eqref{eq:boundghatts}. Although not needed for our purposes, it is worth to point out that in fact the stronger estimate
\begin{equation}
\label{eq:strongergts}
\iint_0^1H_t^2\,\d t\,\d \mm\leq C\iint_0^1g^2_t\,\d t\,\d \mm\qquad \text{for}\qquad H_t:=\sup_{s\neq t}\hat g_t^s,  
\end{equation}
holds, for some universal constant $C$ depending only on the constant in \eqref{eq:boundcompr}. Indeed, putting $H'_t:=\sup_{s\neq t}\hat g_t^s\circ F_0^t=H_t\circ F_0^t$, from \eqref{eq:boundcompr} we have $\iint_0^1H_t^2\,\d t\,\d \mm\leq C\iint_0^1H'^2_t\,\d t\,\d \mm$. Now observe that  the Hardy--Littlewood maximal inequality gives
\[
\begin{split}
\int_0^1H_t'^2\,\d t=\int_0^1\Big|\sup_{s\neq t}\frac  1{|s-t|}\int_{s\wedge t}^{s\vee t} g_r\circ F_0^r\,\d r\Big|^2\,\d t
\leq C\int_0^1  g^2_t\circ F_0^t\,\d t
\end{split}
\]
so that the claim follows by integration in $\mm$ and using again  \eqref{eq:boundcompr}.
}\fr\end{remark}

\subsection{Existence and uniqueness of Parallel Transport}

We introduce the notion of parallel transport for what concerns our setting.

\begin{definition}[Parallel Transport] A Parallel Transport along the flow $(F_t^s)$ of the vector field $(w_t)$ is a vector field $(V_t)\in AC^2_{var}([0,1],L^0(T\X))$ such that
\[
D_tV_t=0,\qquad\mm-a.e.,\ for\ a.e.\ t\in[0,1].
\]
We say that the Parallel Transport $(V_t)$ \emph{starts} from $\bar V\in L^0(T\X)$ provided $V_0=\bar V$.
\end{definition}
The linearity of the condition of being a Parallel Transport together with the Leibniz rule \eqref{eq:leibconv} easily gives:
\begin{proposition}[Uniqueness and preservation of norm]
Let $(V_t)$ be a Parallel Transport. Then the map $t\mapsto |V_t|\circ F_0^t\in L^0(\X)$ is constant.

Moreover, for any $\bar V\in L^0(T\X)$ there exists at most one Parallel Transport starting from $\bar V$.
\end{proposition}
\begin{proof}
By Theorem \ref{thm:leibconv}, if $(V_t)$ is a Parallel Transport then $t\mapsto |V_t|^2\circ F_0^t\in L^0(\X)$ is in $AC^2([0,1],L^0(\X))$ with null derivative. By \eqref{eq:charAC} such map is constant.

Now let $(V^1_t),(V^2_t)$ be two Parallel Transports starting from $\bar V$. Then $t\mapsto V_t:=V^1_t-V^2_t$ is a Parallel Transport starting from 0. By what already proved we deduce that $|V_t|\circ F_0^t=|V_0|=0$ $\mm$-a.e.\ for every $t\in[0,1]$, i.e.\ $V^1_t=V^2_t$ for every $t\in[0,1]$, as claimed.
\end{proof}
We turn to the problem of existence, that will be addressed by transforming it into an appropriate ODE-like problem in $L^2_{fix}$. Let us fix some terminology. A family $(\ell_t)_{t\in[0,1]}$ of  $L^0$-linear and continuous maps from $L^0(T\X)$ into itself is said \emph{Borel} provided $t\mapsto \ell_t(v_t)$ is in  $L^0_{fix}([0,1],L^0(T\X))$ for any $(v_t)\in  L^0_{fix}([0,1],L^0(T\X))$ (arguing as in the proof of Proposition \ref{prop:dFtot2} this is equivalent to asking that $t\mapsto\ell_t(v)$ is in  $L^0_{fix}([0,1],L^0(T\X))$ for any $v\in L^0(T\X)$).  Recalling that for any $L^0$-linear and continuous map $\ell:L^0(T\X)\to L^0(T\X)$ we have $|\ell|=\sup_n\ell(v_n)$ $\mm$-a.e.\ for an appropriate choice of the countable set $\{v_n\}_{n\in\N}\subset L^0(T\X)$, if $(\ell_t)$ is a Borel family then the map $t\mapsto |\ell_t|\in L^0(\X)$ is also Borel.

Assume now that the Borel family $(\ell_t)$ satisfies
\begin{equation}
\label{eq:elll2}
\int_0^1|\ell_t|^2\,\d t<\infty,\qquad\mm-a.e.
\end{equation}
and notice that in this case for any $(v_t)\in L^2_{fix}([0,1],L^0(T\X))$ we have 
\begin{equation}
\label{eq:perL}
\int_0^1|\ell_t(v_t)|\,\d t\leq\int_0^1|\ell_t|\,|v_t|\,\d t\leq|(v_t)|_{L^2_{fix}}\sqrt{\int_0^1|\ell_t|^2\,\d t}\qquad\mm-a.e.
\end{equation}
and thus $(\ell_t(v_t))\in L^1_{fix}$. Hence, recalling the concept of integral of elements in $L^1_{fix}$ discussed in Section \ref{se:intmod}, for any $t\in[0,1]$ we can define 
\begin{equation}
\label{eq:defL}
L_t((v_r)):=\int_0^t\ell_r(v_r)\,\d r\in L^0(T\X)
\end{equation}
and it is clear that $L_t:L^2_{fix}\to L^0(T\X)$ is $L^0(\X)$-linear and continuous. We shall write $L_t:=\int_0^t\ell_r\,\d r$ for the operator defined by the above formula. Notice that from
\[
|L_s((v_r))-L_t((v_r))|=\big|\int_t^s\ell_r(v_r)\,\d r\big|\leq \int_t^s|\ell_r(v_r)|\,\d r \qquad\mm-a.e.
\]
it follows that $t\mapsto L_t((v_r))\in L^0(T\X)$ is continuous w.r.t.\  the $L^0(T\X)$-topology. 

We call $L:L^2_{fix}\to C([0,1],L^0(T\X))\subset L^0_{fix}$ the resulting map, i.e.\ 
\begin{equation}
\label{eq:defLL}
L((v_r))_t:=L_t((v_r)),\qquad\forall t\in[0,1].
\end{equation} 
Notice that \eqref{eq:perL} ensures that $L((v_r))\in L^\infty_{fix}\subset L^2_{fix}$ for any $(v_t)\in L^2_{fix}$. Below, to simplify the notation, we shall write $L(v)$ in place of $L((v_t))$, as well as $L_t(v)$ in place of $L_t((v_r))$.
\begin{proposition}\label{prop:perPT}
Let $\ell_t:L^0(T\X)\to L^0(T\X)$ be a Borel family of $L^0$-linear and continuous maps from $L^0(T\X)$ into itself satisfying \eqref{eq:elll2} and define $L:L^2_{fix}([0,1],L^0(T\X))\to L^2_{fix}([0,1],L^0(T\X))$ as in \eqref{eq:defLL}, \eqref{eq:defL}. 
Then:
\begin{itemize}
\item[i)] for every $(z_t)\in L^2_{fix}([0,1],L^0(T\X))$ there is a unique $(v_t)\in  L^2_{fix}([0,1],L^0(T\X))$ such that
\begin{equation}
\label{eq:vdaz}
v_t=L_t(v)+z_t,\qquad a.e.\ t,\ \mm-a.e..
\end{equation}
\item[ii)] if $(z_t)\in AC^2_{fix}([0,1],L^0(T\X))$, then the unique $(v_t)\in L^2_{fix}([0,1],L^0(T\X))$ solving \eqref{eq:vdaz} admits a continuous representative in $AC^2_{fix}([0,1],L^0(T\X))$ for which \eqref{eq:vdaz} holds for every $t\in[0,1]$ and moreover
\begin{equation}
\label{eq:dervz}
\left\{\begin{array}{rll}
\dot v_t&=\ell_t(v_t)+\dot z_t,&\qquad a.e.\ t,\\
v_0&=z_0.&
\end{array}\right.
\end{equation}
\end{itemize}
\end{proposition}
\begin{proof}\ \\
\noindent{\sc (i)} For $(v_t)\in L^2_{fix}$ and $s\in[0,1]$ we shall define $|(v_t)|_{L^2_{fix}([0,s])}\in L^0(\X)$ as
\[
|(v_t)|_{L^2_{fix}([0,s])}^2:=\int_0^s|v_r|^2\,\d r,
\] 
so that $|(v_t)|_{L^2_{fix}([0,1])}=|(v_t)|_{L^2_{fix}}$. We claim that for every $n\in\N$ and $(v_t)\in L^2_{fix}$ we have
\begin{equation}
\label{eq:claimL}
|L^n(v)|_{L^2_{fix}([0,t])}^2\leq \frac{t^{n}\bar{G}^n}{n!}|(v_t)|_{L^2_{fix}([0,t])}^2.
\end{equation}
where $\bar{G}:= \int_0^1 |l_t|^2\,\d t$.
The case $n=0$ is obvious. Assume we proved the claim for $n$ and let us prove it for $n+1$. We have
\[
\begin{split}
|L^{n+1}(v)|_{L^2_{fix}([0,t])}^2&=\int_0^t|L_s(L^n(v))|^2\,\d s=\int_0^t\big|\int_0^s \ell_r(L^n(v)_r)\,\d r\big|^2\,\d s\\
&\leq \int_0^t \big(\int_0^s|\ell_r|^2\,\d r\big)\,|L^{n}(v)|_{L^2_{fix}([0,s])}^2\,\d s\\
&\stackrel{*}\leq \frac{\bar G^{n+1}}{n!}\int_0^t s^n|(v_t)|_{L^2_{fix}([0,s])}^2\,\d s\\
&\leq  \frac{t^{n+1}\bar G^{n+1}}{(n+1)!}|(v_t)|_{L^2_{fix}([0,t])}^2,
\end{split}
\]
where in the starred inequality we used \eqref{eq:elll2} and the induction assumption. From \eqref{eq:claimL} it follows that $\sum_n|L^n(v)|_{L^2_{fix}}<\infty$ $\mm$-a.e., hence the series $\sum_{n\in\N}L^n(v)$ is a well defined element of $L^2_{fix}$ (meaning that the partial sums form a Cauchy sequence in $L^2_{fix}$). It is clear that the operator $(v_t)\mapsto \sum_{n\in\N}L^n(v)$ is the inverse of ${\rm Id}-L$, indeed
\[
({\rm Id}-L)\sum_{n\in\N}L^n(v)=\lim_N({\rm Id}-L)\sum_{n=0}^NL^n(v)=\lim_Nv-L^{N+1}(v)=v
\]
and similarly  $\sum_{n\in\N}L^n(({\rm Id}-L)(v))=v$. Thus $(v_t)\in L^2_{fix}$ solves \eqref{eq:vdaz} if and only if $v=\sum_{n\in\N}L^n(z)$, proving existence and uniqueness.

\noindent{\sc (ii)}   By \eqref{eq:charACvec} we know that $L(v)\in AC^2_{fix}$ for any $(v_t)\in L^2_{fix}$, thus if $(z_t)\in AC^2_{fix}$ as well, we have that the right-hand side of \eqref{eq:vdaz} is in $AC^2_{fix}$. By the previous step, such right-hand side is the required representative of $(v_t)$ in $AC^2_{fix}$ for which \eqref{eq:vdaz} holds for every $t\in[0,1]$. Then \eqref{eq:dervz} follows from the identity $L_t(v)+z_t=z_0+\int_0^t\ell_s(v_s)+\dot z_s\,\d s$ (recall \eqref{eq:intvt}) and the general property \eqref{eq:charACvec}.
\end{proof}

\begin{theorem}[Existence and uniqueness of Parallel Transport]\label{thm:exuni} Let $\bar V\in L^0(T\X)$ and $(Z_t)\in L^2_{var}$. Then there is a unique  $(V_t)\in AC^2_{var}$ such that 
\begin{equation}
\label{eq:odevar}
\left\{\begin{array}{ll}
D_tV_t=Z_t&\qquad a.e.\ t\in[0,1],\\
V_0=\bar V.&
\end{array}\right.
\end{equation}
In particular, there is a unique
Parallel Transport $(V_t)$ starting from $\bar V$.
\end{theorem}
\begin{proof} Recalling the definition of $AC^2_{var}([0,1],L^0(T\X))$ and that of convective derivative, we see that $(V_t)$ satisfies \eqref{eq:odevar} if and only if  for $v_t:=\d F_t^0(V_t)$ we have $(v_t)\in AC^2_{fix}([0,1],L^0(T\X))$ with 
\begin{equation}
\label{eq:odefix}
\left\{\begin{array}{ll}
\dot v_t=\d F_t^0(Z_t-\nabla_{\d F_0^t(v_t)}w_t)&\qquad a.e.\ t\in[0,1],\\
v_0=\bar V.&
\end{array}\right.
\end{equation}
Thus for a.e.\ $t\in[0,1]$ we define $\ell_t:L^0(T\X)\to L^0(T\X)$ as
\[
\ell_t(v):=-\d F_t^0(\nabla_{\d F_0^t(v)}w_t)
\]
and notice that the bound
\[
|\ell_t(v)|\stackrel{\eqref{eq:pointwise_norm_estimate}}\leq \big(|\d F_t^0|\,|\nabla_{\d F_0^t(v)}w_t|\big)\circ F_0^t\leq \big(|\d F_t^0|\,|{\d F_0^t(v)}|\,|\nabla w_t|\big)\circ F_0^t\stackrel{\eqref{eq:pointwise_norm_estimate},\eqref{eq:estdFt4}}\leq  G|v|\,|\nabla w_t|\circ F_0^t,
\]
valid $\mm$-a.e.\ for every $v\in L^0(T\X)$, yields $|\ell_t|\leq G|\nabla w_t|\circ F_0^t$, so that \eqref{eq:estnw2} ensures that the bound \eqref{eq:elll2} holds. Also, we put 
\[
z_t:=\bar V+\int_0^t\d F_r^0(Z_r)\,\d r.
\]
Notice that the bound $|\d F_t^0(Z_t)|\leq (|\d F_t^0|\,|Z_t|)\circ F_0^t\leq G |Z_t|\circ F_0^t$ (having used \eqref{eq:estdFt4}) ensures that the definition is well-posed. It is then clear from \eqref{eq:charACvec} that $(z_t)\in AC^2_{fix}$.

We can therefore apply $(ii)$ of Proposition \ref{prop:perPT} above with this choice of $(\ell_t)$ and  $(z_t)$: we thus obtain the existence of  $(v_t)\in AC^2_{fix}([0,1],L^0(T\X))$ satisfying \eqref{eq:odefix}, so that   $t\mapsto V_t:=\d F_0^t(v_t)$ is the only solution of \eqref{eq:odevar}, as desired.
\end{proof}

\subsection{\texorpdfstring{$W^{1,2}_{var}([0,1],L^0(T\X))$}{W12var} as intermediate space between
\texorpdfstring{$\mathscr H^{1,2}(\ppi)$}{H12pi} and \texorpdfstring{$\mathscr W^{1,2}(\ppi)$}{W12pi}}\label{se:comparison}

In this section we compare the main definitions given in this manuscript with those provided in the earlier paper \cite{GP20}: we shall prove that in a very natural sense - see \eqref{eq:claimhw}  below - the space $W^{1,2}_{var}([0,1],L^0(T\X))$ lies between the spaces $\mathscr H^{1,2}(\ppi)$ and $\mathscr W^{1,2}(\ppi)$ for relevant $\ppi$'s. 

In order to properly formulate this, some definitions are due. Let $\mu\in\mathscr P(\X)$ be such that $\mu\leq C\mm$ for some $C>0$. Let $F$ be the regular lagrangian flow associated to $w$ satisfying the assumptions of Theorem \ref{thm:AT}. Define $\ppi_\mu:=(F_0^\cdot)_*\mu\in \mathscr P(C([0,1],\X))$, where $F_0^\cdot:\X\to C([0,1],\X)$ is the $\mm$-a.e.\ defined Borel map sending $x$ to $t\mapsto F_0^t(x)$. It is clear from the bounded compression property of the flow and \eqref{eq:unifspFl} that $\ppi_\mu$ is a test plan, whose speed is given by $ {(\ppi_\mu)}'_t=\e_t^*w_t\in\e_t^*L^2(T\X)$ for $\mathcal{L}^1$-a.e.\ $t\in[0,1]$.
Moreover, since $|w_t| \in L^\infty([0,1] \times \X)$, $\ppi_{\mu}$ is a Lipschitz test plan, so the full theory developed
in \cite{GP20} applies (indeed, the test plan $\ppi_{\mu}$ has to be Lipschitz in order to define the functional spaces $\mathscr{H}^{1,2}(\ppi_{\mu})$ and $\mathscr{W}^{1,2}(\ppi_{\mu})$). In \cite{GP20}, in order to speak of vector fields along a test plan, it has been necessary to use the concept of pullback. This was due to the fact that the evaluation maps $\e_t$ are in general not $\ppi$-essentially injective for $\ppi$ test plan. The situation is different now: from the uniqueness of Regular Lagrangian Flows  it easily follows that $\e_t$ is $\ppi_\mu$-essentially injective  for any $t\in[0,1]$, the left inverse being $F_0^\cdot\circ F_t^0$. It is clear that for any $t\in[0,1]$ we have
\begin{equation}
\label{eq:etinv}
\begin{split}
(F_0^\cdot\circ F_t^0)\circ \e_t&=\Id_{C([0,1],\X)}\quad\ppi_\mu-a.e.,\\
\e_t\circ(F_0^\cdot\circ F_t^0)&=\Id_\X\quad(\e_t)_*\ppi_\mu-a.e..
\end{split}
\end{equation}
Now notice that since $(\e_t)_*\ppi_\mu\ll\mm$, we have a well-defined pullback map $\e_t^*:L^0(T\X)\to \e_t^*L^0(T\X)$; we shall often  denote this map $\Psi_t$ rather than $\e_t^*$, i.e.
\[
\begin{array}{rccc}
\Psi_t:&L^0(T\X)&\quad\to\quad & \e_t^*L^0(T\X)\\
&V&\quad\mapsto\quad &\e_t^*(V).
\end{array}
\]
Similarly, defining up to $\mm$-negligible sets the Borel set  ${\sf A}_t:=\{\frac{\d(\e_t)_*\sppi_\mu}{\d\mm}>0\}$, we have that $(F_0^\cdot\circ F_t^0)_*\mm\restr{{\sf A}_t}\ll(F_0^\cdot\circ F_t^0)_*(\e_t)_*\ppi_\mu=\ppi_\mu$  and therefore we have a natural pullback map $(F_0^\cdot\circ F_t^0)^*:\e_t^*L^0(T\X)\to \mathscr M$, where for brevity we denoted by $\mathscr M$ the $L^0(\X,\mm\restr{{\sf A}_t})$-normed module $(F_0^\cdot\circ F_t^0)^*\e_t^*L^0(T\X,\ppi_\mu)$. Then from \eqref{eq:etinv} and the functoriality of the pullback we see that $\mathscr M$ can, and will, be canonically identified with the restriction $L^0(T\X)\restr{{\sf A}_t}$ of $L^0(T\X)$ to the set ${\sf A}_t$. In turn, such restriction can naturally be seen as the subset of $L^0(T\X)$ made of vector fields which are $\mm$-a.e.\ 0 on the complement of ${\sf A}_t$: we shall denote by $\mathscr M\ni v\mapsto \nchi_{{\sf A}_t}v\in L^0(T\X)$  the inclusion map of $\mathscr M$ in $L^0(T\X)$ (a similar construction can be made using the `extension' functor defined in \cite{GPS18}). In summary, we have a natural map $\Phi_t$ from the $L^0(\ppi_\mu)$-normed module $\e_t^*L^0(T\X)$ to the $L^0(\mm)$-normed module $L^0(T\X)$ defined as
\[
\begin{array}{rccc}
\Phi_t:&\e_t^*L^0(T\X)&\quad\to\quad & L^0(T\X)\\
&{\sf v}&\quad\mapsto\quad & \nchi_{{\sf A}_t}(F_0^\cdot\circ F_t^0)^*({\sf v}).
\end{array}
\]
 From \eqref{eq:etinv} and the functoriality of the pullback we see that
\begin{equation}
\label{eq:psiphi}
\begin{split}
\Psi_t(\Phi_t({\sf v}))&={\sf v},\qquad\qquad\forall {\sf v}\in \e_t^*L^0(T\X),\\
\Phi_t(\Psi_t(V))&=\nchi_{{\sf A}_t}V,\qquad\forall V\in L^0(T\X).
\end{split}
\end{equation}
Now recall that the space ${\rm VF}(\ppi_\mu)$ of vector fields along $\ppi$ has been defined in \cite{GP20} as 
\[
{\rm VF}(\ppi_\mu):=\prod_{t\in[0,1]}\e_t^*L^0(T\X)
\]
(actually in \cite{GP20} the pullbacks of $L^2(T\X)$ are considered, but we aligned the definition to the framework used here). Notice that the product of $({\sf v}_t)\in {\rm VF}(\ppi_\mu)$ by a function in $L^0(\ppi_\mu)$ can be defined componentwise and similarly for $({\sf v}_t),({\sf z}_t)\in{\rm VF}(\ppi_\mu)$ the function $\la({\sf v}_t),({\sf z}_t)\ra$ is defined as $t\mapsto\la {\sf v}_t,{\sf z}_t\ra\in L^0(\ppi_\mu)$. 

The maps $\Psi_t,\Phi_t$ naturally induce maps $\Psi,\Phi$ by acting componentwise:

\[
\begin{array}{rcccrccc}
\Psi:&\prod_{t\in[0,1]}L^0(T\X)&\to &{\rm VF}(\ppi_\mu)&\qquad\qquad\Phi:&{\rm VF}(\ppi_\mu)&\to &\prod_{t\in[0,1]}L^0(T\X)\\
&(t\mapsto V_t)&\mapsto &(t\mapsto\Psi_t(V_t)),&\qquad\qquad&(t\mapsto {\sf v}_t)&\mapsto &(t\mapsto\Phi_t({\sf v}_t)).
\end{array}
\]
We have already noticed that ${\rm VF}(\ppi_\mu)$ comes with a natural structure of module over $L^0(\ppi_\mu)$ (but not of $L^0(\ppi_\mu)$-normed module) given by componentwise product. We also define a structure of $L^0(\X)$ module on $\prod_{t\in[0,1]}L^0(T\X)$ by putting, in analogy with \eqref{eq:l0varprod}:
\[
f\times(V_t):=(f\circ F_t^0\,V_t)\qquad\forall f\in L^0(\X),\ (V_t)\in \prod_{t\in[0,1]}L^0(T\X).
\]
These two module structures are related, indeed for $f,(V_t)$ as above, the identity $F_t^0\circ \e_t=\e_0$ valid $\ppi_\mu$-a.e.\ gives 
\[
\e_t^*(f\circ F_t^0\,V_t)=f\circ F_t^0\circ \e_t\,\e_t^*(V_t)=f\circ\e_0\,\e_t^*(V_t)
\]
and therefore
\begin{equation}
\label{eq:modconj}
\Psi(f\times(V_t))=f\circ\e_0\,\Psi(V_\cdot).
\end{equation}

Elements of  ${\rm VF}(\ppi_\mu)$ and $\prod_{t\in[0,1]}L^0(T\X)$  come without any time-regularity. For the latter space this can easily be defined using the fact that  $\prod_{t\in[0,1]}L^0(T\X)$ is the set of maps from $[0,1]$ to $L^0(T\X)$, but in the former space  this is a bit trickier due to the fact that the factors $\e_t^*L^0(T\X)$ change with $t$. We shall  `test' regularity against a properly chosen family of `test objects' (in analogy with the typical construction for direct sums of Hilbert spaces).

We thus define the subspace ${\rm TestVF}(\ppi_\mu)\subset {\rm VF}(\ppi_\mu)$ of \emph{test vector fields} along $\ppi_\mu$ as the collection of vector fields of the form $t\mapsto\sum_{i=0}^n\varphi_i(t)\nchi_{\Gamma_i}\e_t^*(v_i)$ for generic $n\in\N$, $\Gamma_i\subset C([0,1],\X)$ Borel, $\varphi_i\in\LIP([0,1])$, and $v_i\in \testv(\X)$. Now notice that   for $\Gamma\subset C([0,1],\X)$ Borel, the set $\Gamma_0:=(F_0^\cdot)^{-1}(\Gamma)\subset \X$ satisfies $\nchi_{\Gamma_0}\circ\e_0=\nchi_{\Gamma}$ $\ppi_\mu$-a.e., therefore \eqref{eq:modconj} gives
\begin{equation}
\label{eq:perprod}
\Psi\big(\nchi_{\Gamma_0}\times(V_t)\big)=\nchi_\Gamma\,\Psi(V_\cdot)
\end{equation}
and thus
\[
{\rm TestVF}(\ppi_\mu)=\Psi\big({\rm Test}\Pi\big),
\]
where ${\rm Test}\Pi\subset\prod_{t\in[0,1]}L^0(T\X) $ is the set of maps of the form $t\mapsto\sum_{i=0}^n\nchi_{A^i}\times(\varphi_i(t)V_i)$ for $n\in\N$, $A^i\subset\X$ Borel, $V_i\in\testv(\X)$, and $\varphi_i\in\LIP([0,1])$.

In \cite{GP20}, an element $({\sf v}_t)$ of ${\rm VF}(\ppi_\mu)$ has been said Borel (we are slightly changing the terminology, but not the core point) provided $t\mapsto \la {\sf v}_t,{\sf z}_t\ra\in L^0(\ppi_\mu)$ is Borel for every $({\sf z}_t)\in {\rm TestVF}(\ppi_\mu)$. Notice that arguing as in Proposition \ref{prop:dFtot2} it is easy to check that $(V_t)\in \prod_{t\in[0,1]}L^0(T\X)$ is Borel (as a map from $[0,1]$ to $L^0(T\X)$) if and only if $t\mapsto \la V_t,Z_t\ra\in L^0(\X)$ is Borel for every $(Z_t)\in {\rm Test}\Pi$. Then keeping in mind the simple identity $\la\Psi_t(V),\Psi_t(Z)\ra=\la V,Z\ra\circ F_0^t\circ\e_0$ valid $\ppi_\mu$-a.e.\ for every $V,Z\in L^0(T\X)$ and the fact that $\e_0$ is $\ppi_\mu$-essentially injective, it is easy to see that $({\sf v}_t)\in {\rm VF}(\ppi_\mu)$ is Borel if and only if it is of the form $\Psi(V_\cdot)$ for some $(V_t)\in \prod_{t\in[0,1]}L^0(T\X) $ Borel. 

\bigskip

It is not hard to check that for $({\sf v}_t)\in {\rm VF}(\ppi_\mu)$ Borel  the map 
\[
[0,1]\ni t\mapsto \|({\sf v}_s)\|_t:=\|{\sf v}_t\|_{\e_t^*L^2(T\X)}=\||{\sf v}_t|\|_{L^2(\sppi_\mu)}=\sqrt{\int|{\sf v}_t|^2\,\d\ppi_\mu} \in [0,+\infty]
\]
is Borel. To see this, consider $\varphi_n(y) = (y \wedge n) \vee 0$ for $y \in \mathbb{R}$, notice $L^0(\ppi_{\mu}) \ni f \mapsto T_n(f) \in [0,+\infty)$ is Borel where $T_n(f) = \int \varphi_n(f)^2\,\d \ppi_\mu$ and that for every $t$ $\e_t^*L^0(T\X) \ni v \mapsto \int |v|^2\,\d \ppi_{\mu} = \sup_n T_n(|v|)$. Thus we can define $\mathscr L^2(\ppi_\mu)\subset{\rm VF}(\ppi_\mu)$  as the space of those $({\sf v}_t)$'s Borel such that
\[
\|({\sf v}_t)\|_{\mathscr L^2(\sppi_\mu)}^2:=\int_0^1\|({\sf v}_s)\|_t^2\,\d t<\infty,
\]
with the usual identification up to equality for a.e.\ $t$. It is then clear that
\[
\mathscr L^2(\ppi_\mu)=\Psi\Big(\Big\{(V_t)\in L^0_{var}\ : \ \iint_0^1|V_t|^2\circ F_0^t\,\d t\,\d\mu=\int_0^1\int |V_t|^2\circ \e_t\,\d \ppi_\mu\,\d t<\infty\Big\}\Big).
\]
To define  the spaces $\mathscr W^{1,2}(\ppi_\mu),\mathscr H^{1,2}(\ppi_\mu)$ we need first to describe differentiation of test vector fields along $\ppi_\mu$. We start defining $\tilde {\sf D}_t:{\rm Test}\Pi\to L^0_{var}$ as the linear extension of
\begin{equation}
\label{eq:dttilde}
\tilde{\sf D}_t(\nchi_A\times(\varphi(t)V)):=\nchi_A\times(\varphi'(t)V+\varphi(t)\nabla_{w_t}V),\qquad a.e.\ t\in[0,1],
\end{equation}
for $A\subset\X$ Borel, $\varphi\in \LIP([0,1])$, and $V\in{\rm TestV}(\X)$, and then defining the convective derivative ${\sf D}_t:{\rm TestVF}(\ppi_\mu)\to \mathscr L^2(\ppi_\mu) $ as  
\begin{equation}
\label{eq:dttest}
{\sf D}_t\Psi(V_\cdot):=\Psi(\tilde{\sf D}_tV_\cdot)\qquad\forall (V_t)\in{\rm Test}\Pi.
\end{equation}
From \eqref{eq:perprod}  and the  identity $(\ppi_\mu)_t'=\e_t^*(w_t)$ for a.e.\ $t$ (that is a simple consequence of the definitions, see \eqref{eq:c1RLF} and \cite[Theorem 2.3.18]{Gigli14}) this definition of  convective derivative coincides with the one given in \cite{GP20}.

Then $\mathscr W^{1,2}(\ppi_\mu)\subset \mathscr L^2(\ppi_\mu)$ is defined as the collection of $({\sf v}_t)$'s for which there is $({\sf v}'_t)$ such that
\begin{equation}
\label{eq:DtW12}
\int_0^1\int \la {\sf v}_t,{\sf D}_t{\sf z}_t\ra \,\d\ppi_\mu\,\d t=-\int_0^1\int \la {\sf v}'_t,{\sf z}_t\ra \,\d\ppi_\mu\,\d t
\end{equation}
holds for any $({\sf z}_t)\in {\rm TestVF}(\ppi_\mu)$ with `compact support', i.e.\ such that 
\begin{equation}
\label{eq:defcompsupp}
\text{${\sf z}_t=0$ for every $t$ in a neighbourhood of $0$ and $1$. }
\end{equation}
The vector field $({\sf v}'_t)$ is uniquely defined by the above, called \emph{convective derivative} of $({\sf v}_t)$ along $\ppi_\mu$ and denoted  $({\sf D}_t{\sf v}_t)$. Then $\mathscr W^{1,2}(\ppi_\mu)$ is endowed with the norm
\[
\|({\sf v}_t)\|_{\mathscr W^{1,2}}^2:=\|({\sf v}_t)\|_{\mathscr L^2}^2+\|({\sf D}_t{\sf v}_t)\|_{\mathscr L^2}^2
\]
and can - easily - be proved to be a Hilbert space.

It turns out that this latter definition of ${\sf D}_t$ is compatible with the previous one, i.e.\ ${\rm TestVF}(\ppi_\mu)\subset \mathscr W^{1,2}(\ppi_\mu)$ and for a vector field in ${\rm TestVF}(\ppi_\mu)$ the convective derivative defined by formula \eqref{eq:dttest} coincides with the one defined by \eqref{eq:DtW12}. 

In particular, it makes sense to define $\mathscr H^{1,2}(\ppi_\mu)$ as the closure of  ${\rm TestVF}(\ppi_\mu)$ in  $\mathscr W^{1,2}(\ppi_\mu)$. An important property of vector fields in $\mathscr H^{1,2}(\ppi_\mu)$ is that they admit a \emph{continuous representative}. A way to phrase this property is by saying that
\begin{equation}
\label{eq:contreprh12}
\begin{split}
&\text{for every $({\sf v}_t)\in\mathscr H^{1,2}(\ppi_\mu)$ there is  a unique $(\tilde{\sf v}_t)\in {\rm VF}(\ppi_\mu)$}\\
&\text{with ${\sf v}_t=\tilde{\sf v}_t$ for a.e.\ $t$ such that $\Phi(\tilde{\sf v}_\cdot)$ is in $C([0,1],L^0(T\X))$ }
\end{split}
\end{equation}
(from Proposition \ref{prop:dFtot2} it is easy to see that the statement \cite[Theorem 3.23]{GP20} implies this property). In what follows we shall systematically  identify elements of $\mathscr H^{1,2}(\ppi_\mu)$ with their continuous representatives. Another crucial property of elements of $\mathscr H^{1,2}(\ppi_\mu)$ is the absolute continuity of the norms:
\begin{equation}
\label{eq:ach12}
({\sf v}_t)\in \mathscr H^{1,2}(\ppi_\mu)\quad\Rightarrow\quad (t\mapsto|{\sf v}_t|^2)\in AC^2([0,1],L^0(\ppi_\mu))\quad\text{ with }\quad\frac\d{\d t} |{\sf v}_t|^2=2\la{\sf v}_t,{\sf D}_t{\sf v}_t\ra,\quad a.e.\ t.
\end{equation}

\bigskip

As mentioned, our goal in this section is to prove that the space $W^{1,2}_{var}$ is `intermediate' between $\mathscr H^{1,2}(\ppi_\mu)$ and $\mathscr W^{1,2}(\ppi_\mu)$. To make this statement rigorous we should actually consider $\Psi(W^{1,2}_{var})$  and enforce $L^2$-integrability of both the vector field into consideration and its convective derivative. We are therefore led to define the `intermediate' space
\[
{\rm Interm}(\ppi_\mu):=\big\{\Psi(V_\cdot)\ :\ (V_t)\in W^{1,2}_{var}\ \text{ and }\Psi(V_\cdot),\Psi(D_\cdot V_\cdot)\in\mathscr L^2(\ppi_\mu)\big\},
\]
so that the main result of the section can be stated as
\begin{equation}
\label{eq:claimhw}
\mathscr H^{1,2}(\ppi_\mu)\quad\subset\quad {\rm Interm}(\ppi_\mu)\quad\subset\quad \mathscr W^{1,2}(\ppi_\mu)
\end{equation}
with compatible convective derivatives, see Proposition \ref{prop:maincomp} below for the precise statement. It is worth to point out that once these inclusions are proved, from the completeness of $W^{1,2}_{var}$ it is not hard to prove that ${\rm Interm}(\ppi_\mu)$ equipped with the $\mathscr W^{1,2}(\ppi_\mu)$-norm is a Hilbert space.

The proof of \eqref{eq:claimhw}  is based on the following lemma:
\begin{lemma}\label{le:dermezzo}
Let $(V_t)\in W^{1,2}_{var}([0,1],L^0(T\X))$ and $f\in\test\X$. Then $t\mapsto \d f(V_t)\circ F_0^t$ belongs to $W^{1,2}([0,1],L^0(\X))$ and 
\begin{equation}
\label{eq:dermezzo}
\frac{\d}{\d t}\big( \d f(V_t)\circ F_0^t\big)=\d f(D_tV_t)\circ F_0^t+{\rm Hess}f(V_t,w_t)\circ F_0^t,\qquad a.e.\ t,\ \mm-a.e..
\end{equation}
If moreover $(V_t)\in AC^2_{var}([0,1],L^0(T\X))$, then we also have that $t\mapsto \d f(V_t)\circ F_0^t$ belongs to $AC^2([0,1],L^0(\X))$.
\end{lemma}
\begin{proof} It is clear that for $(V_t)\in C([0,1],L^0(T\X))$ the map $t\mapsto \d  f(V_t)\in L^0(\X)$ is continuous, so that by Lemma \ref{le:contcont} also $t\mapsto \d f(V_t)\circ F_0^t$ is continuous. Thus the statement for absolutely continuous vector fields follows from the  Sobolev case and in what follows we focus on the latter.

Let $\mathcal A\subset W^{1,2}_{fix}([0,1],L^0(T\X))$ be the collection of vector fields $(v_t)$ for which the conclusion holds for $V_t:=\d F_0^t(v_t)$. We shall prove that  $\mathcal A$ has the properties $(o),\cdots,(iv)$ in Proposition \ref{prop:Atutto}, so that such proposition gives the conclusion. 

By the linearity in $(V_t)$ of the claim and  of $\d F_0^t$, it is clear that $\mathcal A$ is a vector space, i.e.\ $(o)$ holds. $(i)$ follows from property \eqref{eq:l0lindf} of the differential $\d F_0^t$ and $(ii)$ by direct computation based on the identity $D_t(\varphi(t)V_t)=\varphi'(t)V_t+\varphi(t)D_tV_t$, which in turn is a direct consequence of the definition. 

For $(iii)$, let $(v^n_t)\subset \mathcal A$ be $W^{1,2}_{fix}$-converging to some  $(v^\infty_t)\in W^{1,2}_{fix}([0,1],L^0(T\X))$. Put $V^n_t:=\d F_0^t(v^n_t)$ for every $n\in\N$ and $t\in[0,1]$. We want to prove that $V^\infty_t:=\d F_0^t(v^\infty_t)$ satisfies the conclusions in the statement. Since $(v^n_t)\to (v^\infty_t)$ and $(\dot v^n_t)\to (\dot v^\infty_t)$ in $L^2_{fix}([0,1],L^0(T\X))\hookrightarrow L^0_{fix}([0,1],L^0(T\X))$, Proposition \ref{prop:dFtot2} and the very definition \eqref{eq:defDt} tell that $(V^n_t)\to (V^\infty_t)$ and  $(D_tV^n_t)\to (D_tV^\infty_t)$ in $L^0_{var}([0,1],L^0(T\X))$. It follows that
\[
\begin{array}{ccc}
\d f(V^n_t)\circ F_0^t\quad&\to&\quad \d f(V^\infty_t)\circ F_0^t\\
\d f(D_tV^n_t)\circ F_0^t+{\rm Hess}f(V^n_t,w_t)\circ F_0^t\quad&\to&\quad \d f(D_tV^\infty_t)\circ F_0^t+{\rm Hess}f(V^\infty_t,w_t)\circ F_0^t
\end{array}
\] 
in $L^0([0,1],L^0(\X))$. Since \eqref{eq:dermezzo} holds for $(V^n_t)$ and $|(v^n_t)|_{W^{1,2}_{fix}}\to |(v^\infty_t)|_{W^{1,2}_{fix}}$ in $L^0(\X)$, by Proposition \ref{prop:stabl0w12} the claim will be proved if we show that 
\begin{equation}
\label{eq:perdermezzo}
|(\d f(V^n_t)\circ F_0^t)|_{W^{1,2}}\leq G|(v^n_t)|_{W^{1,2}_{fix}},\qquad\mm-a.e.
\end{equation}
for every $n\in\N$ for some $G\in L^0(\X)$ possibly depending also on $f$. Since $|\d f|\in L^\infty(\X)$ we have
\[
|\d f(V^n_t)\circ F_0^t|_{L^2}\leq C|(V^n_t)|_{L^2_{var}}\stackrel{\eqref{eq:compLp}}\leq G|(v^n_t)|_{L^2_{fix}},\qquad\mm-a.e.
\]
and since also $|w_t|\in L^\infty(\X\times[0,1])$ we also have
\[
\begin{split}
|\d f(D_tV^n_t)\circ F_0^t+{\rm Hess}f(V^n_t,w_t)\circ F_0^t|_{L^2}&\leq C|(D_tV^n_t)|_{L^2_{var}}+C|(V^n_t)|_{L^\infty_{var}}|(|{\rm Hess}f|\circ F_0^t)|_{L^2}\\
\text{(by \eqref{eq:Vtinfty}, \eqref{eq:covtot})}\qquad &\leq G|(v^n_t)|_{W^{1,2}_{fix}}(1+|(|{\rm Hess}f|\circ F_0^t)|_{L^2})\\
\text{(by \eqref{eq:intflo2} for $|{\rm Hess}f|\in L^2(\X)$)}\qquad &\leq G|(v^n_t)|_{W^{1,2}_{fix}},\qquad\mm-a.e..
\end{split}
\]
The claim \eqref{eq:perdermezzo} follows.

To prove $(iv)$ we use Proposition \ref{prop:mainreg}: let $\bar v\in L^0(T\X)$ and put $V_t:=\d F_0^t(\bar v)$ for $t\in[0,1]$. Then it is clear  from Proposition \ref{prop:mainreg} that $t\mapsto \d f(V_t)\circ F_0^t=\d(f\circ F_0^t)(\bar v)$ belongs to $W^{1,2}([0,1],L^0(\X))$ with 
\[
\begin{split}
\frac{\d}{\d t}(\d f(V_t)\circ F_0^t)&=\d(\d f(w_t)\circ F_0^t)(\bar v)=\d(\d f(w_t))(V_t)\circ F_0^t=\big({\rm Hess}f(V_t,w_t)+\d f(D_tV_t)\big)\circ F_0^t,
\end{split}
\]
for a.e.\ $t$, having used the definition of $D_t$ in the last step. This proves $(iv)$ and gives the conclusion.
\end{proof}
From this lemma and the language  recalled in this section we easily obtain the following:
\begin{corollary}\label{cor:linkdd}
Let $({\sf z}_t)\in\mathscr H^{1,2}(\ppi_\mu)$ and $(V_t)\in W^{1,2}_{var}$. Put ${\sf v}_t:=\Psi_t(V_t)$. Then  $t\mapsto\la{\sf z}_t,{\sf v}_t\ra\in L^0(\ppi_\mu)$ belongs to $W^{1,2}([0,1],L^0(\ppi_\mu))$ with 
\[
\frac{\d}{\d t}\la{\sf z}_t,{\sf v}_t\ra=\la {\sf D}_t{\sf z}_t,{\sf v}_t\ra+\la{\sf z}_t,\Psi_t(D_tV_t)\ra,\qquad\ppi_\mu-a.e.,\ a.e.\ t\in[0,1].
\]
If $({\sf z}_t)$ is identified with its continuous representative and $(V_t)\in AC^2_{var}$, then $t\mapsto\la{\sf z}_t,{\sf v}_t\ra\in L^0(\ppi_\mu)$ belongs to $AC^2([0,1],L^0(\ppi_\mu))$. 
\end{corollary}
\begin{proof}\ \\
{\sc Step 0.} We prove the last statement assuming the first claim. Notice that
\[
\begin{split}
\la{\sf z}_t,\Psi_t(V_t)\ra\stackrel{\eqref{eq:psiphi}}=\la\Psi_t(\Phi_t({\sf z}_t)),\Psi_t(V_t)\ra=\la\Phi_t({\sf z}_t),V_t\ra\circ\e_t=\la\Phi_t({\sf z}_t),V_t\ra\circ F_0^t\circ\e_0\qquad\ppi_\mu-a.e..
\end{split}
\] 
Hence if $({\sf z}_t)$ is continuous (recall \eqref{eq:contreprh12}) and so is $(V_t)$, Lemma \ref{le:contcont} ensures that $t\mapsto\la\Phi_t({\sf z}_t),V_t\ra\circ F_0^t\in L^0(\X)$ is continuous, thus the above shows that so is $t\mapsto \la{\sf z}_t,{\sf v}_t\ra\in L^0(\ppi_\mu)$.

\noindent{\sc Step 1.} We assume that  $({\sf z}_t)\in {\rm TestVF}(\ppi_\mu)$, say $({\sf z}_t)=\Psi(Z_\cdot)$ for  $(Z_t):=\sum_{i=0}^n\nchi_{A_i}\times(\varphi_i(t)g_i\nabla f_i)\in{\rm Test}\Pi$, with $A_i$ Borel, $\varphi_i$ Lipschitz and $f_i,g_i \in \test{\X}$.  Notice that by the very definition \eqref{eq:l0varprod} we have
\begin{equation}
\label{eq:dadeftimes}
\la Z_t,V\ra\circ F_0^t=\sum_{i=0}^n\nchi_{A_i}\varphi_i(t)\, g_i \circ F_0^t\,\la \nabla f_i,V\ra\circ F_0^t,\qquad\forall t\in[0,1],\ V\in L^0(T\X).
\end{equation}
Noticing that $(t \mapsto g_i \circ F_0^t) \in W^{1,2}([0,1],L^0(\X))$ with derivative $\frac{\d}{\d t} g_i \circ F_0^t = \d g_i(w_t) \circ F_0^t$ and using Lemma \ref{le:dermezzo}  above, it follows, by the very definition of $W^{1,2}([0,1],L^0(\X))$, that $t\mapsto \la Z_t,V_t\ra\circ F_0^t\in L^0(\X)$ is in $W^{1,2}([0,1],L^0(\X))$ with
\[
\begin{aligned}
\frac{\d}{\d t}\la Z_t,V_t\ra\circ F_0^t&=\sum_{i=0}^n\nchi_{A_i}\Big(\varphi_i(t)\,g_i \,(\la \nabla f_i,D_tV_t\ra+{\rm Hess}f(V_t,w_t))\\
& \qquad \qquad +\varphi_i'(t)\,g_i\,\la \nabla f_i,V_t\ra+\varphi_i(t)\d g_i(w_t) \la \nabla f_i, V_t \ra \,\Big)\circ F_0^t\\
\text{(by \eqref{eq:dttilde})}\qquad&=\la Z_t,D_tV_t\ra\circ F_0^t+\la \tilde{\sf D}_tZ_t,V_t\ra \circ F_0^t
\end{aligned}
\]
and thus (recall also \eqref{eq:charw12}) for $(\mathcal L^1)^2$-a.e.\ $t,s\in[0,1]$ with $t<s$ we have 
\begin{equation}
\label{eq:interm}
\la Z_s,V_s\ra\circ F_0^s-\la Z_t,V_t\ra\circ F_0^t=\int_t^s\la Z_r,D_rV_r\ra\circ F_0^r+\la \tilde{\sf D}_rZ_r,V_r\ra \circ F_0^r\,\d r,\qquad\mm-a.e..
\end{equation}
Now observe that pre-composing the   defining identity $|\Psi_t(V)|=|V|\circ\e_t$ valid $\ppi_\mu$-a.e. by $F_0^\cdot$ and using the second in \eqref{eq:etinv} we get
\begin{equation}
\label{eq:scambio}
|\Psi_t(V)|\circ F_0^\cdot=|V|\circ F_0^t,\qquad\mm-a.e.\ on\ {\sf A}_0,\ \forall V\in L^0(T\X)
\end{equation}
(recall that ${\sf A}_t:=\{\frac{\d(\e_t)_*\sppi_\mu}{\d\mm}>0\}$).  We thus  multiply \eqref{eq:interm}  by $\nchi_{{\sf A}_0}$ and pre-compose it by $\e_0$, then use \eqref{eq:scambio} and the first in \eqref{eq:etinv} for $t=0$ and observe that $\nchi_{{\sf A}_0}\circ\e_0=1$ $\ppi_\mu$-a.e.\  to get that  for $(\mathcal L^1)^2$-a.e.\ $t,s\in[0,1]$ with $t<s$ we have
\begin{equation}
\label{eq:ancheh}
\la {\sf z}_s,{\sf v}_s\ra-\la {\sf z}_t,{\sf v}_t\ra=\int_t^s\la {\sf z}_r,\Psi_r(D_rV_r)\ra+\la {\sf D}_r{\sf z}_r,{\sf v}_r\ra \,\d r,\qquad\ppi_\mu-a.e.,
\end{equation}
having recalled \eqref{eq:dttest}.   By \eqref{eq:charw12bound}, this proves our claim in the case $({\sf z}_t)\in {\rm TestVF}(\ppi_\mu)$.

\noindent{\sc Step 2.} For the general case, we start claiming that for $({\sf z}_t)\in \mathscr H^{1,2}(\ppi_\mu)$ we have
\begin{equation}
\label{eq:h12infty}
|({\sf z}_t)|_{L^\infty}\leq 2\big(|({\sf z}_t)|_{L^2}+|({\sf D}_t{\sf z}_t)|_{L^2}\big),
\end{equation}
where here and below we adopt the notation $|\cdot|_{L^p}$ introduced in Section \ref{se:intmod}, referring this time to the space $L^0(\ppi_\mu)$ in place of $L^0(\mm)$. Indeed, from \eqref{eq:ach12} we see that $|\frac\d{\d t}|{\sf z}_t|^2|\leq |{\sf z}_t|^2+|{\sf D}_t{\sf z}_t|^2$ for a.e.\ $t$ and thus after integration we obtain
\[
|{\sf z}_t|^2\leq |{\sf z}_s|^2+\int_{t\wedge s}^{t\vee s}|{\sf z}_r|^2+|{\sf D}_r{\sf z}_r|^2\,\d r\leq |{\sf z}_s|^2+\int_0^1|{\sf z}_r|^2+|{\sf D}_r{\sf z}_r|^2\,\d r\qquad\ppi_\mu-a.e.,
\]
for $(\mathcal L^1)^2$-a.e.\ $t,s\in[0,1]$. Integrating in $s$ we deduce \eqref{eq:h12infty}. 

We shall use \eqref{eq:h12infty} in conjunction with Proposition \ref{prop:stabl0w12}  to conclude.  Thus let $({\sf z}_t)\in \mathscr H^{1,2}(\ppi_\mu)$ and $({\sf z}^n_t)\subset {\rm TestVF}(\ppi_\mu)$ be $\mathscr W^{1,2}(\ppi_\mu)$-converging to it. From the definition of $\mathscr W^{1,2}(\ppi_\mu)$-norm it is clear that
\[
\begin{split}
\la {\sf z}^n_t,{\sf v}_t\ra\quad&\to \quad \la {\sf z}_t,{\sf v}_t\ra\\
\la {\sf z}^n_t,\Psi_t(D_tV_t)\ra+\la {\sf D}_t{\sf z}^n_t,{\sf v}_t\ra \quad&\to\quad \la {\sf z}_t,\Psi_t(D_tV_t)\ra+\la {\sf D}_t{\sf z}_t,{\sf v}_t\ra 
\end{split}
\]
in $L^0([0,1],L^0(\ppi_\mu))$ and that
\[
\begin{split}
|({\sf z}_t^n)|_{L^2}\quad&\to\quad |({\sf z}_t)|_{L^2},\qquad\qquad\qquad 
|({\sf D}_t{\sf z}_t^n)|_{L^2}\quad\to\quad |({\sf D}_t{\sf z}_t)|_{L^2}
\end{split}
\]
in $L^0(\ppi_\mu)$. Also, from the definition of $\Psi$ we have that $|(\Psi(V_\cdot))|_{L^p}=|(V_t)|_{L^p_{var}}\circ\e_0$, thus
\[
\begin{split}
|\la {\sf z}^n_t,{\sf v}_t\ra|_{L^2}\quad&\leq\quad   |({\sf z}^n_t)|_{L^2}  |({\sf v}_t)|_{L^\infty}= |({\sf z}^n_t)|_{L^2}  |(V_t)|_{L^\infty_{var}}\circ \e_0\\
\text{(by \eqref{eq:Vtinfty})}\qquad&\leq \quad G |({\sf z}^n_t)|_{L^2}  |(V_t)|_{W^{1,2}_{var}} \circ \e_0,
\end{split}
\]
and
\[
\begin{split}
|\la {\sf z}^n_t,\Psi_t(D_tV_t)\ra+\la {\sf D}_t{\sf z}^n_t,{\sf v}_t\ra|_{L^2}\quad&\leq\quad  |({\sf z}^n_t)|_{L^\infty} |(\Psi_t(D_tV_t))|_{L^2}+| ({\sf D}_t{\sf z}^n_t)|_{L^2}|({\sf v}_t)|_{L^\infty}\\
\text{(by \eqref{eq:Vtinfty})}\qquad&\leq \quad G\, |(V_t)|_{W^{1,2}_{var}} \circ \e_0 \,\big(|({\sf z}^n_t)|_{L^2}+|({\sf D}_t{\sf z}^n_t)|_{L^2}\big).
\end{split}
\]
Hence Proposition \ref{prop:stabl0w12} gives the conclusion.
\end{proof}
The main result  of the section can now be proved rather easily:
\begin{proposition}\label{prop:maincomp}
The inclusions in \eqref{eq:claimhw} hold and the underlying notions of convective derivatives agree, i.e.\ for $(V_t)\in W^{1,2}_{var}$ such that ${\sf v}_\cdot:=\Psi(V_\cdot)\in\mathscr L^2(\ppi_\mu)$ we have $\Psi_t(D_tV_t)={\sf D}_t{\sf v}_t$ for a.e.\ $t\in[0,1]$.
\end{proposition}
\begin{proof}\ \\
{\sc Step 1.} We prove the second inclusion in \eqref{eq:claimhw} and the identification of convective derivatives. Let $({\sf v}_t)\in{\rm Interm}(\ppi_\mu)$, say $({\sf v_t})=\Psi(V_\cdot)$ for $(V_t)\in W^{1,2}_{var}$. Also, let $({\sf z}_t)\in {\rm TestVF}(\ppi_\mu)$ be  with compact support in time in the sense of \eqref{eq:defcompsupp}. Then from Corollary \ref{cor:linkdd} above, since ${\rm TestVF}(\ppi_{\mu}) \subseteq \mathscr{H}^{1,2}(\ppi_\mu)$, we  obtain that
\[
\int_0^1\la {\sf z}_r,\Psi_r(D_rV_r)\ra+\la {\sf D}_r{\sf z}_r,{\sf v}_r\ra \,\d r=0\qquad\ppi_\mu-a.e..
\]
The integrability assumptions coming from the hypothesis $\Psi(V_\cdot)\in{\rm Interm}(\ppi_\mu)$ and $({\sf z}_t)\in {\rm TestVF}(\ppi_\mu)$  ensure that the left-hand side of the above is in $L^1(\ppi_\mu)$, thus upon integration we see that \eqref{eq:DtW12} holds with ${\sf v}'_t=\Psi_t(D_tV_t)$ for a.e.\ $t\in[0,1]$. By the arbitrariness of $({\sf z}_t)\in {\rm TestVF}(\ppi_\mu)$,  this is the claim.

\noindent{\sc Step 2.} We prove the first inclusion in \eqref{eq:claimhw}. Let $({\sf z}_t)\in \mathscr H^{1,2}(\ppi_\mu)$ be identified with its continuous representative as in \eqref{eq:contreprh12}. Put $Z_t:=\Phi_t({\sf D}_t{\sf z}_t)$ and $\bar V:=\Phi_0({\sf z}_0)$ and notice that the identity $|\Phi_t({\sf z})|=\nchi_{{\sf A}_t}|{\sf z}|\circ F_0^\cdot\circ F_t^0$ gives $|\Phi_t({\sf z})|\circ F_0^t=\nchi_{{\sf A}_0}|{\sf z}|\circ F_0^\cdot$, which in turn implies, thanks to $({\sf D}_t{\sf z}_t)\in \mathscr L^2(\ppi_\mu)$, that $(Z_t)\in L^2_{var}$. 

By Theorem \ref{thm:exuni}  there is (a unique) $(V_t)\in AC^2_{var}$ satisfying \eqref{eq:odevar} with these choices of $(Z_t),\bar V$. To conclude the proof it is therefore sufficient to show that $\Psi_t(V_t)={\sf z}_t$ for every $t\in[0,1]$. Notice that the identity $|\Psi_t(V)|=|V|\circ\e_t$ gives that $t\mapsto |\Psi_t(V_t)|^2$ is in $AC^2([0,1],L^0(\ppi_\mu))$ with $\frac12\frac\d{\d t} |\Psi_t(V_t)|^2=\la \Psi_t(V_t),\Psi_t(D_tV_t)\ra $. Thus taking into account \eqref{eq:ach12}, Corollary \ref{cor:linkdd}, and Theorem \ref{thm:leibconv}, we see that $t\mapsto |{\sf z}_t-\Psi_t(V_t)|^2=|{\sf z}_t|^2+|\Psi_t(V_t)|^2-2\la {\sf z}_t,\Psi_t(V_t)\ra$ is in $AC^{2}([0,1],L^0(\ppi_\mu))$ with derivative given by
\[
\begin{split}
\tfrac12\frac{\d}{\d t}|{\sf z}_t-\Psi_t(V_t)|^2&=\la {\sf z}_t,{\sf D}_t{\sf z}_t\ra+\la \Psi_t(V_t),\Psi_t(D_tV_t)\ra -\la {\sf z}_t,\Psi_t(D_tV_t) \ra-\la{\sf D}_t{\sf z}_t,\Psi_t(V_t)\ra=0,\quad a.e.\ t
\end{split}
\]
having used the fact that $\Psi_t(D_tV_t)=\Psi_t(\Phi_t({\sf D}_t{\sf z}_t))={\sf D}_t{\sf z}_t$. Since $\Psi_0(V_0)=\Psi_0(\Phi_0({\sf z}_0))={\sf z}_0$, the conclusion follows.
\end{proof}

\appendix
\section{Existence of the parallel transport in \texorpdfstring{$\mathscr{W}^{1,2}(\ppi)$}{W12pi}}
\label{sec:existence_W^{1,2}}
In this section, we show an existence result of the parallel transport of an initial vector field along a Lipschitz test plan in the class $\mathscr{W}^{1,2}(\ppi)$. In particular, we can show, by an abstract argument of functional analysis, that for a given initial vector field $\bar{V} \in \e_0^*L^2(T\X)$ we can find a Borel vector field $t \rightarrow V_t$ along a Lipschitz test plan $\ppi$ such that $V \in \mathscr W^{1,2}(\ppi)$, $(D_{\sppi} V)_t =0$ for a.e.\ $t$, and satisfies the initial condition in an appropriate sense.
We use in this section a vanishing viscosity approach: we approximate our problem with a sequence of problems that are coercive; on this class of problems we can apply a variant of Lax--Milgram lemma; thanks to compactness, we can pass to the limit and obtain a solution to our problem.
More precisely, the variant of Lax--Milgram lemma is the following and it is taken from \cite{Ambrosio-Trevisan14}.
\begin{lemma}[Lions] 
\label{lem:Lions}
Let $E$ and $H$ be a normed and a Hilbert space, respectively. Assume that $E$ is continuously embedded in $H$, with $\|v\|_H \le \|v\|_E$ for every $v \in E$. Let $B \colon H \times E \rightarrow \mathbb{R}$ be a bilinear form such that $B(\cdot,v)$ is continuous for every $v \in E$. If $B$ is coercive, namely there exists $c >0$ such that $B(v,v) \ge c \| v \|_E^2$ for every $v \in E$, then for all $l \in E'$ there exists $h \in H$ such that
\begin{equation*}
B(h,v) = l(v)\quad\text{ for every }v \in E    
\end{equation*}
and
\begin{equation}
\label{eq:lions_estimate}
\| h\|_H \le \frac{\| l \|_{V'}}{c}.
\end{equation}
\end{lemma}
We introduce the following class of approximations. For a given $\varepsilon$, we solve in a distributional sense the partial differential equation:
\begin{equation}
\label{eq:approx_pde}
({\rm D}_{\sppi} V)_t = \varepsilon(-V_t+  ({\rm D}_{\sppi}^2 V)_t),
\end{equation}
looking for a solution in $\mathscr H^{1,2}(\ppi)$.
\begin{definition}[Parallel transport in $\mathscr W^{1,2}(\ppi)$]
\label{def:pt_W12}
Let \((\X,\sfd,\mm)\) be an \({\sf RCD}(K,\infty)\) space and
\(\ppi\) a Lipschitz test plan on \(\X\). Given \(\bar V\in\e_0^*L^2(T\X)\) we say that $V \in \mathscr  W^{1,2}(\ppi)$ is a parallel transport in $\mathscr W^{1,2}(\ppi)$ of $\bar V$ along $\ppi$ if $D_{\sppi} V = 0$ and for every $Z \in {\rm TestVF}(\ppi)$
\begin{equation}\label{eq:initial_cond}
    R(Z)_0 = \int \langle \bar V, Z_0 \rangle\,\d\ppi,
\end{equation}
where we denote by $t \mapsto R(Z)_t$ the absolutely continuous representative of $t \mapsto \int \langle V_t,Z_t \rangle\,\d\ppi$.
\end{definition}
\begin{theorem}[Existence of PT in \(\mathscr W^{1,2}(\ppi)\)]
Let \((\X,\sfd,\mm)\) be an \({\sf RCD}(K,\infty)\) space and
\(\ppi\) a Lipschitz test plan on \(\X\). Let \(\bar V\in\e_0^*L^2(T\X)\)
be given. Then there exists \(V\in\mathscr W^{1,2}(\ppi)\) that is a parallel transport in \(\mathscr W^{1,2}(\ppi)\) of $\bar V$ along $\ppi$.
\end{theorem}
\begin{proof}
Fix \(\eps\in(0,1/2)\). Consider the Hilbert space
\(H\coloneqq\big(\mathscr H^{1,2}(\ppi),\|\cdot\|_{\mathscr H^{1,2}(\sppi)}\big)\).
Define also
\[\begin{split}
E&\coloneqq\big\{Z\in{\rm TestVF}(\ppi)\;\big|\;{\rm spt}(Z)\subseteq[0,1)\big\}\subseteq \mathscr H^{1,2}(\ppi),\\
\|Z\|_E&\coloneqq\big(\|Z_0\|^2_{\e_0^*L^2(T\X)}+\|Z\|^2_H\big)^{1/2}
\quad\text{ for every }Z\in E.
\end{split}\]
Clearly, \(\|Z\|_H\leq\|Z\|_E\) for every \(Z\in E\). Now let us define
\(B\colon H\times E\to\R\) and \(\ell\colon E\to\R\) as
\[\begin{split}
B(V,Z)&\coloneqq\int_0^1\!\!\!\int-\la V_t,{\rm D}_\sppi Z_t\ra
+\eps\,\la V_t,Z_t\ra+\eps\,\la{\rm D}_\sppi V_t,{\rm D}_\sppi Z_t\ra\,\d\ppi\,\d t,\\
\ell(Z)&\coloneqq\int\la\bar V,Z_0\ra\,\d\ppi,
\end{split}\]
respectively. The map \(B\) is bilinear by construction. Moreover, for some
constant \(C>0\) we have that \(\big|B(V,Z)\big|\leq C\,\|V\|_H\,\|Z\|_H\) for
every \(V\in H\) and \(Z\in E\), thus in particular \(B(\cdot,Z)\) is continuous
for any \(Z\in E\). The Leibniz rule grants that
\(-2\int_0^1\!\!\int\la Z_t,{\rm D}_\sppi Z_t\ra\,\d\ppi\,\d t=\int|Z_0|^2\,\d\ppi\)
holds for every \(Z\in E\), whence coercivity of the map \(B\) follows: given any
\(Z\in E\), we have
\[
B(Z,Z)=\frac{1}{2}\int|Z_0|^2\,\d\ppi+
\eps\int_0^1\!\!\!\int|Z_t|^2+|{\rm D}_\sppi Z_t|^2\,\d\ppi\,\d t\geq\eps\,\|Z\|^2_E.
\]
Furthermore, it holds that \(\ell\in E'\) and
\(\|\ell\|_{E'}\leq\|\bar V\|_{\e_0^*L^2(T\X)}\).
Therefore, Lemma \ref{lem:Lions} yields the existence of an element \(V^\eps\in H\)
such that \(\|V^\eps\|_H\leq\|\Bar{} V\|_{\e_0^*L^2(T\X)}/\eps\) and
\(B(V^\eps,Z)=\ell(Z)\) for every \(Z\in E\), which explicitly reads as
\begin{equation}\label{eq:ex_PT_W_aux0}
\int_0^1\!\!\!\int-\la V^\eps_t,{\rm D}_\sppi Z_t\ra+\eps\,\la V^\eps_t,Z_t\ra+
\eps\,\la{\rm D}_\sppi V^\eps_t,{\rm D}_\sppi Z_t\ra\,\d\ppi\,\d t
=\int\la\bar V,Z_0\ra\,\d\ppi
\end{equation}
for every \(Z\in E\). Given any \(Z\in{\rm TestVF}(\ppi)\) and
\(\varphi\in{\rm LIP}([0,1])\) with \({\rm spt}(\varphi)\subseteq[0,1)\),
it holds that \(t\mapsto\varphi(t)Z_t\) belongs to \(E\) and
\({\rm D}_\sppi(\varphi Z)_t=\varphi'(t)Z_t+\varphi(t)\,{\rm D}_\sppi Z_t\)
for a.e.\ \(t\in[0,1]\). Then
\begin{equation}\label{eq:ex_PT_W_aux1}\begin{split}
\varphi(0)\int\la\bar V,Z_0\ra\,\d\ppi=&
\int_0^1\varphi(t)\int-\la V^\eps_t,{\rm D}_\sppi Z_t\ra+\eps\,\la V^\eps_t,Z_t\ra+
\eps\,\la{\rm D}_\sppi V^\eps_t,{\rm D}_\sppi Z_t\ra\,\d\ppi\,\d t\\
&+\int_0^1\varphi'(t)\int-\la V^\eps_t,Z_t\ra+
\eps\,\la{\rm D}_\sppi V^\eps_t,Z_t\ra\,\d\ppi\,\d t.
\end{split}\end{equation}
Fix a Lebesgue point \(s\in(0,1)\) of \(t\mapsto\int-\la V^\eps_t,Z_t\ra
+\eps\,\la{\rm D}_\sppi V^\eps_t,Z_t\ra\,\d\ppi\). Define \(\varphi_n\) as
\[
\varphi_n(t)\coloneqq\left\{\begin{array}{ll}
1\\
-n(t-s)+1\\
0
\end{array}\quad\begin{array}{ll}
\text{ if }t\in[0,s),\\
\text{ if }t\in[s,s+1/n),\\
\text{ if }t\in[s+1/n,1],
\end{array}\right.
\]
for all \(n\in\N\), \(n>1/(1-s)\). Note that
\((\varphi_n)_n\subseteq{\rm LIP}([0,1])\) is a bounded sequence in \(L^\infty(0,1)\),
\({\rm spt}(\varphi_n)\subseteq[0,1)\) for all \(n\), and
\(\varphi_n\to\nchi_{[0,s]}\) pointwise as \(n\to\infty\). Moreover, it holds that
\[\begin{split}
\int_0^1\varphi_n'(t)\int-\la V^\eps_t,Z_t\ra+
\eps\,\la{\rm D}_\sppi V^\eps_t,Z_t\ra\,\d\ppi\,\d t
=&n\int_s^{s+1/n}\int\la V^\eps_t,Z_t\ra-
\eps\,\la{\rm D}_\sppi V^\eps_t,Z_t\ra\,\d\ppi\,\d t\\
&\to\int\la V^\eps_s,Z_s\ra-\eps\,\la{\rm D}_\sppi V^\eps_s,Z_s\ra\,\d\ppi
\quad\text{ as }n\to\infty.
\end{split}\]
Therefore, by plugging \(\varphi=\varphi_n\) into \eqref{eq:ex_PT_W_aux1} and letting
\(n\to\infty\), we deduce that
\begin{equation}\label{eq:ex_PT_W_aux2}\begin{split}
\int\la\bar V,Z_0\ra\,\d\ppi=&
\int_0^s\!\!\!\int-\la V^\eps_t,{\rm D}_\sppi Z_t\ra+\eps\,\la V^\eps_t,Z_t\ra+
\eps\,\la{\rm D}_\sppi V^\eps_t,{\rm D}_\sppi Z_t\ra\,\d\ppi\,\d t\\
&+\int\la V^\eps_s,Z_s\ra-\eps\,\la{\rm D}_\sppi V^\eps_s,Z_s\ra\,\d\ppi.
\end{split}\end{equation}
Given that \(V^\eps\in\mathscr H^{1,2}(\ppi)\), we can find a sequence
\((Z^n)_n\subseteq{\rm TestVF}(\ppi)\) that \(\mathscr W^{1,2}(\ppi)\)-converges
to \(V^\eps\). 
We start noticing that from \cite[Theorem 3.23]{GP20} there exists a continuous injection \( i: \mathscr H^{1,2}(\ppi) \rightarrow \mathscr C(\ppi)\) such that \( \| V^\eps - Z^n \|_{\mathscr C(\sppi)} \le \sqrt{2} \| V^\eps -Z^n \|_{\mathscr W^{1,2}(\sppi)} \), which grants that \( \lim_{n \rightarrow \infty} \| V^\eps_0-Z^n_0 \|_{\e_0^*L^2(T\X)} = 0\). Therefore, it follows that 
\[ \lim_{n \rightarrow \infty} \int \la \bar V, Z^n_0 \ra\, \d \ppi =\int \la \bar V, V^\eps_0 \ra\, \d \ppi.
\]
By plugging \(Z=Z^n\) into \eqref{eq:ex_PT_W_aux2}, letting
\(n\to\infty\), and using the Leibniz rule in $\mathscr H^{1,2}(\ppi)$, we get
\[\begin{split}
\int\la\bar V,V^\eps_0\ra\,\d\ppi&=
\int_0^s\!\!\!\int-\la V^\eps_t,{\rm D}_\sppi V^\eps_t\ra+\eps\,|V^\eps_t|^2+
\eps\,|{\rm D}_\sppi V^\eps_t|^2\,\d\ppi\,\d t
+\int|V^\eps_s|^2-\eps\,\la{\rm D}_\sppi V^\eps_s,V^\eps_s\ra\,\d\ppi\\
&\geq\int_0^s\!\!\!\int-\la V^\eps_t,{\rm D}_\sppi V^\eps_t\ra\,\d\ppi\,\d t
+\int|V^\eps_s|^2-\eps\,\la{\rm D}_\sppi V^\eps_s,V^\eps_s\ra\,\d\ppi\\
&=-\frac{1}{2}\int|V^\eps_s|^2\,\d\ppi+\frac{1}{2}\int|V^\eps_0|^2\,\d\ppi
+\int|V^\eps_s|^2-\eps\,\la{\rm D}_\sppi V^\eps_s,V^\eps_s\ra\,\d\ppi\\
&=\frac{1}{2}\int|V^\eps_s|^2\,\d\ppi+\frac{1}{2}\int|V^\eps_0|^2\,\d\ppi
-\eps\,\la{\rm D}_\sppi V^\eps_s,V^\eps_s\ra\,\d\ppi.
\end{split}\]
Since \(\frac{1}{2}\int|V^\eps_0|^2\,\d\ppi-\int\la\bar V,V^\eps_0\ra\,\d\ppi
=\frac{1}{2}\int|\bar V-V^\eps_0|^2\,\d\ppi-\frac{1}{2}\int|\bar V|^2\,\d\ppi\),
we can rewrite the former expression as
\[ \frac{1}{2}\int|\bar V-V^\eps_0|^2\,\d\ppi-\frac{1}{2}\int|\bar V|^2\,\d\ppi
\le \eps \int \langle {\rm D}_{\sppi} V_s^\eps,V_s^\eps \rangle \,\d \ppi - \frac{1}{2} \int |V_s^\eps|^2 \,\d \ppi. \]
Therefore, we obtain that
\[
\frac{1}{2}\int|V^\eps_s|^2\,\d\ppi\leq\frac{1}{2}\int|V^\eps_s|^2\,\d\ppi
+\frac{1}{2}\int|\bar V-V^\eps_0|^2\,\d\ppi\leq\frac{1}{2}\int|\bar V|^2\,\d\ppi
+\eps\int\la{\rm D}_\sppi V^\eps_s,V^\eps_s\ra\,\d\ppi.
\]
By integrating the above inequality over the interval \([0,1]\), multiplying by \(2\), and
then applying Young's inequality \(ab\leq\eps a^2+\frac{b^2}{4\eps}\), we infer that
\[\begin{split}
\int_0^1\!\!\!\int|V^\eps_s|^2\,\d\ppi\,\d s&\leq\int|\bar V|^2\,\d\ppi+
2\eps\int_0^1\!\!\!\int\la{\rm D}_\sppi V^\eps_s,V^\eps_s\ra\,\d\ppi\,\d s\\
&\leq\int|\bar V|^2\,\d\ppi+
2\eps^2\int_0^1\!\!\!\int|{\rm D}_\sppi V^\eps_s|^2\,\d\ppi\,\d s
+\frac{1}{2}\int_0^1\!\!\!\int|V^\eps_s|^2\,\d\ppi\,\d s,
\end{split}\]
whence accordingly
\(\frac{1}{2}\int_0^1\!\!\int|V^\eps_s|^2\,\d\ppi\,\d s\leq\int|\bar V|^2\,\d\ppi
+2\eps^2\int_0^1\!\!\int|{\rm D}_\sppi V^\eps_s|^2\,\d\ppi\,\d s
\leq 3\,\|\bar V\|^2_{\e_0^*L^2(T\X)}\). Observe also that
\(\{\eps V^\eps\}_{\eps\in(0,1/2)}\) is bounded in \(H\).
Therefore, there exist \(V\in\mathscr L^2(\ppi)\), \(W\in H\),
and a sequence \(\eps_n\searrow 0\), such that \(V^{\eps_n}\weakto V\)
weakly in \(\mathscr L^2(\ppi)\) and \(\eps_n V^{\eps_n}\weakto W\)
weakly in \(H\). In particular, it must hold that \(W=0\). Hence,
by letting \(n\to\infty\) in the identity
\[
\int_0^1\!\!\!\int-\la V^{\eps_n}_t,{\rm D}_\sppi Z_t\ra+\la\eps_n V^{\eps_n}_t,Z_t\ra+
\la{\rm D}_\sppi(\eps_n V^{\eps_n})_t,{\rm D}_\sppi Z_t\ra\,\d\ppi\,\d t
=\int\la\bar V,Z_0\ra\,\d\ppi,
\]
which holds for every \(n\in\N\) and \(Z\in{\rm TestVF}_c(\ppi)\) by
\eqref{eq:ex_PT_W_aux0}, we can finally conclude that
\(\int_0^1\!\!\int\la V_t,{\rm D}_\sppi Z_t\ra\,\d\ppi\,\d t=0\)
is satisfied for every \(Z\in{\rm TestVF}_c(\ppi)\). This grants
that \(V\in\mathscr W^{1,2}(\ppi)\) and \({\rm D}_\sppi V=0\).
Finally, let us prove \eqref{eq:initial_cond}. Fix any \(Z\in{\rm TestVF}(\ppi)\) and denote by
\(t \mapsto R(Z)_t\) the absolutely continuous representative of \(t \mapsto \int \la V_t, Z_t \ra \,\d \ppi\) (that belongs to \(W^{1,1}(0,1)\)).
Write \eqref{eq:ex_PT_W_aux2} with $\varepsilon=\varepsilon_n$, integrate over $s\in[0,1]$, and then let $n\to\infty$:
by exploiting the weak convergence $V^{\varepsilon_n}\rightharpoonup V$ in $\mathscr L^2(\ppi)$ (and by using the dominated convergence theorem), we obtain that
\[\begin{split}
\int\la\bar V,Z_0\ra\,\d\ppi&=-\int_0^1\!\!\!\int_0^s\!\!\!\int\la V_t,{\rm D}_\sppi Z_t\ra\,\d\ppi\,\d t\,\d s+\int_0^1\!\!\!\int\la V_s,Z_s\ra\,\d\ppi\,\d s\\
&=-\int_0^1\!\!\!\int_0^s\bigg(\frac{\d}{\d t}R(Z)_t\bigg)\d t\,\d s+\int_0^1 R(Z)_s\,\d s\\
&=-\int_0^1 R(Z)_s-R(Z)_0\,\d s+\int_0^1 R(Z)_s\,\d s=R(Z)_0,
\end{split}\]
where we applied the Leibniz rule with one vector field in $\mathscr H^{1,2}(\ppi)$ and the other in $\mathscr W^{1,2}(\ppi)$.
Hence, the statement is achieved.
\end{proof}
\def\cprime{$'$} \def\cprime{$'$}

\end{document}